\documentclass[10pt]{amsart}
\usepackage{amssymb,amsmath,bbm}
\usepackage{stmaryrd}
\usepackage{units}
\usepackage{cancel}
\usepackage[all]{xy}
\newtheorem{thm}{Theorem}[section]
\theoremstyle{definition}
\newtheorem{definition}[thm]{Definition}
\theoremstyle{theorem}
\newtheorem{prop}[thm]{Proposition}
\newtheorem{lemma}[thm]{Lemma}
\theoremstyle{remark}
\newtheorem{example}[thm]{Example}
\allowdisplaybreaks
\numberwithin{equation}{section}
\newcommand{\im}{\mathbf{i}}
\newcommand{\e}{\mathbf{e}}
\newcommand{\n}{\vspace{\baselineskip}}
\newcommand{\dd}{\text{d}}
\newcommand{\ad}{\text{ad}}
\newcommand{\Ad}{\text{Ad}}
\newcommand{\tv}{{\tt v}}
\numberwithin{equation}{subsection}
\usepackage{fancyhdr}
\usepackage{lipsum}
\newcommand\shortitle{Equivalence Problem for 7-Dimensional, 2-Nondegenerate CR Manifolds}
\begin{document}
\title{The Local \MakeUppercase{\shortitle} whose Cubic Form is of Conformal Unitary Type}

\author{Curtis Porter}
\address{Department of Mathematics, Texas A\&M University, College Station, TX 77843-3368}
\email{cporter@math.tamu.edu}

\begin{abstract}
We apply E. Cartan's method of equivalence to classify 7-dimensional, 2-nondegenerate CR manifolds $M$ up to local CR equivalence in the case that the cubic form of $M$ satisfies a certain symmetry property with respect to the Levi form of $M$. The solution to the equivalence problem is given by a parallelism on a principal bundle over $M$. When the nondegenerate part of the Levi form has definite signature, the parallelism takes values in $\mathfrak{su}(2,2)$. When this signature is split and an additional ``isotropy-switching" hypothesis is satisfied, the parallelism takes values in $\mathfrak{su}(3,1)$. Differentiating the parallelism provides a complete set of local invariants of $M$. We exhibit an explicit example of a real hypersurface in $\mathbb{C}^4$ whose invariants are nontrivial.
\end{abstract}

\maketitle
\tableofcontents

\section{Introduction}

A CR manifold $M$ of CR-dimension $m$ and CR-codimension $c$ is intrinsically defined to abstract the structure of a smooth, real, codimension-$c$ submanifold of a complex manifold of complex dimension $m+c$. The most trivial example of such a submanifold is $\mathbb{C}^m\times\mathbb{R}^c\subset\mathbb{C}^{m+c}$, and the obstruction to the existence of a local CR equivalence $M\to\mathbb{C}^m\times\mathbb{R}^c$ is the Levi form $\mathcal{L}$ of $M$, a $\mathbb{C}^c$-valued Hermitian form on the CR bundle of $M$ whose signature in the $c=1$ case is a basic invariant of $M$'s CR structure. As such, attempts to classify CR manifolds of hypersurface-type ($c=1$) fundamentally depend on the degree of degeneracy of $\mathcal{L}$. 

E. Cartan first applied his method of equivalence to real, pseudoconvex ($\mathcal{L}\neq0$) hypersurfaces in $\mathbb{C}^2$ (\cite{CartanCR}), and his work was generalized by N. Tanaka (\cite{TanakaCR}) and Chern-Moser (\cite{ChernMoser}) to solve the equivalence problem for hypersurface-type CR submanifolds whose Levi form has signature $(p,q)$ with $p+q=m$. That they ``solved the equivalence problem'' is to say they constructed an $\mathfrak{su}(p+1,q+1)$-valued parallelism $\omega$ on a principal bundle over $M$, and differentiating $\omega$ provides a complete set of local invariants of $M$.

M. Freeman later proved (\cite{FreemanStraightening}) that the Tanaka-Chern-Moser (TCM) classification could be extended to those $M$ which are locally CR equivalent to $N\times\mathbb{C}^k$, where $N$ satisfies the hypotheses of the TCM case. For such ``CR-straightenable'' $M$, $\mathcal{L}$ has a $k$-dimensional kernel, but Freeman showed that this information is not enough to determine if a general $M$ with $\dim_\mathbb{C}(\ker\mathcal{L})=k$ can be locally straightened, as higher-order generalizations of the Levi form detect obstructions to a diffeomorphism $M\to N\times\mathbb{C}^k$ being a CR equivalence. In particular, the cubic form $\mathcal{C}$ of $M$ must vanish identically for the TCM classification to apply. 

When $\mathcal{C}$ has a trivial kernel, $M$ is called \emph{2-nondegenerate}. The phenomenon of 2-nondegeneracy first appears for $\dim_\mathbb{R}M=5$, so the method of equivalence was initially employed to treat a restricted CR-equivalence class of such $M$ by P. Ebenfelt in \cite{Eb5dim}, \cite{EbCorrection}, and then the general 5-dimensional case was addressed by Isaev-Zaitsev (\cite{IsaevZaitsev}) and Medori-Spiro (\cite{MedoriSpiro}). The 5-dimensional equivalence problem is solved by the construction of an $\mathfrak{so}(3,2)$-valued parallelism on a principal bundle over $M$.

In the present paper, we consider the equivalence problem for 7-dimensional, 2-nondegenerate $M$ for which the cubic form $\mathcal{C}$ satisfies certain algebraic conditions that are automatic in the 5-dimensional case. Specifically, we show that $\mathcal{C}$ is determined by a family of antilinear operators $\ad_K$ on the CR bundle of $M$. The operators $\ad_K$ are symmetric with respect to $\mathcal{L}$, and we impose the hypothesis that they are unitary, up to (nonzero) scale. As such, we say in this case that $\mathcal{C}$ is of conformal unitary type. The nondegenerate part of $\mathcal{L}$ either has definite signature $(2,0)$ or split signature $(1,1)$. In the latter case, the operators $\ad_K$ act nontrivially on two complex, $\mathcal{L}$-isotropic lines in each fiber of the CR bundle. 

At this point, the split-signature equivalence problem branches into two distinct subcases depending on whether $\ad_K$ preserves the real span of any isotropic CR vector. We solve the equivalence problem for the case that $\mathcal{L}$ has definite signature along with the split-signature subcase that $\ad_K$ has no real eigenvalues, saving the third scenario for a future article. (A. Santi recently constructed homogeneous models for all three scenarios in \cite{Santi_models}.) Our solution to the local equivalence problem is furnished by a parallelism $\omega$ on a principal bundle over $M$, where $\omega$ takes values in $\mathfrak{su}(2,2)$ in the definite case or $\mathfrak{su}(3,1)$ in the split-signature subcase.

Differentiating $\omega$ provides a complete set of local invariants of $M$. When all of these invariants vanish, $M$ is locally CR equivalent to an $SU(2,2)$ or $SU(3,1)$ orbit $M_\star$ inside the Grassmannian manifold $Gr(2,4)$ of complex two-planes in $\mathbb{C}^4$. The study of orbits of real forms in complex flag manifolds was initiated by J. Wolf in \cite{WolfRealOrbits}, and his examination of the structure of these orbits included their foliation by maximal complex submanifolds (compare to Freeman's \cite{FreemanComplexFoliations}). Altomani, Medori, and Nacinovich study the CR structure of these orbits in \cite{CRorbits}. When the invariants of $M$ are nonvanishing, no CR equivalence $M\to M_\star\subset Gr(2,4)$ exists. An example of $M$ with nontrivial invariants is given by the hypersurface $z^4+z^{\overline{4}}+(z^3+z^{\overline{3}})\ln\left(\tfrac{(z^1+z^{\overline{1}})(z^2+z^{\overline{2}})}{(z^3+z^{\overline{3}})^2}\right)=0$ in $\mathbb{C}^4$ -- see \S\ref{nonflat} for a detailed analysis. 
 
We proceed to a description of the contents of the paper. In \S\ref{background}, the necessary background on CR geometry and 2-nondegeneracy is reviewed; much of this material is covered in detail in E. Chirka's \cite{ChirkaCR}. The equivalence problem is solved in \S\ref{eqprob}. A standard reference for the algorithmic procedure of Cartan's method of equivalence is \cite{GardnerEquiv}. The author also greatly benefited from the exposition of \cite{BryantGriffithsGrossman}, wherein the general theory is illuminated by the extended examples of Monge-Amp\`{e}re equations and conformal geometry. Because of the technical nature of the calculation, we offer a brief overview of the steps involved.

In \S\ref{initialG}, the filtration on $\mathbb{C}TM$ determined by the CR bundle and Levi kernel is encoded in a principal bundle $B_0$ of complex coframes on $M$ adapted to this filtration -- an ``order zero'' adaptation. The structure group $G_0$ of $B_0$ is 21-dimensional, and the globally defined tautological forms on $B_0$ are extended to a full coframing of $B_0$ over any local trivialization $B_0\cong G_0\times M$ by the Maurer-Cartan forms of $G_0$. These Lie-algebra-valued ``pseudoconnection'' forms are only locally determined up to combinations of the tautological forms which take values in the same Lie algebra. 

We gradually eliminate this ambiguity in the pseudoconnection forms when we restrict to subbundles of $B_0$ defined by coframes that are adapted to higher order, as this reduces the dimension of the structure group and its Lie algebra. Therefore, in \S\ref{first2}, we perform the first such reductions. Restricting to the subbundle $B_1\subset B_0$ of coframes which are ``orthonormal'' for the nondegenerate part of $\mathcal{L}$ reduces the structure group to a 17-dimensional subgroup $G_1\subset G_0$. Similarly, our hypothesis on the cubic form implies there is a subbundle $B_2\subset B_1$ of coframes which are analogously adapted to $\mathcal{C}$, and the structure group $G_2\subset G_1$ has dimension 13.

In \S\ref{absorption}, we exploit the ambiguity in the pseudoconnection forms on $B_2$ in order to simplify the expressions of the exterior derivatives of the tautological forms. This process is known as absorbing torsion, and simplifying the equations facilitates the final two reductions in \S\ref{last2}. The subbundles $B_4\subset B_3\subset B_2$ constructed therein have structure groups $G_4\subset G_3\subset G_2$ reduced from dimension 13 to $\dim G_3=9$, and ultimately to $\dim G_4=7$. At this point, no further reduction is possible without destroying the tautological forms, but the pseudoconnection forms on $B_4$ are still not uniquely defined. 

To finish the calculation, in \S\ref{prolongation} we prolong to the bundle $B_4^{(1)}$ over $B_4$ that parameterizes the remaining ambiguity of the pseudoconnection forms on $B_4$ in the same way that $B_4$ parameterizes the ambiguity in our adapted coframes of $M$. In this sense we begin the method of equivalence anew, but the structure group of $B_4^{(1)}$ as a bundle over $B_4$ is only 1-dimensional. After finding expressions for the derivatives of the tautological forms on $B_4^{(1)}$, the ambiguity in the pseudoconnection form on $B_4^{(1)}$ is completely eliminated by absorbing torsion in these expressions. 

The coframing of $B_4^{(1)}$ so constructed defines a parallelism $\omega$. In \S\ref{parallelism} we study the properties of $\omega$. The invariants obtained by differentiating $\omega$ are shown to measure the obstruction to the existence of a local CR equivalence from $M$ to a homogeneous quotient of $SU(2,2)$ or $SU(3,1)$ by a subgroup isomorphic to the structure group of $B_4^{(1)}$ as a bundle over $M$. This homogeneous space $M_\star$ is called the homogeneous model of our particular CR geometry in the spirit of F. Klein's Erlangen program (\cite[\S1.4]{CapSlovak}). In fact, we show that the lowest order invariants suffice to detect local CR equivalence to $M_\star$. 

Next we ask if $\omega$ satisfies an equivariance condition to define a Cartan connection. While this turns out to be true for the bundle $B_4^{(1)}\to B_4$, it fails for $B_4^{(1)}\to M$, as evidenced by the presence of two-forms in the curvature tensor of $\omega$ which are not semibasic for the latter bundle projection. Finally, we exhibit a real hypersurface $M\subset\mathbb{C}^4$ that is not locally isomorphic to $M_\star$, demonstrating the existence of so-called ``non-flat'' CR manifolds which satisfy our hypotheses.

\vspace{\baselineskip}\section{Background and Notation}\label{background}

%%%%%%%%%%%%%%%%%%%%%%%%%%%%%%%%%%%%%%%%%%%%%%%%%%%%%%%%%%%%%%%%%%%%%%%%%%%%%%%%%%
\subsection{CR Manifolds and 2-Nondegeneracy}
%%%%%%%%%%%%%%%%%%%%%%%%%%%%%%%%%%%%%%%%%%%%%%%%%%%%%%%%%%%%%%%%%%%%%%%%%%%%%%%%%%

Let $M$ be a smooth manifold of real dimension $2(n+k)+c$ for $n,k,c\in\mathbb{N}$. For any vector bundle $p:E\to M$, $E_x:=p^{-1}(x)$ denotes the fiber of $E$ over $x\in M$, $\Gamma(E)$ denotes the sheaf of smooth (local) sections of $E$, and $\mathbb{C}E$ denotes the complexified vector bundle whose fiber over $x$ is $\mathbb{C}E_x:=E_x\otimes_\mathbb{R}\mathbb{C}$. Throughout the paper we adhere to the summation convention, and we let $\im:=\sqrt{-1}$. The letters $i,j$, etc. may therefore be used as indices without any danger of confusion, and we do so without compunction.

A CR structure of CR dimension $(n+k)$ and codimension $c$ is determined by a rank-$2(n+k)$ subbundle $D$ of the tangent bundle $TM$, and an \emph{almost complex} structure $J$ on $D$; i.e., a smooth bundle endomorphism $J:D\to D$ which satisfies $J^2=-\mathbbm{1}_D$, where $\mathbbm{1}_D$ denotes the identity map of $D$. The induced action of $J$ on $\mathbb{C}D$ splits each fiber $\mathbb{C}D_x=H_x\oplus\overline{H}_x$, where $H\subset\mathbb{C}D$ denotes the smooth, $\mathbb{C}$-rank-$(n+k)$ subbundle of $\im$-eigenspaces of $J$, while $\overline{H}$ is that of $-\im$-eigenspaces. We refer to $H$ as the CR bundle of $M$.

If $M_1,M_2$ are two CR manifolds with respective CR structures $(D_1,J_1),(D_2,J_2)$ determining CR bundles $H_1,H_2$, then a \emph{CR map} is a smooth map $F:M_1\to M_2$ whose pushforward $F_*:TM_1\to TM_2$ satisfies $F_*(D_1)\subset D_2$ and $F_*\circ J_1=J_2\circ F_*$. Equivalently, a smooth map $F$ is a CR map if the induced action of $F_*$ on $\mathbb{C}TM_1$ satisfies $F_*(H_1)\subset H_2$. A local \emph{CR equivalence} is a local diffeomorphism which is a CR map.

Local sections $\Gamma(H)$ of the CR bundle are called \emph{CR vector fields}. A CR structure is \emph{integrable} if the Lie bracket of any two CR vector fields is again a CR vector field, often abbreviated $[H,H]\subset H$ (or by conjugating, $[\overline{H},\overline{H}]\subset\overline{H}$). We restrict our attention to integrable CR structures. Note that CR integrability does not imply that $D$ is an integrable subbundle of $TM$, which would additionally require $[H,\overline{H}]\subset H\oplus\overline{H}$. The latter occurs only in the most trivial examples of CR manifolds, and the obstruction to this triviality is the familiar \emph{Levi form}, the sesquilinear bundle map 
\begin{equation*}
 \mathcal{L}:H\times H\to \mathbb{C}TM/\mathbb{C}D,
\end{equation*}
defined as follows. For $X_x,Y_x\in H_x$ and $X,Y\in\Gamma(H)$ such that $X|_x=X_x$ and $Y|_x=Y_x$,
\begin{equation*}
 \mathcal{L}(X_x,Y_x):=\im[X,\overline{Y}]|_x\mod\mathbb{C}D.
\end{equation*}

The \emph{Levi kernel} $K_x\subset H_x$ is therefore given by $K_x:=\{X_x\in H_x\ |\ \mathcal{L}(X_x,Y_x)=0,\ \forall Y_x\in H_x\}$. When $K_x=0$ for every $x$, the CR structure is said to be \emph{Levi-nondegenerate} or \emph{1-nondegenerate}. We consider only the case where $K\subset H$ is a smooth subbundle of constant rank $\dim_\mathbb{C}K_x=k$, and by taking complex conjugates we could similarly define $\overline{K}\subset\overline{H}$. An application of the Newlander-Nirenberg theorem shows that $K\oplus\overline{K}\subset\mathbb{C}D$ is the complexification of a $J$-invariant, integrable subbundle $D^\circ\subset D$, so that $M$ is foliated by complex manifolds of complex dimension $k$. Thus, a local coordinate chart adapted to this \emph{Levi foliation} provides a local diffeomorphism $F:M\to N\times\mathbb{C}^k$, where $N$ is a CR manifold of CR dimension $n$ and CR codimension $c$. However, the CR structure of $N$ is not necessarily integrable, so $F$ is not a CR map in general (\cite{
FreemanStraightening}), and the obstruction to the 
existence of such a ``CR straightening'' is a generalization of the Levi 
form which 
is sometimes called the \emph{cubic form} (\cite{WebsterCubic}) or \emph{third order tensor} (\cite{EbCubic}): 
\begin{equation*}
 \mathcal{C}:K\times H\times H\to \mathbb{C}TM/\mathbb{C}D.
\end{equation*}

For $X_x\in K_x$ and $Y_{x},Z_{x}\in H_x$ with CR vector fields $X\in\Gamma(K)$ and $Y,Z\in\Gamma(H)$ which locally extend them, we define 
\begin{equation*}
 \mathcal{C}(X_x,Y_{x},Z_{x}):=\im[[X,\overline{Y}],\overline{Z}]|_x\mod\mathbb{C}D.
\end{equation*}

The kernel of the cubic form may be defined as a subbundle of $K$ in the same manner as the Levi kernel, and it is exactly when this kernel is all of $K$ that the CR structure transverse to the Levi foliation is integrable, hence the foliate coordinate map $F$ above is a CR straightening. At the other extreme is the case where the kernel of the cubic form is trivial, and in this situation we say that the CR structure is \emph{2-nondegenerate}.

%%%%%%%%%%%%%%%%%%%%%%%%%%%%%%%%%%%%%%%%%%%%%%%%%%%%%%%%%%%%%%%%%%%%%%%%%%%%%%%%%%
\vspace{\baselineskip}\subsection{Examples in Dimension 5}
%%%%%%%%%%%%%%%%%%%%%%%%%%%%%%%%%%%%%%%%%%%%%%%%%%%%%%%%%%%%%%%%%%%%%%%%%%%%%%%%%%

Consider $\mathbb{C}^3$ with complex coordinates $z^1,z^2,z^3$ where $z^i=x^i+\im y^i$ for $i=1,2,3$. In the following examples, we have a real hypersurface $M\subset\mathbb{C}^3$ with $n+k=2$ and $c=1$.

\begin{example}
When $M$ is the hypersurface $y^3=0$, the Levi form of $M$ is completely degenerate, so $n=0$ while $k=2$. Every 5-dimensional CR manifold $M'$ with $c=1,n=0,k=2$ is locally CR-equivalent to this trivial case.
\end{example}
\begin{example}
The ``CR Sphere'' (\cite[\S1.1.6]{CapSlovak}) $M$ is the hypersurface $|z^1|^2+|z^2|^2=|z^3|^2$. The Levi form of $M$ is completely non-degenerate, so $n=2$ and $k=0$. $M$ may be exhibited as a homogeneous quotient of $SU(3,1)$ by a parabolic subgroup $P$. By the results of Tanaka and Chern-Moser, every 5-dimensional CR manifold $M'$ with $c=1$ which is Levi-nondegenerate admits a principal $P$-bundle and a Cartan connection on this bundle whose curvature measures the obstruction to $M'$ being locally CR-equivalent to the CR sphere.
\end{example}
\begin{example}
The ``tube over the future light cone'' (\cite{FreemanLeviDeg},\cite{IsaevZaitsev}) is the hypersurface $M$ given by $(x^1)^2+(x^2)^2=(x^3)^2$ where $x^3>0$, which is a homogeneous quotient of $SO^\circ(3,2)$ (the connected component of the identity) by a non-parabolic subgroup. Here we have $n=k=1$, the lowest dimension in which 2-nondegeneracy is possible.
\end{example}

%%%%%%%%%%%%%%%%%%%%%%%%%%%%%%%%%%%%%%%%%%%%%%%%%%%%%%%%%%%%%%%%%%%%%%%%%%%%%%%%%%
\vspace{\baselineskip}\subsection{The Cubic Form}
%%%%%%%%%%%%%%%%%%%%%%%%%%%%%%%%%%%%%%%%%%%%%%%%%%%%%%%%%%%%%%%%%%%%%%%%%%%%%%%%%%

In order to specialize to the case of ``hypersurface-type'' CR manifolds, from now on we fix $c=1$. We neglect the trivial case when $\mathcal{L}$ is completely degenerate, so that $n>0$ and $D$ is a bracket-generating hyperplane distribution. In the hypersurface-type case, $\mathcal{L}$ and $\mathcal{C}$ take values in a complex line bundle, so a local trivialization $\mathbb{C}TM/\mathbb{C}D\to\mathbb{C}$ which maps $TM/D\to\mathbb{R}\subset\mathbb{C}$ presents $\mathcal{L}$ as a sesquilinear form on $H$. Such a trivialization is locally provided by a nonvanishing one-form $\theta^0\in\Omega^1(M)\subset\Omega^1(M,\mathbb{C})$ that annihiliates $D$ (which we denote $\theta^0\in\Gamma(D^\bot)$). In the notation above, the resulting Hermitian form is given by 
\begin{equation*}
 \mathcal{L}_0(X_x,Y_x):=\im\theta^0|_x([X,\overline{Y}]).
\end{equation*}

We similarly define $\mathcal{C}_0$. Note that $\mathcal{L}$ is actually a conformal class of such forms, as $t\theta^0$ for any real, nonvanishing $t\in C^\infty(M)$ will also trivialize $\mathbb{C}TM/\mathbb{C}D$ as needed. By changing the sign of $\theta^0$ if necessary, we may assume that the ratio of positive to negative eigenvalues of $\mathcal{L}_0$ is at least one, after which $\mathcal{L}_0$ is a determined pointwise up to a scalar which preserves this ratio. By definition of $K$, $\mathcal{L}_0$ descends to a nondegenerate Hermitian form 
\begin{equation}\label{uL0}
 \underline{\mathcal{L}}_0:H/K\times H/K \to \mathbb{C}.
\end{equation}

It is straightforward (\cite[Thm 4.4]{FreemanStraightening}) to show that $\mathcal{C}_0$ also descends to $\underline{\mathcal{C}}_0:K\times H/K\times H/K\to\mathbb{C}$. For $X_x\in K_x$ and $Y_x,Z_x\in H_x$ with $X\in\Gamma(K)$ and $Y,Z\in\Gamma(H)$ locally extending them, let an underline denote the image of a CR vector under the canonical quotient projection $H\to H/K$ (e.g., $\underline{Y}\in\Gamma(H/K)$). We have    
\begin{align*}
 \underline{\mathcal{C}}_0(X_x,\underline{Y}_x,\underline{Z}_x)=\mathcal{C}_0(X_x,Y_x,Z_x)=\im\theta^0|_x([[X,\overline{Y}],\overline{Z}]).
\end{align*}
If we fix $X_x\in K_x$, we can define $\ad_{X_x}:H_x/K_x\to H_x/K_x\cong \mathbb{C}D_x/(K_x\oplus\overline{H}_x)$ by 
\begin{equation}\label{ad}
 \ad_{X_x}(\underline{Y}_x)=[X,\overline{Y}]|_x\mod K_x\oplus\overline{H}_x,
\end{equation}
and $\ad_{X_x}$ is well-defined and tensorial (albeit antilinear) by the integrability of $K\oplus\overline{K}$ and the Leibniz rule for the Lie bracket. Therefore, 
\begin{align*}
 \underline{\mathcal{C}}_0(X_x,\underline{Y}_x,\underline{Z}_x)=\underline{\mathcal{L}}_0(\ad_{X_x}(\underline{Y}_x),\underline{Z}_x),
\end{align*}
and by the nondegeneracy of $\underline{\mathcal{L}}_0$ the cubic form is completely determined by the family of antilinear operators $\ad_X$ for $X\in K$. Note that 2-nondegeneracy implies that $\ad_X$ and $\ad_{X'}$ are linearly independent endomorphisms whenever $X$ and $X'$ are linearly independent. Another property of this family of operators follows from the Jacobi identity,
\begin{align*}
\underline{\mathcal{L}}_0(\ad_{X_x}(\underline{Y}_x),\underline{Z}_x)&=\im\theta^0|_x([[X,\overline{Y}],\overline{Z}])\\
&=\im\theta^0|_x(-\underbrace{[[\overline{Y},\overline{Z}],X]}_{\in H\oplus\overline{H}}-[[\overline{Z},X],\overline{Y}])\\
&=-\im\theta^0|_x([\overline{Y},[X,\overline{Z}]])\\
&=\overline{\underline{\mathcal{L}}_0(\underline{Y}_x,\ad_{X_x}(\underline{Z}_x))}.
\end{align*}

Therefore, the antilinear operators $\ad_{X}$ for $X\in K$ satisfy a sort of normality property with respect to $\underline{\mathcal{L}}_0$. Distinguished among the set of normal operators on a Hermitian inner product space is the group of unitary operators that act bijectively and preserve the inner product. More generally, we could consider those invertible operators which preserve the inner product up to some nonzero conformal factor, and it is in this vein that we offer: 

\begin{definition}
The cubic form $\mathcal{C}$ of a 2-nondegenerate CR manifold $M$ is said to be of \emph{conformal unitary type} if 
\begin{align*}
&\underline{\mathcal{L}}(\ad_X(\underline{Y}),\ad_X(\underline{Z}))=\lambda\overline{\underline{\mathcal{L}}(\underline{Y},\underline{Z})},
&\forall X\in K;\  Y,Z\in H,
\end{align*}
where $\lambda$ is a non-vanishing, $\mathbb{C}$-valued function on $M$.
\end{definition}
Note that the cubic form of a 5-dimensional, 2-nondegenerate CR manifold is automatically of conformal unitary type. This paper will treat the most direct generalization of the hypotheses for the 5-dimensional case. We therefore determine a complete set of local invariants of $M$ under any CR equivalence, where $M$ is a 2-nondegenerate, hypersurface-type CR manifold with 
\begin{align*}
&\dim_\mathbb{R}M=7,
&\text{rank}_\mathbb{C}K=1,
\end{align*}
such that $\mathcal{C}$ is of conformal unitary type. Our hypotheses imply $H/K$ has complex rank 2, so $\underline{\mathcal{L}}_0$ either has signature $(2,0)$ or $(1,1)$. In order to consider the most general case, we let 
\begin{align}\label{epsilon_delta}
&\epsilon=\pm1,
&\delta_1=0,
&&\delta_{-1}=1,
&&\Rightarrow \epsilon=(-1)^{\delta_\epsilon}. 
\end{align}
We can now say that the signature of $\underline{\mathcal{L}}_0$ is $(2-\delta_\epsilon,\delta_\epsilon)$, and any matrix representation of this Hermitian form may be diagonalized with diagonal entries $1,\epsilon$. 

Even so, there are two distinct subcases when $\epsilon=-1$, and the normalizations in the calculation of \S\ref{eqprob} will only permit us to consider one of them simultaneously with our treatment of the definite ($\epsilon=1$) case. Briefly speaking, $H/K$ is the complex span of two $\underline{\mathcal{L}}_0$-isotropic lines when $\epsilon=-1$, and the $\mathbb{R}$-linear action of $\ad_K$ on $H/K$ may or may not preserve the real span of any vectors lying on these isotropic lines, leading to the following    
\begin{definition}\label{isotropy}
When $\underline{\mathcal{L}}$ has signature $(1,1)$ and $\ad_K:H/K\to H/K$ preserves a real, $\underline{\mathcal{L}}$-isotropic line, we say $\ad_K$ is \emph{isotropy-preserving}. Alternatively, the case when $\ad_K$ does not preserve any real isotropic lines will be called \emph{isotropy-switching}.     
\end{definition}

Lemma \ref{ad_K_normalization} will show that either the isotropy-preserving subcase or the isotropy-switching subcase can be studied in conjunction with the definite case, but the indicated choices of normalization necessarily exclude one of these $\epsilon=-1$ subcases. Because the Lie-algebra-valued parallelisms for the $\epsilon=1$ and isotropy-switching scenarios are readily constructed simultaneously (c.f. \S\ref{homogmodel}), we restrict our attention to these. Homogeneous models for all three scenarios are discussed in A. Santi's \cite{Santi_models}.

%%%%%%%%%%%%%%%%%%%%%%%%%%%%%%%%%%%%%%%%%%%%%%%%%%%%%%%%%%%%%%%%%%%%%%%%%%%%%%%%%%
\vspace{\baselineskip}\subsection{Local Coframing Formulation}\label{coframings}
%%%%%%%%%%%%%%%%%%%%%%%%%%%%%%%%%%%%%%%%%%%%%%%%%%%%%%%%%%%%%%%%%%%%%%%%%%%%%%%%%%

A \emph{0-adapted coframing} $\theta$ in a neighborhood of $x\in M$ consists of local one-forms $\theta^0,\theta^1,\theta^2,\theta^3\in\Gamma(\overline{H}^\bot)\subset\Omega^1(M,\mathbb{C})$ -- and their complex conjugates -- so that $\theta$ satisfies

\begin{align*}
 &\theta^0\in\Gamma(D^\bot)\subset\Omega^1(M),
 &\theta^1,\theta^2\in\Gamma(K^\bot)\subset\Omega^1(M,\mathbb{C}),
\end{align*}
\begin{align*}
\theta^0\wedge\theta^1\wedge\theta^2\wedge\theta^3\wedge\theta^{\overline{1}}\wedge\theta^{\overline{2}}\wedge\theta^{\overline{3}}\neq 0.
\end{align*}
Here, $\theta^{\overline{k}}$ denotes the complex conjugate $\overline{\theta^k}$ of a $\mathbb{C}$-valued form. CR integrability $[\overline{H},\overline{H}]\subset\overline{H}$ is equivalent to 
\begin{equation}\label{CRint}
 \dd\theta^i\equiv0\mod\{\theta^0,\theta^1,\theta^2,\theta^3\};\  0\leq i\leq 3,  
\end{equation}
while the integrability of $D^\circ$ (recall that $\mathbb{C}D^\circ=K\oplus\overline{K}$) additionally gives
\begin{equation}\label{Lkerint}
 \dd\theta^l\equiv0\mod\{\theta^0,\theta^1,\theta^2,\theta^{\overline{1}},\theta^{\overline{2}}\};\  0\leq l\leq 2.  
\end{equation}
Furthermore, since $\theta^0$ is $\mathbb{R}$-valued, 
\begin{align}\label{Lform}
 \dd\theta^0\equiv \im\ell_{j\overline{k}}\theta^j\wedge\theta^{\overline{k}}\mod\{\theta^0\}; \ \ (1\leq j,k\leq2),
\end{align}
for some $\ell_{j\overline{k}}=\overline{\ell_{k\overline{j}}}\in C^\infty(M,\mathbb{C})$, where $\ell:=\left[\begin{smallmatrix}\ell_{1\overline{1}}&\ell_{1\overline{2}}\\\ell_{2\overline{1}}&\ell_{2\overline{2}}\end{smallmatrix}\right]$ is nondegenerate and provides a local matrix representation of $\underline{\mathcal{L}}_0$ (as a Hermitian form) as in \eqref{uL0}. 

We invoke \eqref{CRint} and \eqref{Lkerint} to write 
\begin{equation*}
 \dd\theta^j\equiv u^j_{\overline{k}}\theta^3\wedge\theta^{\overline{k}}\mod\{\theta^0,\theta^1,\theta^2\}; \ (1\leq j,k\leq 2),  
\end{equation*}
for some $u^j_{\overline{k}}\in C^\infty(M,\mathbb{C})$, so that $u:=\left[\begin{smallmatrix}u^1_{\overline{1}}&u^1_{\overline{2}}\\u^2_{\overline{1}}&u^2_{\overline{2}}\end{smallmatrix}\right]$ is a local matrix representation of $\ad_{X_3}$ as in \eqref{ad}, where $X_3\in\Gamma(K)$ is dual to $\theta^3$ in our coframing $\theta$ -- i.e., $\theta^3(X_3)=1$ while $\theta^l(X_3)=\theta^{\overline{i}}(X_3)=0$ for $0\leq l\leq 2$ and $1\leq i\leq 3$. The hypothesis of 2-nondegeneracy merely says that the matrix $u$ is not zero, but the hypothesis that the cubic form of $M$ is of conformal unitary type implies that $u$ is (up to conjugation and scale) unitary with respect to the $2\times 2$ matrix $\ell$ -- specifically, $u$ is invertible and 
\begin{align}\label{conformalunitary}
\overline{u}^t\ell u=\lambda\overline{\ell}, 
\end{align}
for some non-vanishing $\lambda\in C^\infty(M,\mathbb{C})$.

Expressing $\theta$ as the column vector $[\theta^0,\theta^1,\theta^2,\theta^3]^t$ and fixing index ranges $1\leq j,k\leq 2$, we can summarize our analysis in this section thusly:
\begin{equation}\label{MSE}
 \dd\theta=\left[\begin{array}{c}\dd\theta^0\\\dd\theta^1\\\dd\theta^2\\\dd\theta^3\end{array}\right]\equiv
 \left[\begin{array}{c}
\im\ell_{j\overline{k}}\theta^j\wedge\theta^{\overline{k}}\\
u^1_{\overline{k}}\theta^3\wedge\theta^{\overline{k}}\\
u^2_{\overline{k}}\theta^3\wedge\theta^{\overline{k}}\\
0\end{array}\right] \mod\left\{\begin{array}{c}
\theta^0\\\theta^0,\theta^1,\theta^2\\\theta^0,\theta^1,\theta^2\\\theta^0,\theta^1,\theta^2,\theta^3\end{array}\right\}.
\end{equation}

We conclude this section with a remark about notation. As we have above, we will continue to denote the conjugate of every $\mathbb{C}$-valued one-form by putting overlines on its indices. By contrast, we indicate the conjugate of a $\mathbb{C}$-valued function with an overline on the function itself, without changing the indices. For example, the conjugate of the second identity in \eqref{MSE} would be written 
\begin{equation*}
 \dd\theta^{\overline{1}}\equiv \overline{u}^1_{\overline{1}}\theta^{\overline{3}}\wedge\theta^{1}+\overline{u}^1_{\overline{2}}\theta^{\overline{3}}\wedge\theta^{2}\mod\{\theta^0,\theta^{\overline{1}},\theta^{\overline{2}}\}.
\end{equation*}

\vspace{\baselineskip}\section{The Equivalence Problem}\label{eqprob}

%%%%%%%%%%%%%%%%%%%%%%%%%%%%%%%%%%%%%%%%%%%%%%%%%%%%%%%%%%%%%%%%%%%%%%%%%%%%%%%%%%
\subsection{Initial $G$-Structure}\label{initialG}
%%%%%%%%%%%%%%%%%%%%%%%%%%%%%%%%%%%%%%%%%%%%%%%%%%%%%%%%%%%%%%%%%%%%%%%%%%%%%%%%%%

Let $V=\mathbb{R}\oplus\mathbb{C}^3$, presented as column vectors
\begin{align*}
 V=\left\{\left[\begin{smallmatrix}r\\z_1\\z_2\\z_3\end{smallmatrix}\right]: r\in\mathbb{R}; z_1,z_2,z_3\in\mathbb{C}\right\}.
\end{align*}
For $x\in M$, a coframe $v_x:T_xM\stackrel{\simeq}{\longrightarrow} V$ is a linear isomorphism that will be called \emph{0-adapted} if 

\begin{itemize}
 \item $v_x(D_x)=\left\{\left[\begin{smallmatrix}0\\z_1\\z_2\\z_3\end{smallmatrix}\right]: z_1,z_2,z_3\in\mathbb{C}\right\}$,\n
 
 \item $v_x|_{D_x}\circ J=\im v_x|_{D_x}$,\n
 
 \item $v_x(D^\circ_x)=\left\{\left[\begin{smallmatrix}0\\0\\0\\z_3\end{smallmatrix}\right]: z_3\in\mathbb{C}\right\}$.
\end{itemize}\n

Let $\underline{\pi}:B_0\to M$ denote the bundle of all 0-adapted coframes, where $\underline{\pi}(v_x)=x$. A local section $s:M\to B_0$ in a neighborhood of $x$ with $s(x)=v_x$ is a 0-adapted coframing $\theta$, written as a column vector like in \S \ref{coframings}, so that $\theta|_x=v_x$. The \emph{tautological one-form} $\eta\in\Omega^1(B_0,V)$ is intrinsically (therefore globally) defined by 
\begin{equation}\label{intautform}
 \eta|_{v_x}(X):=v_x(\underline{\pi}_*(X|_{v_x})), \hspace{1 cm}\forall X\in\Gamma(TB_0).
\end{equation}
It follows directly from the definition of $\eta$ that if $\theta$ is a 0-adapted coframing given by a local section $s$ of $B_0$, then the tautological form satisfies the so-called \emph{reproducing property}: $\theta=s^*\eta$. Naturally, the reproducing property extends to 
\begin{equation}\label{reprop}
 \dd\theta=s^*\dd\eta.
\end{equation}

We will find a local expression for $\eta$ by locally trivializing $B_0$ in a neighborhood of any $x\in M$. To this end, first note that if $v_x,\tilde{v}_x\in B_0$ are two coframes in the fiber over $x$, then by the definition of 0-adaptation, it must be that 

\begin{align}\label{G0}
 &\tilde{v}_x=\left[\begin{array}{cccc}t&0&0&0\\c^1&a^1_1&a^1_2&0\\c^2&a^2_1&a^2_2&0\\c^3&b_1&b_2&b_3\end{array}\right]v_x; &\text{ where } 
&&\left\{\begin{array}{lr} 
 t\in\mathbb{R}\setminus\{0\},&\\
  & \\
 c^j,b_k\in\mathbb{C}\  (b_3\neq 0); &1\leq j,k\leq3,\\
  & \\
 \left[\begin{smallmatrix}a^1_1&a^1_2\\a^2_1&a^2_2\end{smallmatrix}\right]\in GL_2\mathbb{C}.
 \end{array}\right. 
\end{align}
Call the subgroup of $GL(V)$ given by all such matrices $G_0$, and its Lie algebra $\mathfrak{g}_0$. $G_0$ acts transitively on the fibers of $B_0$, so fixing a 0-adapted coframing $\theta_{\mathbbm{1}}$ in a neighborhood of $x$ determines a local trivialization $B_0\cong G_0\times M$, as every other $\theta$ may be written 
\begin{align}\label{B0fibercoords}
 \left[\begin{array}{c}\theta^0\\\theta^1\\\theta^2\\\theta^3\end{array}\right]= 
 \left[\begin{array}{cccc}t&0&0&0\\c^1&a^1_1&a^1_2&0\\c^2&a^2_1&a^2_2&0\\c^3&b_1&b_2&b_3\end{array}\right]
 \left[\begin{array}{c}\theta^0_{\mathbbm{1}}\\\theta^1_{\mathbbm{1}}\\\theta^2_{\mathbbm{1}}\\\theta^3_{\mathbbm{1}}\end{array}\right] 
\end{align}
for some $G_0$-valued matrix of smooth functions defined on our neighborhood of $x$. In this trivialization, the fixed coframing $\theta_{\mathbbm{1}}$ corresponds to the identity matrix $\mathbbm{1}\in G_0$, and by restricting to $\theta|_x$, $\theta_{\mathbbm{1}}|_x$ on each side of \eqref{B0fibercoords}, we see that the $G_0$-valued matrix entries parameterize all $v_x\in B_0$ in the fiber over $x$, hence furnish local fiber coordinates for $B_0$. 

By the reproducing property, the tautological $V$-valued one-form $\eta$ on $B_0$ may now be expressed locally as
\begin{align}\label{loctautformlong}
 \left[\begin{array}{c}\eta^0\\\eta^1\\\eta^2\\\eta^3\end{array}\right]= 
 \left[\begin{array}{cccc}t&0&0&0\\c^1&a^1_1&a^1_2&0\\c^2&a^2_1&a^2_2&0\\c^3&b_1&b_2&b_3\end{array}\right]
 \left[\begin{array}{c}\underline{\pi}^*\theta^0_{\mathbbm{1}}\\\underline{\pi}^*\theta^1_{\mathbbm{1}}\\\underline{\pi}^*\theta^2_{\mathbbm{1}}\\\underline{\pi}^*\theta^3_{\mathbbm{1}}\end{array}\right], 
\end{align}
or more succinctly,
\begin{equation}\label{loctautformshort}
 \eta=g^{-1}\underline{\pi}^*\theta_{\mathbbm{1}}.
\end{equation}
The matrix in \eqref{loctautformlong} is considered to be the inverse $g^{-1}\in C^\infty(B_0,G_0)$ in \eqref{loctautformshort} so that left-multiplication on coframes defines a right-principal $G_0$ action on $B_0$. Differentiating \eqref{loctautformshort} yields the structure equation 
\begin{equation}\label{B0SEshort}
 \dd\eta=-g^{-1}\dd g\wedge\eta+g^{-1}\underline{\pi}^*\dd\theta_{\mathbbm{1}}.
\end{equation}
The \emph{pseudoconnection form} $g^{-1}\dd g$ takes values in the Lie algebra $\mathfrak{g}_0$. We see from the parameterization \eqref{G0} of $G_0$ that $\mathfrak{g}_0$ may be presented as matrices of the form 
\begin{align*}
 \left[\begin{array}{cccc}\tau&0&0&0\\\gamma^1&\alpha^1_1&\alpha^1_2&0\\\gamma^2&\alpha^2_1&\alpha^2_2&0\\\gamma^3&\beta_1&\beta_2&\beta_3\end{array}\right],
\end{align*}
where all of the entries are independent, $\tau\in\mathbb{R}$, and the rest of the entries take arbitrary complex values. For later convenience, we prefer instead to use the following, less obvious choice of parameterization for $\mathfrak{g}_0$:
\begin{align*}
 \left[\begin{array}{cccc}2\tau&0&0&0\\\gamma^1&\alpha^1_1&\alpha^1_2&0\\\gamma^2&\alpha^2_1&\alpha^2_2&0\\\gamma^3&\im\gamma^2-\beta_1&\im\gamma^1-\beta_2&\beta_3\end{array}\right].
\end{align*}
By taking the entries of this matrix to be forms in $\Omega^1(B_0,\mathbb{C})$ which complete $\eta$ to a local coframing of $B_0$, the structure equation \eqref{B0SEshort} can be written
\begin{align}\label{B0SElong}
 \dd \left[\begin{array}{c}\eta^0\\\eta^1\\\eta^2\\\eta^3\end{array}\right]=
 -\left[\begin{array}{cccc}2\tau&0&0&0\\\gamma^1&\alpha^1_1&\alpha^1_2&0\\\gamma^2&\alpha^2_1&\alpha^2_2&0\\\gamma^3&\im\gamma^2-\beta_1&\im\gamma^1-\beta_2&\beta_3\end{array}\right]\wedge
 \left[\begin{array}{c}\eta^0\\\eta^1\\\eta^2\\\eta^3\end{array}\right]+
 \left[\begin{array}{c}\Xi^0\\\Xi^1\\\Xi^2\\\Xi^3\end{array}\right],
\end{align}
where the semibasic two-form $\Xi:=g^{-1}\underline{\pi}^*\dd\theta_{\mathbbm{1}}\in\Omega^2(B_0,V)$ is apparent torsion. Note that the left-hand side of \eqref{B0SEshort} is a globally defined two-form, while the terms on the right-hand side each depend on our local trivialization of $B_0$. In particular, the pseudoconnection forms in the matrix $g^{-1}\dd g$ are determined only up to $\mathfrak{g}_0$-compatible combinations of the semibasic one-forms $\{\eta^j,\eta^{\overline{j}}\}_{j=0}^3$, which will in turn affect the presentation of the apparent torsion forms. We will use this ambiguity to simplify our local expression for $\Xi$, but first we must find what it is.

Fix index ranges $1\leq j,k\leq 2$. The differential reproducing property \eqref{reprop} and the identites \eqref{MSE} imply
\begin{equation*}\begin{aligned}
 \Xi^0&=\im L_{j\overline{k}}\eta^j\wedge\eta^{\overline{k}}+\xi^0_0\wedge\eta^0,\\
 \Xi^j&= U^j_{\overline{k}}\eta^3\wedge\eta^{\overline{k}}+\xi^j_0\wedge\eta^0+\xi^j_1\wedge\eta^1+\xi^j_2\wedge\eta^2,\\
 \Xi^3&=\xi^3_0\wedge\eta^0+\xi^3_1\wedge\eta^1+\xi^3_2\wedge\eta^2+\xi^3_3\wedge\eta^3,
\end{aligned}\end{equation*}
for some unknown, semibasic one-forms $\xi\in\Omega^1(B_0,\mathbb{C})$ (with $\xi^0_0$ $\mathbb{R}$-valued) and functions $L_{j\overline{k}}, U^j_{\overline{k}}\in C^\infty(B_0,\mathbb{C})$ whose value along the coframing $\theta$ described in \S \ref{coframings} would be 
\begin{align}\label{LlUu}
&L_{j\overline{k}}(\theta|_x)=\ell_{j\overline{k}}(x)
&\text{and}
&&U^j_{\overline{k}}(\theta|_x)=u^j_{\overline{k}}(x).
\end{align}

We will ``absorb'' as much of $\Xi$ into our pseudoconnection forms as possible. It is a standard notational abuse to recycle the name of a pseudoconnection form after altering it to absorb apparent torsion. We will try to minimize confusion by denoting modified forms with hats, and then dropping the hats from the notation as each phase of the absorption process terminates. For example, the top line of \eqref{B0SElong} reads
\begin{equation*}
\begin{aligned}
 \dd\eta^0&=-2\tau\wedge\eta^0+\im L_{j\overline{k}}\eta^j\wedge\eta^{\overline{k}}+\xi^0_0\wedge\eta^0\\
&=-(2\tau-\xi^0_0)\wedge\eta^0+\im L_{j\overline{k}}\eta^j\wedge\eta^{\overline{k}},
\end{aligned}\end{equation*}
so if we let $2\hat{\tau}=2\tau-\xi^0_0$, we have simplified the expression to 
\begin{equation*}
  \dd\eta^0=-2\hat{\tau}\wedge\eta^0+\im L_{j\overline{k}}\eta^j\wedge\eta^{\overline{k}}.
\end{equation*}
Observe that $2\hat{\tau}$ must remain $\mathbb{R}$-valued for this absorption to be $\mathfrak{g}_0$-compatible, which is exactly the case as $\xi^0_0$ is $\mathbb{R}$-valued. To absorb the rest of the $\xi$'s, set 
\begin{align*}
&\hat{\alpha}^j_k=\alpha^j_k-\xi^j_k, &&\hat{\gamma}^j=\gamma^j-\xi^j_0, &&&\hat{\gamma}^3=\gamma^3-\xi^3_0,\\
&\hat{\beta}_1=\beta_1-\im\xi^2_0+\xi^3_1, &&\hat{\beta}_2=\beta_2-\im\xi^1_0+\xi^3_2, &&&\hat{\beta}_3=\beta_3-\xi^3_3.
\end{align*}
Now the structure equations \eqref{B0SElong} may be written
\begin{align}\label{B0SE}
 \dd \left[\begin{array}{c}\eta^0\\\eta^1\\\eta^2\\\eta^3\end{array}\right]=
 -\left[\begin{array}{cccc}2\hat{\tau}&0&0&0\\\hat{\gamma}^1&\hat{\alpha}^1_1&\hat{\alpha}^1_2&0\\\hat{\gamma}^2&\hat{\alpha}^2_1&\hat{\alpha}^2_2&0\\\hat{\gamma}^3&\im\hat{\gamma}^2-\hat{\beta}_1&\im\hat{\gamma}^1-\hat{\beta}_2&\hat{\beta}_3\end{array}\right]\wedge
 \left[\begin{array}{c}\eta^0\\\eta^1\\\eta^2\\\eta^3\end{array}\right]+
 \left[\begin{array}{c}\im L_{j\overline{k}}\eta^j\wedge\eta^{\overline{k}}\\U^1_{\overline{k}}\eta^3\wedge\eta^{\overline{k}} \\ U^2_{\overline{k}}\eta^3\wedge\eta^{\overline{k}} \\0\end{array}\right].
\end{align}

%%%%%%%%%%%%%%%%%%%%%%%%%%%%%%%%%%%%%%%%%%%%%%%%%%%%%%%%%%%%%%%%%%%%%%%%%%%%%%%%%%
\vspace{\baselineskip}\subsection{First Two Reductions}\label{first2}
%%%%%%%%%%%%%%%%%%%%%%%%%%%%%%%%%%%%%%%%%%%%%%%%%%%%%%%%%%%%%%%%%%%%%%%%%%%%%%%%%%

We are done absorbing torsion for the moment, so we will drop the hats off of the pseudoconnection forms in \eqref{B0SE}. The remaining torsion terms are not absorbable, but we can normalize them by first ascertaining how the functions $L,U$ in \eqref{B0SE} vary along the fiber over fixed points of $M$, then choosing agreeable values from among those that $L,U$ achieve in each fiber, and finally restricting to a subbundle of $B_0$ determined by the subgroup of $G_0$ which stabilizes the chosen torsion tensor over each fiber. To proceed, first differentiate the equation for $\dd\eta^0$ and reduce modulo $\eta^0,\eta^3,\eta^{\overline{3}}$.

\begin{equation*}
\begin{aligned}
 0&=\dd(\dd\eta^0)\\
&\equiv
\im(\dd L_{1\overline{1}}+ L_{1\overline{1}}(2\tau- \alpha^1_1- \alpha^{\overline{1}}_{\overline{1}})- L_{1\overline{2}}\alpha^{\overline{2}}_{\overline{1}}- L_{2\overline{1}}\alpha^2_1)\wedge\eta^1\wedge\eta^{\overline{1}}\\
&+\im(\dd L_{1\overline{2}}+ L_{1\overline{2}}(2\tau- \alpha^1_1- \alpha^{\overline{2}}_{\overline{2}})- L_{1\overline{1}}\alpha^{\overline{1}}_{\overline{2}}- L_{2\overline{2}}\alpha^2_1)\wedge\eta^1\wedge\eta^{\overline{2}}\\
&+\im(\dd L_{2\overline{1}}+ L_{2\overline{1}}(2\tau- \alpha^2_2- \alpha^{\overline{1}}_{\overline{1}})- L_{1\overline{1}}\alpha^1_2- L_{2\overline{2}}\alpha^{\overline{2}}_{\overline{1}})\wedge\eta^2\wedge\eta^{\overline{1}}\\
&+\im(\dd L_{2\overline{2}}+ L_{2\overline{2}}(2\tau- \alpha^2_2- \alpha^{\overline{2}}_{\overline{2}})- L_{1\overline{2}}\alpha^1_2- L_{2\overline{1}}\alpha^{\overline{1}}_{\overline{2}})\wedge\eta^2\wedge\eta^{\overline{2}}
&\mod\{\eta^0,\eta^3,\eta^{\overline{3}}\}.
\end{aligned}
\end{equation*}
If we momentarily agree that $j\neq k$, we can summarize these conditions 
\begin{equation}\label{dL}
\left.\begin{array}{ll}
\dd L_{j\overline{j}}&\equiv  -L_{j\overline{j}}(2\tau- \alpha^j_j- \alpha^{\overline{j}}_{\overline{j}})+ L_{j\overline{k}}\alpha^{\overline{k}}_{\overline{j}}+ L_{k\overline{j}}\alpha^k_j\\
\dd L_{j\overline{k}}&\equiv -L_{j\overline{k}}(2\tau- \alpha^j_j- \alpha^{\overline{k}}_{\overline{k}})+ L_{j\overline{j}}\alpha^{\overline{j}}_{\overline{k}}+ L_{k\overline{k}}\alpha^k_j\end{array}\right\} \mod\{\eta^0,\eta^1,\eta^2,\eta^3,\eta^{\overline{1}},\eta^{\overline{2}},\eta^{\overline{3}}\}.
\end{equation}
Using the notation \eqref{epsilon_delta}, we will restrict to the subbundle $B_1\subset B_0$ given by the level sets 
\begin{align}\label{Levi_diagonal}
&L_{1\overline{1}}=1,
&L_{2\overline{2}}=\epsilon,
&&L_{1\overline{2}}=L_{2\overline{1}}=0,
\end{align}
which is simply the bundle of 0-adapted coframes in which $\theta^1,\theta^2$ are dual to CR vector fields that are orthonormal for the Levi form. Such coframings must exist, as the Levi form is Hermitian. In the notation of \S\ref{coframings}, $B_1$ is determined by local 0-adapted coframings $\theta$ which additionally satisfy
\begin{align}\label{MSE1}
 \dd\left[\begin{array}{c}\theta^0\\\theta^1\\\theta^2\\\theta^3\end{array}\right]=
 \left[\begin{array}{c}
\im\theta^1\wedge\theta^{\overline{1}}+\epsilon\im\theta^2\wedge\theta^{\overline{2}}\\
u^1_{\overline{1}}\theta^3\wedge\theta^{\overline{1}}+u^1_{\overline{2}}\theta^3\wedge\theta^{\overline{2}}\\
u^2_{\overline{1}}\theta^3\wedge\theta^{\overline{1}}+u^2_{\overline{2}}\theta^3\wedge\theta^{\overline{2}}\\
0\end{array}\right] \mod\left\{\begin{array}{c}
\theta^0\\\theta^0,\theta^1,\theta^2\\\theta^0,\theta^1,\theta^2\\\theta^0,\theta^1,\theta^2,\theta^3\end{array}\right\}.
\end{align}

We call such coframings \emph{1-adapted}, and fix a new $\theta_{\mathbbm{1}}$ among them to locally trivialize $B_1$. Computing directly with the coordinates of $G_0$ as in \eqref{B0fibercoords}, one finds that any such $\theta$ with its Levi form so normalized differs from $\theta_{\mathbbm{1}}$ by an element in $G_0$ with 
\begin{align}\label{G1}
&t=|a^1_1|^2+\epsilon|a^2_1|^2=\epsilon|a^1_2|^2+|a^2_2|^2 & \text{and}
&&a^1_1\overline{a}^1_2+\epsilon a^2_1\overline{a}^2_2=0,
\end{align}
(which together imply $|a^1_1|^2=|a^2_2|^2$). This subgroup $G_1\subset G_0$ is therefore the stabilizer of our choice of torsion normalization, and the structure group of the subbundle $B_1\subset B_0$. When restricted to $B_1$, we see by \eqref{dL} that the pseudoconnection forms satisfy
\begin{align}\label{redtormod1}
 &2\tau\equiv \alpha^1_1+\alpha^{\overline{1}}_{\overline{1}}\equiv \alpha^2_2+\alpha^{\overline{2}}_{\overline{2}},
 &\alpha^1_2+\epsilon\alpha^{\overline{2}}_{\overline{1}}\equiv0 &&\mod\{\eta^0,\eta^1,\eta^2,\eta^3,\eta^{\overline{1}},\eta^{\overline{2}},\eta^{\overline{3}}\}.
\end{align}

Let $\iota_1:B_1\hookrightarrow B_0$ be the inclusion map. When we pull back our coframing of $B_0$ along $\iota_1$ to get a coframing of $B_1$, we introduce new names for some one-forms, but we also recycle many of the current names. For those being recycled, we view the following definition as recursive. Those being recycled are
\begin{align*}
 \left[\begin{array}{c}\eta\\\tau\\\gamma^j\\\beta_k\end{array}\right]:=
 \iota_1^*\left[\begin{array}{c}\eta\\\tau\\\gamma^j\\\beta_k\end{array}\right]; \hspace{1cm}(1\leq j,k\leq 3),
\end{align*}
while we also introduce 
\begin{align}\label{rsdef}
 \left[\begin{array}{c}\varrho\\\varsigma\\\alpha^1\\\xi^1_1\\\zeta^2_1\\\xi^2_2\end{array}\right]:=
 \iota_1^*
 \left[\begin{array}{c}-\tfrac{\im}{2}(\alpha^1_1-\alpha^{\overline{1}}_{\overline{1}})\\-\tfrac{\im}{2}(\alpha^2_2-\alpha^{\overline{2}}_{\overline{2}})\\\alpha^1_2\\\tau-\tfrac{1}{2}(\alpha^1_1+\alpha^{\overline{1}}_{\overline{1}})\\-(\alpha^2_1+\epsilon\alpha^{\overline{1}}_{\overline{2}})\\\tau-\tfrac{1}{2}(\alpha^2_2+\alpha^{\overline{2}}_{\overline{2}})\end{array}\right].
\end{align}
Note that $\xi^1_1$ and $\xi^2_2$ are $\mathbb{R}$-valued, and by \eqref{redtormod1}, we know 
\begin{align}\label{xismod1}
 \xi^1_1,\zeta^2_1,\xi^2_2\equiv0\mod\{\eta^0,\eta^1,\eta^2,\eta^3,\eta^{\overline{1}},\eta^{\overline{2}},\eta^{\overline{3}}\}.
\end{align}
If we keep the names $U^j_{\overline{k}}:=\iota_1^*U^j_{\overline{k}}$, then pulling back \eqref{B0SE} to $B_1$ yields new structure equations
\begin{equation}\label{B1SE}
\begin{aligned} 
 \dd \left[\begin{array}{c}\eta^0\\\eta^1\\\eta^2\\\eta^3\end{array}\right]=
 -\left[\begin{array}{cccc}2\tau&0&0&0\\\gamma^1&\tau+\im\varrho&\alpha^1&0\\\gamma^2&-\epsilon\alpha^{\overline{1}}&\tau+\im\varsigma&0\\\gamma^3&\im\gamma^2-\beta_1&\im\gamma^1-\beta_2&\beta_3\end{array}\right]\wedge
 \left[\begin{array}{c}\eta^0\\\eta^1\\\eta^2\\\eta^3\end{array}\right]
 +
 \left[\begin{array}{c}
 \im\eta^1\wedge\eta^{\overline{1}}+\epsilon\im\eta^2\wedge\eta^{\overline{2}}\\U^1_{\overline{k}}\eta^3\wedge\eta^{\overline{k}}+\xi^1_1\wedge\eta^1 \\ U^2_{\overline{k}}\eta^3\wedge\eta^{\overline{k}}+\zeta^2_1\wedge\eta^1+\xi^2_2\wedge\eta^2 \\0\end{array}\right].
\end{aligned}
\end{equation}

We turn our attention to normalizing the $U^j_{\overline{k}}$. Differentiating $\dd\eta^0$ and reducing modulo $\eta^0,\eta^1,\eta^2$ will reveal that these functions are not independent on $B_1$.
\begin{align*}
 0=\dd(\dd\eta^0)\equiv 
\im (U^1_{\overline{2}}-\epsilon U^2_{\overline{1}})\eta^3\wedge\eta^{\overline{2}}\wedge\eta^{\overline{1}} &\mod\{\eta^0,\eta^1,\eta^2\},
\end{align*}
so $U^1_{\overline{2}}=\epsilon U^2_{\overline{1}}$, and we can declutter some notation by naming 
\begin{align}\label{U_on_B1}
&U:=U^2_{\overline{1}}=\epsilon U^1_{\overline{2}}, 
&U_1:=U^1_{\overline{1}}, 
&&U_2:=U^2_{\overline{2}}. 
\end{align}
The hypothesis that the cubic form is of conformal unitary type implies some additional relations between the functions \eqref{U_on_B1}. From \eqref{conformalunitary}, \eqref{LlUu}, and \eqref{Levi_diagonal}, we deduce
\begin{align}\label{U_conf_unit}
&U_1\overline{U}_1=U_2\overline{U}_2,
&U\overline{U}_1+\overline{U}U_2=0.
\end{align}

To see how the functions \eqref{U_on_B1} vary in a fiber over $x\in M$, we differentiate $\dd\eta^1$ and $\dd\eta^2$ and reduce modulo $\eta^0,\eta^1,\eta^2$.
\begin{align*}
 0&=\dd(\dd\eta^1)\\
 &\equiv 
 (\dd U_1-U_1(\beta_3-2\im\varrho)+2U\alpha^1+\epsilon U\zeta^{\overline{2}}_{\overline{1}})\wedge\eta^3\wedge\eta^{\overline{1}}\\
 &+(\epsilon\dd U-\epsilon U(\beta_3-\im\varrho-\im\varsigma)-U_1\alpha^{\overline{1}}+U_2\alpha^1+\epsilon U(\xi^2_2-\xi^1_1))\wedge\eta^3\wedge\eta^{\overline{2}}&&\mod\{\eta^0,\eta^1,\eta^2\},
\end{align*}
and similarly
\begin{align*}
0&=\dd(\dd\eta^2)\\
&\equiv
(\dd U- U(\beta_3-\im\varrho-\im\varsigma)-\epsilon U_1\alpha^{\overline{1}}+\epsilon U_2\alpha^1+U(\xi^1_1-\xi^2_2)-U_1\zeta^2_1+U_2\zeta^{\overline{2}}_{\overline{1}})\wedge\eta^3\wedge\eta^{\overline{1}}\\
&+(\dd U_2-U_2(\beta_3-2\im\varsigma)-2U\alpha^{\overline{1}}-\epsilon U\zeta^2_1)\wedge\eta^3\wedge\eta^{\overline{2}}
\hspace{3.62cm}\mod\{\eta^0,\eta^1,\eta^2\}.
\end{align*}
With \eqref{xismod1} in mind, we summarize
\begin{align}\label{dUs}
 \left.\begin{array}{lll}
  \dd U_1&\equiv& U_1(\beta_3-2\im\varrho)-2U\alpha^1  \\    
  \dd U&\equiv& U(\beta_3-\im\varrho-\im\varsigma)+\epsilon U_1\alpha^{\overline{1}}-\epsilon U_2\alpha^1 \\
  \dd U_2&\equiv& U_2(\beta_3-2\im\varsigma)+2U\alpha^{\overline{1}}\end{array}\right\}
  \mod\{\eta^0,\eta^1,\eta^2,\eta^3,\eta^{\overline{1}},\eta^{\overline{2}},\eta^{\overline{3}}\}.
\end{align}

Along with the relations \eqref{U_conf_unit}, the conformal unitary condition requires that the matrix $\left[\begin{array}{cr}U_1&\epsilon U\\U&U_2\end{array}\right]$ have full rank. In light of \eqref{U_conf_unit}, the square of the modulus of the determinant of this matrix is 
\begin{align*}
\left|U_1U_2-\epsilon U^2\right|^2&=(U_1U_2-\epsilon U^2)(\overline{U}_1\overline{U}_2-\epsilon \overline{U}^2)\\
&=\left| U_1\right|^4+2\epsilon\left(\left| U\right|\left| U_1\right|\right)^2+\left| U\right|^4\\
&=(\left| U_1\right|^2+\epsilon \left| U\right|^2)^2.
\end{align*} 
When $\epsilon=1$, the determinant is nonzero for any nontrivial matrix satisfying \eqref{U_conf_unit}. However, when $\epsilon=-1$, any matrix with $\left| U_1\right|=\left| U\right|$ is degenerate. The space of matrices satisfying \eqref{U_conf_unit} and having full rank is therefore disconnected when $\epsilon=-1$, and in particular the diagonal matrices ($U=0$, $\left| U_1\right|=\left| U_2\right|\neq0$) lie in a connected component distinct from that of the anti-diagonal matrices ($U_1=U_2=0$, $U\neq0$). We must distinguish between the following two subcases when $\epsilon=-1$:
\begin{align}\label{iso_represented}
&\left| U\right|>\left| U_1\right|=\left| U_2\right|,
&\left| U\right|<\left| U_1\right|=\left| U_2\right|.
\end{align}

\begin{lemma}\label{ad_K_normalization}
If $\epsilon=-1$ and $\left| U(\theta_{\mathbbm{1}}|_x)\right|>\left| U_1(\theta_{\mathbbm{1}}|_x)\right|=\left| U_2(\theta_{\mathbbm{1}}|_x)\right|$, there exists a coframe $\theta$ in the fiber of $B_1$ over $x$ such that 
\begin{align*}
&U(\theta)=1,
&&U_1(\theta)=U_2(\theta)=0.
\end{align*}
Alternatively, if $\epsilon=-1$ and $\left| U(\theta_{\mathbbm{1}}|_x)\right|<\left| U_1(\theta_{\mathbbm{1}}|_x)\right|=\left| U_2(\theta_{\mathbbm{1}}|_x)\right|$, there exists a coframe $\theta$ in the fiber of $B_1$ over $x$ such that 
\begin{align*}
&U(\theta)=0,
&&U_1(\theta)=U_2(\theta)=1.
\end{align*}
When $\epsilon=1$, both such coframes exist in the fiber over $x$.
\end{lemma}

\begin{proof}
It is immediate that $\left| U(\theta_{\mathbbm{1}}|_x)\right|>0$ in the first subcase of $\epsilon=-1$. In order to treat this and the case $\epsilon=1$ simultaneously, we first show that we may assume $\left| U(\theta_{\mathbbm{1}}|_x)\right|>0$. Suppose to the contrary that $U(\theta_{\mathbbm{1}}|_x)=0$. Let $X_\alpha,Y_\alpha\in\Gamma(TB_1)$ be the vertical vector fields dual to the real and imaginary parts of $\alpha^1$; i.e., 
\begin{align*}
&\text{Re}(\alpha^1)(X_\alpha)=\text{Im}(\alpha^1)(Y_\alpha)=1,
&\text{Re}(\alpha^1)(Y_\alpha)=\text{Im}(\alpha^1)(X_\alpha)=0,
\end{align*}
while every other pseudoconnection form (along with the tautological forms) annihilates both $X_\alpha$ and $Y_\alpha$. The fiber of $B_1$ over $x$ is foliated by integral curves of these \emph{fundamental vector fields}, and we name the curves $c_X({\tt t}),c_Y({\tt t}):\mathbb{R}\to B_1$ which pass through $\theta_{\mathbbm{1}}$ when ${\tt t}=0$. By \eqref{dUs}, we calculate
\begin{align*}
 \left.\frac{d}{d{\tt t}}\right|_{{\tt t}=0}U(c_X({\tt t}))&=\dd U\left(\left.\frac{d}{d{\tt t}}\right|_{{\tt t}=0}c_X({\tt t})\right)\\
 &=(U(\beta_3-\im\varrho-\im\varsigma)+\epsilon U_1\alpha^{\overline{1}}-\epsilon U_2\alpha^1)\left(\left.X_\alpha\right|_{c_X(0)}\right)\\
 &=\epsilon(U_1-U_2)(\theta_{\mathbbm{1}}),
\end{align*}
and similarly,
\begin{align*}
 \left.\frac{d}{d{\tt t}}\right|_{{\tt t}=0}U(c_Y({\tt t}))&=(U(\beta_3-\im\varrho-\im\varsigma)+\epsilon U_1\alpha^{\overline{1}}-\epsilon U_2\alpha^1)\left(\left.Y_\alpha\right|_{c_Y(0)}\right)\\
 &=-\epsilon\im(U_1+U_2)(\theta_{\mathbbm{1}}).
\end{align*}
Since we have assumed $U(\theta_{\mathbbm{1}}|_x)=0$, one of these derivatives must be nonzero, hence $U$ is not identically zero in the fiber over $x$. If necessary, we could \emph{flow along the curve} $c_X({\tt t})$ or $c_Y({\tt t})$ by choosing a value of ${\tt t}$ such that $U(c_X({\tt t}))$ or $U(c_Y({\tt t}))$ is nonzero, thereby choosing another $\theta_{\mathbbm{1}}$. We proceed with the assumption $\left| U(\theta_{\mathbbm{1}})\right|>0$. 

Next we demonstrate how to find a coframe $\theta_0$ such that $U(\theta_0)$ is $\mathbb{R}$-valued and positive. To this end, let $X_\varrho\in\Gamma(TB_1)$ be the fundamental vector field dual to $\varrho$, and $c_\varrho({\tt t})$ the integral curve of $X_\varrho$ such that $c_\varrho(0)=\theta_{\mathbbm{1}}$. If we let $\e$ denote the natural exponential, we calculate
\begin{align*}
 \frac{d}{d{\tt t}}U(c_\varrho({\tt t}))
 &=(U(\beta_3-\im\varrho-\im\varsigma)+\epsilon U_1\alpha^{\overline{1}}-\epsilon U_2\alpha^1)\left(X_\varrho\right)\\
 &=-\im U(c_\varrho({\tt t}))\\
\Rightarrow U(c_\varrho({\tt t}))&=U(\theta_{\mathbbm{1}})\e^{-\im{\tt t}}, 
\end{align*}
so for some ${\tt t}_0$ and $\theta_0=c_\varrho({\tt t}_0)$ we can indeed ensure $U(\theta_0)$ is real and positive. The other two equations of \eqref{dUs} will likewise show 
\begin{align*}
&U_1(c_\varrho({\tt t}))=U_1(\theta_{\mathbbm{1}})\e^{-2\im{\tt t}},
&U_2(c_\varrho({\tt t}))=U_2(\theta_{\mathbbm{1}}),
\end{align*}
so flowing along the curve $c_\varrho({\tt t})$ does not change the modulus of any of $U,U_1,U_2$, and we remain in the first subcase for $\epsilon=-1$. The second equation of \eqref{dUs} will reveal how the imaginary part of $U$ varies in each fiber,
\begin{equation}\label{d_U_im}
\begin{aligned}
\dd U-\dd\overline{U}&\equiv \im(U+\overline{U})\left(\text{Im}(\beta_3)-\varrho-\varsigma\right)+\im(U-\overline{U})\text{Re}(\beta_3)+\epsilon (U_1+\overline{U}_2)\alpha^{\overline{1}}-\epsilon (U_2+\overline{U}_1)\alpha^1\\
&\mod\{\eta^0,\eta^1,\eta^2,\eta^3,\eta^{\overline{1}},\eta^{\overline{2}},\eta^{\overline{3}}\}.
\end{aligned}
\end{equation}
Let $B_{1.5}\subset B_1$ be the subbundle defined by the level set $U-\overline{U}=0$ and let $B_{1.5}^\circ\subset B_{1.5}$ be the open neighborhood of $\theta_0$ where $U=\overline{U}\neq 0$ (which is all of $B_{1.5}$ for the first subcase of $\epsilon=-1$). We keep the same names for the tautological and connection forms when pulled back to $B_{1.5}$, though they are not all independent on $B_{1.5}^\circ$. The conformal unitary condition \eqref{U_conf_unit} implies $U_1+\overline{U}_2=0$ on $B_{1.5}^\circ$, so the vanishing of the left-hand-side of \eqref{d_U_im} shows
\begin{align*}
\im\text{Im}(\beta_3)\equiv \im\varrho+\im\varsigma\mod\{\eta^0,\eta^1,\eta^2,\eta^3,\eta^{\overline{1}},\eta^{\overline{2}},\eta^{\overline{3}}\},
\end{align*} 
but aside from this equivalence the pseudoconnection forms remain independent. When pulled back to $B_{1.5}^\circ$ we may therefore write the equations \eqref{dUs} 
\begin{align}\label{dUs_restricted}
 \left.\begin{array}{lll}
  \dd U_1&\equiv& U_1(\text{Re}(\beta_3)+\im\varsigma-\im\varrho)-2U\alpha^1  \\    
  \dd U&\equiv& U\text{Re}(\beta_3)+\epsilon U_1\alpha^{\overline{1}}+\epsilon \overline{U}_1\alpha^1\end{array}\right\}
  \mod\{\eta^0,\eta^1,\eta^2,\eta^3,\eta^{\overline{1}},\eta^{\overline{2}},\eta^{\overline{3}}\}.
\end{align}

Recycling the names $X_\alpha,Y_\alpha\in\Gamma(TB_{1.5})$ for the fundamental vector fields dual to $\text{Re}(\alpha^1)$ and $\text{Im}(\alpha^1)$, respectively, we define on $B_{1.5}^\circ$
\begin{align*}
&\widetilde{X}=-\frac{1}{2U}X_\alpha,
&\widetilde{Y}=-\frac{1}{2U}Y_\alpha,
\end{align*}
and note that the integral curves of $X_\alpha,Y_\alpha$ are also tangent to $\widetilde{X},\widetilde{Y}$ (respectively) on $B_{1.5}^\circ$. We name $c_{\widetilde{X}},c_{\widetilde{Y}}:\mathbb{R}\to B_{1.5}^\circ$ the integral curves of $\widetilde{X},\widetilde{Y}$ that satisfy $c_{\widetilde{X}}(0)=c_{\widetilde{Y}}(0)=\theta_0$. Using the first equation of \eqref{dUs_restricted}, we calculate
\begin{align*}
&\frac{d}{d{\tt t}}U_1(c_{\widetilde{X}}({\tt t}))=1,
&\frac{d}{d{\tt t}}U_2(c_{\widetilde{X}}({\tt t}))=\im,\\
\Rightarrow &U_1(c_{\widetilde{X}}({\tt t}))=U_1(\theta_0)+{\tt t},
&U_1(c_{\widetilde{Y}}({\tt t}))=U_1(\theta_0)+\im{\tt t}.
\end{align*}

Thus we see that in the case $\epsilon=1$ or the first subcase \eqref{iso_represented} of $\epsilon=-1$, we can flow along the integral curves of $\widetilde{X},\widetilde{Y}$ to a coframe where $U_1=U_2=0$. Our assumption that $\left| U\right|>\left| U_1\right|=\left| U_2\right|$ for the first subcase of $\epsilon=-1$ is tacitly used when we flow along $c_{\widetilde{X}}({\tt t})$ and $c_{\widetilde{Y}}({\tt t})$ to reduce $U_1$ and $U_2$ to zero. If $\left| U_1(\theta_0)\right|>\left| U(\theta_0)\right|$, then these integral curves would have to pass through a coframe where $\left| U_1\right|=\left| U\right|$, and we saw that this corresponds to a degenerate matrix when $\epsilon=-1$. Having established the existence of coframes in the fiber of $B_1$ with $U_1=U_2=0$, the first part of the lemma is resolved by flowing along integral curves dual to $\text{Re}(\beta_3)$ and $\text{Im}(\beta_3)$ to a coframe where $U=1$.

The alternative subcase when $\epsilon=-1$ and $\left| U(\theta_{\mathbbm{1}}|_x)\right|<\left| U_1(\theta_{\mathbbm{1}}|_x)\right|=\left| U_2(\theta_{\mathbbm{1}}|_x)\right|$ may be handled in similar fashion, so we merely indicate the steps involved. By hypothesis, $\left| U_1(\theta_{\mathbbm{1}}|_x)\right|=\left| U_2(\theta_{\mathbbm{1}}|_x)\right|>0$. Flowing along curves tangent to the fundamental vector fields that are dual to $\varrho$ and $\varsigma$ will lead to a coframe where $U_1$, $U_2$ are real and positive, hence equal by \eqref{U_conf_unit}. On the level set where the imaginary parts of $U_1=U_2$ are identically zero, $U$ is imaginary by \eqref{U_conf_unit} and the constraints on the pseudoconnection forms leave the imaginary part of $\alpha^1$ independent. Rescaling the fundamental vector field which is dual to $\text{Im}(\alpha^1)$ by $\tfrac{1}{U_1}$ will yield an integral curve that flows to a coframe where $U=0$. 

We conclude by noting that the normalizations for the latter subcase of $\epsilon=-1$ may also be applied when $\epsilon=1$. One simply needs to begin at a coframe where $U_1\neq0$, and the existence of such coframes follows by the same argument with which the proof began.
\end{proof}

Lemma \ref{ad_K_normalization} shows that the two $\epsilon=-1$ subcases \eqref{iso_represented} correspond to the isotropy-switching and isotropy-preserving scenarios identified in Definition \ref{isotropy}. We will normalize so that we can treat the former simultaneously with the $\epsilon=1$ case, leaving the isotropy-preserving subcase for a future article. Therefore, let us restrict to the level set 
\begin{align*}
 &U=1,&U_1=U_2=0,
\end{align*}
which defines a subbundle $\iota_2:B_2\hookrightarrow B_1$ of \emph{2-adapted coframes}. Sections of $B_2$ are local 1-adapted coframings $\theta$ as in \eqref{MSE1}, but which additionally satisfy 
\begin{align}\label{MSE2}
 \dd\left[\begin{array}{c}\theta^0\\\theta^1\\\theta^2\\\theta^3\end{array}\right]=
 \left[\begin{array}{c}
\im\theta^1\wedge\theta^{\overline{1}}+\epsilon\im\theta^2\wedge\theta^{\overline{2}}\\
\epsilon\theta^3\wedge\theta^{\overline{2}}\\
\theta^3\wedge\theta^{\overline{1}}\\
0\end{array}\right] \mod\left\{\begin{array}{c}
\theta^0\\\theta^0,\theta^1,\theta^2\\\theta^0,\theta^1,\theta^2\\\theta^0,\theta^1,\theta^2,\theta^3\end{array}\right\}.
\end{align}
Among such \emph{2-adapted coframings} we fix a new $\theta_{\mathbbm{1}}$ in order to locally trivialize $B_2$. We saw that $B_1$ was locally trivialized $B_1\cong G_1\times M$ by \eqref{B0fibercoords}, where the subgroup $G_1\subset G_0$ was defined by the added conditions \eqref{G1}. Now one calculates that a matrix in $G_1$ applied to the new $\theta_{\mathbbm{1}}$ will preserve our latest normalization if and only if we additionally have 
\begin{align*}
 &a^1_1=b_3a^{\overline{2}}_{\overline{2}}, &a^1_2=\epsilon b_3a^{\overline{2}}_{\overline{1}}, &&a^2_2=b_3a^{\overline{1}}_{\overline{1}}, &&\epsilon a^2_1=b_3a^{\overline{1}}_{\overline{2}}.
\end{align*}
Since the diagonal terms in the matrices are nonvanishing, these relations imply $a^1_2=a^2_1=0$, while $b_3\in\mathbb{C}$ is unimodular. Let $G_2\subset G_1$ denote this reduced group of matrices, which is the structure group of $B_2$. If we let $\e$ denote the natural exponential, then we may parameterize $G_2$ by  
\begin{align}\label{G2}
\left[\begin{array}{cccc}t^2&0&0&0\\c^1& t\e^{\im r}&0&0\\c^2&0&t\e^{\im s}&0\\c^3&b_1&b_2&\e^{\im(r+s)}\end{array}\right];\hspace{1cm} r,s,0\neq t\in\mathbb{R}; c^j,b_k\in\mathbb{C}.
\end{align}

By \eqref{dUs}, we see that when restricted to $B_2$, we have
\begin{align}\label{redtormod2}
 &\beta_3\equiv \im\varrho+\im\varsigma, &\alpha^1\equiv 0 &&\mod\{\eta^0,\eta^1,\eta^2,\eta^3,\eta^{\overline{1}},\eta^{\overline{2}},\eta^{\overline{3}}\}.
\end{align}
Pulling back our coframing along the inclusion $\iota_2$, we rename accordingly. First, some familiar names
\begin{align*}
 \left[\begin{array}{c}\eta\\\tau\\\varrho\\\varsigma\\\gamma\\\beta_1\\\beta_2\end{array}\right]:=
 \iota_2^*\left[\begin{array}{c}\eta\\\tau\\\varrho\\\varsigma\\\gamma\\\beta_1\\\beta_2\end{array}\right].
\end{align*}
The only new forms we must define are semibasic by \eqref{redtormod2}, viz,
\begin{align*}
 \left[\begin{array}{c}\xi^1_2\\\xi^3_3\end{array}\right]:=
 \iota_2^*\left[\begin{array}{c}-\alpha^{1}\\-\beta_3+\im\varrho+\im\varsigma\end{array}\right].
\end{align*}
We will also preserve the names of the unknown apparent torsion forms on $B_1$, except to combine terms where appropriate:
\begin{align*}
 \left[\begin{array}{c}\xi^1_1\\\xi^2_2\\\xi^2_1\end{array}\right]:=
 \iota_2^*\left[\begin{array}{c}\xi^1_1\\\xi^2_2\\\zeta^2_1+\epsilon\alpha^{\overline{1}}\end{array}\right].
\end{align*}
Pulling back \eqref{B1SE} along $\iota_2$ yields new structure equations on $B_2$:
\begin{align}\label{B2SE}
 \dd \left[\begin{array}{c}\eta^0\\\eta^1\\\eta^2\\\eta^3\end{array}\right]=
 -\left[\begin{array}{cccc}2\tau&0&0&0\\\gamma^1&\tau+\im\varrho&0&0\\\gamma^2&0&\tau+\im\varsigma&0\\\gamma^3&\im\gamma^2-\beta_1&\im\gamma^1-\beta_2&\im\varrho+\im\varsigma\end{array}\right]\wedge
 \left[\begin{array}{c}\eta^0\\\eta^1\\\eta^2\\\eta^3\end{array}\right]+
 \left[\begin{array}{c}
 \im\eta^1\wedge\eta^{\overline{1}}+\epsilon\im\eta^2\wedge\eta^{\overline{2}}\\\epsilon\eta^3\wedge\eta^{\overline{2}}+\xi^1_1\wedge\eta^1+\xi^1_2\wedge\eta^2 \\ \eta^3\wedge\eta^{\overline{1}}+\xi^2_1\wedge\eta^1+\xi^2_2\wedge\eta^2 \\\xi^3_3\wedge\eta^3\end{array}\right],
\end{align}
where $\xi^1_1,\xi^2_2$ are still $\mathbb{R}$-valued, and by \eqref{xismod1},\eqref{redtormod2}, we can say 
\begin{align}\label{xismod2}
 \xi^1_1,\xi^1_2,\xi^2_1,\xi^2_2,\xi^3_3\equiv0\mod\{\eta^0,\eta^1,\eta^2,\eta^3,\eta^{\overline{1}},\eta^{\overline{2}},\eta^{\overline{3}}\}.
\end{align}

%%%%%%%%%%%%%%%%%%%%%%%%%%%%%%%%%%%%%%%%%%%%%%%%%%%%%%%%%%%%%%%%%%%%%%%%%%%%%%%%%%
\vspace{\baselineskip}\subsection{Absorption}\label{absorption}
%%%%%%%%%%%%%%%%%%%%%%%%%%%%%%%%%%%%%%%%%%%%%%%%%%%%%%%%%%%%%%%%%%%%%%%%%%%%%%%%%%

This section is devoted to absorbing as much as we can of the apparent torsion from the $\xi$'s in \eqref{B2SE}. It is easy to see that we can absorb any $\eta^0$ components of these forms into the $\gamma$'s (using the $\beta$'s to correct the equation for $\dd\eta^3$ if necessary). As such, we suppress these components when we adduce \eqref{xismod2} to expand $\xi^i_j=f^i_{jk}\eta^k+t^i_{j\overline{k}}\eta^{\overline{k}}$:

\begin{equation*}\begin{aligned}
 \xi^1_1&=f^1_{11}\eta^1+f^1_{12}\eta^2+f^1_{13}\eta^3+\overline{f}^1_{11}\eta^{\overline{1}}+\overline{f}^1_{12}\eta^{\overline{2}}+\overline{f}^1_{13}\eta^{\overline{3}},  \\
 \xi^2_2&=f^2_{21}\eta^1+f^2_{22}\eta^2+f^2_{23}\eta^3+\overline{f}^2_{21}\eta^{\overline{1}}+\overline{f}^2_{22}\eta^{\overline{2}}+\overline{f}^2_{23}\eta^{\overline{3}},  \\
 \xi^1_2&=f^1_{21}\eta^1+f^1_{22}\eta^2+f^1_{23}\eta^3+t^1_{2\overline{1}}\eta^{\overline{1}}+t^1_{2\overline{2}}\eta^{\overline{2}}+t^1_{2\overline{3}}\eta^{\overline{3}}, \\
 \xi^2_1&=f^2_{11}\eta^1+f^2_{12}\eta^2+f^2_{13}\eta^3+t^2_{1\overline{1}}\eta^{\overline{1}}+t^2_{1\overline{2}}\eta^{\overline{2}}+t^2_{1\overline{3}}\eta^{\overline{3}},  \\
 \xi^3_3&=f^3_{31}\eta^1+f^3_{32}\eta^2+f^3_{33}\eta^3+t^3_{3\overline{1}}\eta^{\overline{1}}+t^3_{3\overline{2}}\eta^{\overline{2}}+t^3_{3\overline{3}}\eta^{\overline{3}},  
\end{aligned}\end{equation*} 
for some functions $f,t\in C^\infty(B_2,\mathbb{C})$. Because $\xi^1_1$ and $\xi^2_2$ are $\mathbb{R}$-valued, $t^j_{j\overline{k}}=\overline{f}^j_{jk}$ for $j=1,2$. Though these coefficients are unknown, we discover relationships between them by differentiating the structure equations. First differentiate $\im\dd\eta^0$ and reduce modulo $\eta^0$.

\begin{align*}
0&=\dd(\im\dd\eta^0)\\
&\equiv-2\xi^1_1\wedge\eta^1\wedge\eta^{\overline{1}}-(\xi^1_2+\epsilon\xi^{\overline{2}}_{\overline{1}})\wedge\eta^2\wedge\eta^{\overline{1}}-(\epsilon\xi^2_1+\xi^{\overline{1}}_{\overline{2}})\wedge\eta^1\wedge\eta^{\overline{2}}-2\epsilon\xi^2_2\wedge\eta^2\wedge\eta^{\overline{2}}\\
&\equiv(2f^1_{12}-f^1_{21}-\epsilon\overline{t}^2_{1\overline{1}})\eta^2\wedge\eta^1\wedge\eta^{\overline{1}}+(2\overline{f}^1_{12}-\overline{f}^1_{21}-\epsilon t^2_{1\overline{1}})\eta^{\overline{2}}\wedge\eta^1\wedge\eta^{\overline{1}}+2f^1_{13}\eta^3\wedge\eta^1\wedge\eta^{\overline{1}}\\
&+(2\epsilon f^2_{21}-\epsilon f^2_{12}-\overline{t}^1_{2\overline{2}})\eta^1\wedge\eta^2\wedge\eta^{\overline{2}}+(2\epsilon\overline{f}^2_{21}-\epsilon\overline{f}^2_{12}-t^1_{2\overline{2}})\eta^{\overline{1}}\wedge\eta^2\wedge\eta^{\overline{2}}+2\epsilon f^2_{23}\eta^3\wedge\eta^2\wedge\eta^{\overline{2}}\\
&+(f^1_{23}+\epsilon\overline{t}^2_{1\overline{3}})\eta^3\wedge\eta^2\wedge\eta^{\overline{1}}+(\overline{f}^1_{23}+\epsilon t^2_{1\overline{3}})\eta^{\overline{3}}\wedge\eta^1\wedge\eta^{\overline{2}}+2\overline{f}^1_{13}\eta^{\overline{3}}\wedge\eta^1\wedge\eta^{\overline{1}}\\
&+(\epsilon f^2_{13}+\overline{t}^1_{2\overline{3}})\eta^3\wedge\eta^1\wedge\eta^{\overline{2}}+(\epsilon\overline{f}^2_{13}+t^1_{2\overline{3}})\eta^{\overline{3}}\wedge\eta^2\wedge\eta^{\overline{1}}+2\epsilon\overline{f}^2_{23}\eta^{\overline{3}}\wedge\eta^2\wedge\eta^{\overline{2}} &&\mod\{\eta^0\}.
\end{align*}
Coefficients of independent three-forms vanish independently, so this has revealed six distinct vanishing conditions and their complex conjugates. For example, we now know that $f^1_{13}=f^2_{23}=0$. We will see that these six equations allow us to simplify our apparent torsion tensor via absorption, but first we find five more equations by differentiating $\dd\eta^1$ and $\dd\eta^2$ and reducing modulo $\eta^0,\eta^1,\eta^2$.
\begin{align*}
0&=\dd(\dd\eta^1)\\
&\equiv\epsilon(\xi^3_3+\xi^2_2-\xi^1_1)\wedge\eta^3\wedge\eta^{\overline{2}}+(\epsilon\xi^{\overline{2}}_{\overline{1}}-\xi^1_2)\wedge\eta^3\wedge\eta^{\overline{1}}\\
&\equiv (\epsilon\overline{f}^2_{21}-\epsilon\overline{f}^2_{12}-\epsilon\overline{f}^1_{11}+\epsilon t^3_{3\overline{1}}+t^1_{2\overline{2}})\eta^{\overline{1}}\wedge\eta^3\wedge\eta^{\overline{2}}+\epsilon t^3_{3\overline{3}}\eta^{\overline{3}}\wedge\eta^3\wedge\eta^{\overline{2}}+(\epsilon\overline{f}^2_{13}-t^1_{2\overline{3}})\eta^{\overline{3}}\wedge\eta^3\wedge\eta^{\overline{1}}\\&\mod\{\eta^0,\eta^1,\eta^2\},
\end{align*}
and similarly,
\begin{align*}
0&=\dd(\dd\eta^2)\\
&\equiv(\xi^3_3+\xi^1_1-\xi^2_2)\wedge\eta^3\wedge\eta^{\overline{1}}+(\xi^{\overline{1}}_{\overline{2}}-\epsilon\xi^2_1)\wedge\eta^3\wedge\eta^{\overline{2}}\\
&\equiv (\overline{f}^1_{12}-\overline{f}^2_{22}-\overline{f}^1_{21}+t^3_{3\overline{2}}+\epsilon t^2_{1\overline{1}})\eta^{\overline{2}}\wedge\eta^3\wedge\eta^{\overline{1}}+t^3_{3\overline{3}}\eta^{\overline{3}}\wedge\eta^3\wedge\eta^{\overline{1}}+(\overline{f}^1_{23}-\epsilon t^2_{1\overline{3}})\eta^{\overline{3}}\wedge\eta^3\wedge\eta^{\overline{2}}\\&\mod\{\eta^0,\eta^1,\eta^2\}.
\end{align*}
In addition to concluding that 
\begin{align*}
f^1_{13}=f^2_{23}=t^3_{3\overline{3}}=0, 
\end{align*}
we have eight vanishing conditions. The first four 
\begin{align*}
 0&=f^1_{23}+\epsilon\overline{t}^2_{1\overline{3}}, \\
  0&=\overline{f}^1_{23}-\epsilon t^2_{1\overline{3}}, \\
  0&=\epsilon f^2_{13}+\overline{t}^1_{2\overline{3}}, \\
  0&=\epsilon\overline{f}^2_{13}-t^1_{2\overline{3}},
\end{align*}
imply 
\begin{align*}
f^1_{23}=f^2_{13}=\overline{t}^1_{2\overline{3}}=\overline{t}^2_{1\overline{3}}=0, 
\end{align*}
while the latter four
\begin{equation}\begin{aligned}\label{torvan}
  0&=2\epsilon f^2_{21}-\epsilon f^2_{12}-\overline{t}^1_{2\overline{2}}, \\
  0&=2f^1_{12}-f^1_{21}-\epsilon\overline{t}^2_{1\overline{1}}, \\
  0&=\epsilon\overline{f}^2_{21}-\epsilon\overline{f}^2_{12}-\epsilon\overline{f}^1_{11}+\epsilon t^3_{3\overline{1}}+t^1_{2\overline{2}}, \\
  0&=\overline{f}^1_{12}-\overline{f}^2_{22}-\overline{f}^1_{21}+t^3_{3\overline{2}}+\epsilon t^2_{1\overline{1}},
\end{aligned}\end{equation}
will be useful for absorbing the remaining terms. The structure equations \eqref{B2SE} may now be expanded to read
\begin{equation}\label{B2SExpand}
\begin{aligned}
 \dd \left[\begin{array}{c}\eta^0\\\eta^1\\\eta^2\\\eta^3\end{array}\right]&=
 -\left[\begin{array}{cccc}2\tau&0&0&0\\\gamma^1&\tau+\im\varrho&0&0\\\gamma^2&0&\tau+\im\varsigma&0\\\gamma^3&\im\gamma^2-\beta_1&\im\gamma^1-\beta_2&\im\varrho+\im\varsigma\end{array}\right]\wedge
 \left[\begin{array}{c}\eta^0\\\eta^1\\\eta^2\\\eta^3\end{array}\right]\\
&\\
 &+
 \left[\begin{array}{c}
 \im\eta^1\wedge\eta^{\overline{1}}+\epsilon\im\eta^2\wedge\eta^{\overline{2}}\\
 \epsilon\eta^3\wedge\eta^{\overline{2}}+(f^1_{12}\eta^2+\overline{f}^1_{11}\eta^{\overline{1}}+\overline{f}^1_{12}\eta^{\overline{2}})\wedge\eta^1+(f^1_{21}\eta^1+t^1_{2\overline{1}}\eta^{\overline{1}}+t^1_{2\overline{2}}\eta^{\overline{2}})\wedge\eta^2 \\
 \eta^3\wedge\eta^{\overline{1}}+(f^2_{12}\eta^2+t^2_{1\overline{1}}\eta^{\overline{1}}+t^2_{1\overline{2}}\eta^{\overline{2}})\wedge\eta^1+(f^2_{21}\eta^1+\overline{f}^2_{21}\eta^{\overline{1}}+\overline{f}^2_{22}\eta^{\overline{2}})\wedge\eta^2 \\
 (f^3_{31}\eta^1+f^3_{32}\eta^2+t^3_{3\overline{1}}\eta^{\overline{1}}+t^3_{3\overline{2}}\eta^{\overline{2}})\wedge\eta^3\end{array}\right].
\end{aligned}
\end{equation}

 We will simplify notation by focusing only on those two-forms which are involved in each step of the absorption. For example, in the structure equation for $\dd\eta^3$, we have
\begin{align*}
 \dd\eta^3&=\beta_1\wedge\eta^1+\beta_2\wedge\eta^2+f^3_{31}\eta^1\wedge\eta^3+f^3_{32}\eta^2\wedge\eta^3+\dots\\
 &=(\beta_1-f^3_{31}\eta^3)\wedge\eta^1+(\beta_2-f^3_{32}\eta^3)\wedge\eta^2+\dots 
\end{align*}
so we let $\hat{\beta}_1=\beta_1-f^3_{31}\eta^3$ and $\hat{\beta}_2=\beta_2-f^3_{32}\eta^3$ to absorb these terms. Now that they are gone, we drop the hats off of $\beta_1,\beta_2$, as we will need to modify them again when considering other terms. Many of the remaining aborbable terms will be absorbed into the diagonal pseudoconnection forms $\im\varrho$ and $\im\varsigma$. Note that we can only alter them by purely imaginary, semibasic one-forms. Before proceeding, we state that the result of our absorption will be that the apparent torsion tensor in \eqref{B2SExpand} will be changed to   
 \begin{align}\label{intorsion}
 \left[\begin{array}{c}
 \im\eta^1\wedge\eta^{\overline{1}}+\epsilon\im\eta^2\wedge\eta^{\overline{2}}\\
 \epsilon\eta^3\wedge\eta^{\overline{2}}+\epsilon(t^1_{2\overline{2}}\eta^{\overline{1}}+t^2_{1\overline{1}}\eta^{\overline{2}})\wedge\eta^1+(t^1_{2\overline{1}}\eta^{\overline{1}}+t^1_{2\overline{2}}\eta^{\overline{2}})\wedge\eta^2 \\
 \eta^3\wedge\eta^{\overline{1}}+(t^2_{1\overline{1}}\eta^{\overline{1}}+t^2_{1\overline{2}}\eta^{\overline{2}})\wedge\eta^1+\epsilon(t^1_{2\overline{2}}\eta^{\overline{1}}+t^2_{1\overline{1}}\eta^{\overline{2}})\wedge\eta^2 \\
 0\end{array}\right].
\end{align}
 
 We will arrive at \eqref{intorsion} in two steps -- one for each of the apparent torsion coefficients $t^3_{3\overline{1}}$ and $t^3_{3\overline{2}}$ that currently remain in the equation for $\dd\eta^3$ in \eqref{B2SExpand}. First consider
\begin{align*}
 \dd\eta^3&=\beta_1\wedge\eta^1-(\im\varrho+\im\varsigma)\wedge\eta^3+t^3_{3\overline{1}}\eta^{\overline{1}}\wedge\eta^3+\dots \\
 &=(\beta_1-\overline{t}^3_{3\overline{1}}\eta^{3})\wedge\eta^1-(\im\varrho+\im\varsigma-t^3_{3\overline{1}}\eta^{\overline{1}}+\overline{t}^3_{3\overline{1}}\eta^{1})\wedge\eta^3+\dots 
\end{align*}
Let $\hat{\beta}_1:=\beta_1-\overline{t}^3_{3\overline{1}}\eta^{3}$. Note that if we choose any imaginary form $\zeta\in\Omega^1(B_2,\im\mathbb{R})$, and define  
\begin{align}\label{iris1}
 &\im\hat{\varrho}:=\im\varrho-\tfrac{1}{2}(t^3_{3\overline{1}}\eta^{\overline{1}}-\overline{t}^3_{3\overline{1}}\eta^{1})+\zeta,
 &\im\hat{\varsigma}:=\im\varsigma-\tfrac{1}{2}(t^3_{3\overline{1}}\eta^{\overline{1}}-\overline{t}^3_{3\overline{1}}\eta^{1})-\zeta,
\end{align}
then we have successfully absorbed the $t^3_{3\overline{1}}$ term in the expression for $\dd\eta^3$. We will choose $\zeta$ so that we also absorb terms in the expressions for $\dd\eta^1,\dd\eta^2$. Let
\begin{align*}
 \zeta:&=-\tfrac{1}{2}\left(\overline{f}^1_{11}-\overline{f}^2_{12}+\overline{f}^2_{21}-\epsilon t^1_{2\overline{2}}\right)\eta^{\overline{1}}+\tfrac{1}{2}\left(f^1_{11}-f^2_{12}+f^2_{21}-\epsilon\overline{t}^1_{2\overline{2}}\right)\eta^1.
\end{align*}
By the third equation in \eqref{torvan},
\begin{align*}
 t^3_{3\overline{1}}\eta^{\overline{1}}-\overline{t}^3_{3\overline{1}}\eta^{1}&=\left(-\overline{f}^2_{21}+\overline{f}^2_{12}+\overline{f}^1_{11}-\epsilon t^1_{2\overline{2}}\right)\eta^{\overline{1}}-\left(-f^2_{21}+f^2_{12}+f^1_{11}-\epsilon\overline{t}^1_{2\overline{2}}\right)\eta^1,
\end{align*}
so in \eqref{iris1} we have 
\begin{align}
 &\im\hat{\varrho}=\im\varrho-\overline{f}^1_{11}\eta^{\overline{1}}+\epsilon t^1_{2\overline{2}}\eta^{\overline{1}}+f^1_{11}\eta^1-\epsilon\overline{t}^1_{2\overline{2}}\eta^1,\label{ir1}\\
 &\im\hat{\varsigma}=\im\varsigma+\overline{f}^2_{21}\eta^{\overline{1}}-\overline{f}^2_{12}\eta^{\overline{1}}-f^2_{21}\eta^1+f^2_{12}\eta^1.\label{is1}
\end{align}
Now \eqref{ir1} shows
\begin{align*}
 \dd\eta^1&=-\im\varrho\wedge\eta^1+\overline{f}^1_{11}\eta^{\overline{1}}\wedge\eta^1+\dots\\
 &=-(\im\varrho-\overline{f}^1_{11}\eta^{\overline{1}}+\epsilon t^1_{2\overline{2}}\eta^{\overline{1}}+f^1_{11}\eta^1-\epsilon\overline{t}^1_{2\overline{2}}\eta^1)\wedge\eta^1+\epsilon t^1_{2\overline{2}}\eta^{\overline{1}}\wedge\eta^1+\dots\\
 &=-\im\hat{\varrho}\wedge\eta^1+\epsilon t^1_{2\overline{2}}\eta^{\overline{1}}\wedge\eta^1+\dots
\end{align*}
On the other hand, by the first equation in \eqref{torvan} we can write \eqref{is1} as 
\begin{align*}
 \im\hat{\varsigma}&=\im\varsigma-\overline{f}^2_{21}\eta^{\overline{1}}+(2\overline{f}^2_{21}-\overline{f}^2_{12})\eta^{\overline{1}}-f^2_{21}\eta^1+f^2_{12}\eta^1\\
 &=\im\varsigma-\overline{f}^2_{21}\eta^{\overline{1}}+\epsilon t^1_{2\overline{2}}\eta^{\overline{1}}-f^2_{21}\eta^1+f^2_{12}\eta^1,
\end{align*}
 which shows
\begin{align*}
 \dd\eta^2&=-\im\varsigma\wedge\eta^2+f^2_{12}\eta^2\wedge\eta^1+f^2_{21}\eta^1\wedge\eta^2+\overline{f}^2_{21}\eta^{\overline{1}}\wedge\eta^2+\dots\\
 &=-(\im\varsigma-\overline{f}^2_{21}\eta^{\overline{1}}+\epsilon t^1_{2\overline{2}}\eta^{\overline{1}}-f^2_{21}\eta^1+f^2_{12}\eta^1)\wedge\eta^2+\epsilon t^1_{2\overline{2}}\eta^{\overline{1}}\wedge\eta^2+\dots\\
 &=-\im\hat{\varsigma}\wedge\eta^2+\epsilon t^1_{2\overline{2}}\eta^{\overline{1}}\wedge\eta^2+\dots
\end{align*}
This concludes the first step of the absorption, by which we modified \eqref{B2SExpand} to yield  
 \begin{align*}
 \dd \left[\begin{array}{c}\eta^0\\\eta^1\\\eta^2\\\eta^3\end{array}\right]&=
 -\left[\begin{array}{cccc}2\tau&0&0&0\\\gamma^1&\tau+\im\hat{\varrho}&0&0\\\gamma^2&0&\tau+\im\hat{\varsigma}&0\\\gamma^3&\im\gamma^2-\hat{\beta}_1&\im\gamma^1-\beta_2&\im\hat{\varrho}+\im\hat{\varsigma}\end{array}\right]\wedge
 \left[\begin{array}{c}\eta^0\\\eta^1\\\eta^2\\\eta^3\end{array}\right]\\
&\\
 &+
 \left[\begin{array}{c}
 \im\eta^1\wedge\eta^{\overline{1}}+\epsilon\im\eta^2\wedge\eta^{\overline{2}}\\
 \epsilon\eta^3\wedge\eta^{\overline{2}}+(f^1_{12}\eta^2+\epsilon t^1_{2\overline{2}}\eta^{\overline{1}}+\overline{f}^1_{12}\eta^{\overline{2}})\wedge\eta^1+(f^1_{21}\eta^1+t^1_{2\overline{1}}\eta^{\overline{1}}+t^1_{2\overline{2}}\eta^{\overline{2}})\wedge\eta^2 \\
 \eta^3\wedge\eta^{\overline{1}}+(t^2_{1\overline{1}}\eta^{\overline{1}}+t^2_{1\overline{2}}\eta^{\overline{2}})\wedge\eta^1+(\epsilon t^1_{2\overline{2}}\eta^{\overline{1}}+\overline{f}^2_{22}\eta^{\overline{2}})\wedge\eta^2 \\
 t^3_{3\overline{2}}\eta^{\overline{2}}\wedge\eta^3\end{array}\right].
\end{align*}

We begin round two by dropping the hats off the pseudoconnection forms. Round two will proceed analogously to round one, only this time we will use the two remaining vanishing conditions; i.e., the second and the last equations of \eqref{torvan}. We have 
\begin{align*}
 \dd\eta^3&=\beta_2\wedge\eta^2-(\im\varrho+\im\varsigma)\wedge\eta^3+t^3_{3\overline{2}}\eta^{\overline{2}}\wedge\eta^3+\dots\\
 &=(\beta_2-\overline{t}^3_{3\overline{2}}\eta^3)\wedge\eta^2-(\im\varrho+\im\varsigma-t^3_{3\overline{2}}\eta^{\overline{2}}+\overline{t}^3_{3\overline{2}}\eta^{2})\wedge\eta^3+\dots
\end{align*}
so let $\hat{\beta}_2=\beta_2-\overline{t}^3_{3\overline{2}}\eta^3$. We'll look for a new semibasic $\zeta\in\Omega^1(B_2,\im\mathbb{R})$ to write  
\begin{align}\label{iris2}
 &\im\hat{\varrho}:=\im\varrho-\tfrac{1}{2}(t^3_{3\overline{2}}\eta^{\overline{2}}-\overline{t}^3_{3\overline{2}}\eta^{2})+\zeta,
 &\im\hat{\varsigma}:=\im\varsigma-\tfrac{1}{2}(t^3_{3\overline{2}}\eta^{\overline{2}}-\overline{t}^3_{3\overline{2}}\eta^{2})-\zeta,
 \end{align}
and use the fact that the final equation in \eqref{torvan} implies 
\begin{align*}
 t^3_{3\overline{2}}\eta^{\overline{2}}-\overline{t}^3_{3\overline{2}}\eta^{2}=\left(-\overline{f}^1_{12}+\overline{f}^2_{22}+\overline{f}^1_{21}-\epsilon t^2_{1\overline{1}}\right)\eta^{\overline{2}}
 -\left(-f^1_{12}+f^2_{22}+f^1_{21}-\epsilon\overline{t}^2_{1\overline{1}}\right)\eta^{2}.
\end{align*}
This time, define 
\begin{align*}
 \zeta:=\tfrac{1}{2}\left(\overline{f}^1_{12}+\overline{f}^2_{22}-\overline{f}^1_{21}-\epsilon t^2_{1\overline{1}}\right)\eta^{\overline{2}}
 -\tfrac{1}{2}\left(f^1_{12}+f^2_{22}-f^1_{21}-\epsilon\overline{t}^2_{1\overline{1}}\right)\eta^{2},
\end{align*}
so that \eqref{iris2} reads
\begin{align*}
 \im\hat{\varsigma}&=\im\varsigma-\overline{f}^2_{22}\eta^{\overline{2}}+\epsilon t^2_{1\overline{1}}\eta^{\overline{2}}+f^2_{22}\eta^{2}-\epsilon\overline{t}^2_{1\overline{1}}\eta^{2},\\
 \im\hat{\varrho}&=\im\varrho+\overline{f}^1_{12}\eta^{\overline{2}}-\overline{f}^1_{21}\eta^{\overline{2}}-f^1_{12}\eta^{2}+f^1_{21}\eta^{2}\\
 &=\im\varrho-\overline{f}^1_{12}\eta^{\overline{2}}+\epsilon t^2_{1\overline{1}}\eta^{\overline{2}}-f^1_{12}\eta^{2}+f^1_{21}\eta^{2},
 \end{align*}
where the last equality follows from the second equation in \eqref{torvan}. As promised, we now have 
\begin{equation}\label{B2SEabs}
\begin{aligned}
 \dd \left[\begin{array}{c}\eta^0\\\eta^1\\\eta^2\\\eta^3\end{array}\right]&=
 -\left[\begin{array}{cccc}2\tau&0&0&0\\\gamma^1&\tau+\im\hat{\varrho}&0&0\\\gamma^2&0&\tau+\im\hat{\varsigma}&0\\\gamma^3&\im\gamma^2-\beta_1&\im\gamma^1-\hat{\beta}_2&\im\hat{\varrho}+\im\hat{\varsigma}\end{array}\right]\wedge
 \left[\begin{array}{c}\eta^0\\\eta^1\\\eta^2\\\eta^3\end{array}\right]\\
&\\
 &+
 \left[\begin{array}{c}
 \im\eta^1\wedge\eta^{\overline{1}}+\epsilon\im\eta^2\wedge\eta^{\overline{2}}\\
 \epsilon\eta^3\wedge\eta^{\overline{2}}+\epsilon(t^1_{2\overline{2}}\eta^{\overline{1}}+t^2_{1\overline{1}}\eta^{\overline{2}})\wedge\eta^1+(t^1_{2\overline{1}}\eta^{\overline{1}}+t^1_{2\overline{2}}\eta^{\overline{2}})\wedge\eta^2 \\
 \eta^3\wedge\eta^{\overline{1}}+(t^2_{1\overline{1}}\eta^{\overline{1}}+t^2_{1\overline{2}}\eta^{\overline{2}})\wedge\eta^1+\epsilon(t^1_{2\overline{2}}\eta^{\overline{1}}+t^2_{1\overline{1}}\eta^{\overline{2}})\wedge\eta^2 \\
 0\end{array}\right].
\end{aligned}\end{equation}

%%%%%%%%%%%%%%%%%%%%%%%%%%%%%%%%%%%%%%%%%%%%%%%%%%%%%%%%%%%%%%%%%%%%%%%%%%%%%%%%%%
\vspace{\baselineskip}\subsection{Last Two Reductions}\label{last2}
%%%%%%%%%%%%%%%%%%%%%%%%%%%%%%%%%%%%%%%%%%%%%%%%%%%%%%%%%%%%%%%%%%%%%%%%%%%%%%%%%%

After removing the hats from our pseudoconnection forms, we normalize some of the remaining torsion coefficients and reduce the structure group as before. To see how these functions vary in the fiber, we first differentiate $\dd\eta^1$ and reduce modulo $\eta^0,\eta^1,\eta^3$.
\begin{align*}
0&=\dd(\dd\eta^1)\\
&\equiv (\dd t^1_{2\overline{1}}-t^1_{2\overline{1}}(\tau-2\im\varrho+\im\varsigma))\wedge\eta^{\overline{1}}\wedge\eta^2+(\dd t^1_{2\overline{2}}-t^1_{2\overline{2}}(\tau-\im\varrho)-\epsilon\beta_2)\wedge\eta^{\overline{2}}\wedge\eta^2 &\mod\{\eta^0,\eta^1,\eta^3\}.
\end{align*}
Now differentiate $\dd\eta^2$ and reduce modulo $\eta^0,\eta^2,\eta^3$.
\begin{align*}
 0&=\dd(\dd\eta^2)\\
&\equiv (\dd t^2_{1\overline{1}}-t^2_{1\overline{1}}(\tau-\im\varsigma)-\beta_1)\wedge\eta^{\overline{1}}\wedge\eta^1+(\dd t^2_{1\overline{2}}-t^2_{1\overline{2}}(\tau+\im\varrho-2\im\varsigma))\wedge\eta^{\overline{2}}\wedge\eta^1 &\mod\{\eta^0,\eta^2,\eta^3\}.
\end{align*}
The two identities
\begin{align}\label{dts}
 \left.\begin{array}{l}
 \dd t^1_{2\overline{2}}\equiv t^1_{2\overline{2}}(\tau-\im\varrho)+\epsilon\beta_2\\ \\
 \dd t^2_{1\overline{1}}\equiv t^2_{1\overline{1}}(\tau-\im\varsigma)+\beta_1
 \end{array}\right\}\mod\{\eta^0,\eta^1,\eta^2,\eta^3,\eta^{\overline{1}},\eta^{\overline{2}}\}
\end{align}
imply that there is a subbundle $B_3\subset B_2$ of \emph{3-adapted coframes} on which 
\begin{align*}
 t^1_{2\overline{2}}=t^2_{1\overline{1}}=0.
\end{align*}
Observe how \eqref{dts} shows that when restricted to $B_3$, we have 
\begin{align}\label{betamod}
 \beta_1,\beta_2\equiv0\mod\{\eta^0,\eta^1,\eta^2,\eta^3,\eta^{\overline{1}},\eta^{\overline{2}}\}.
\end{align}
We fix a 3-adapted coframing $\theta_\mathbbm{1}$ in order to locally trivialize $B_3$. An explicit parameterization of the structure group $G_3\subset G_2$ of $B_3$ is found by taking $g^{-1}\in C^\infty(B_2,G_2)$ to be the matrix in \eqref{G2} and solving in coordinates the differential equations $\beta_1=0$ and $\beta_2=0$ from the identity
\begin{align*}
 g^{-1}\dd g=\left[\begin{array}{cccc}2\tau&0&0&0\\\gamma^1&\tau+\im\varrho&0&0\\\gamma^2&0&\tau+\im\varsigma&0\\\gamma^3&\im\gamma^2-\beta_1&\im\gamma^1-\beta_2&\im\varrho+\im\varsigma\end{array}\right].
\end{align*}
The result of this calculation is that $G_3$ is comprised of those matrices in $G_2$ which satisfy $b_1=\tfrac{\im}{t}\e^{\im r}c^2$ and $b_2=\tfrac{\im}{t}\e^{\im s}c^1$ so that we locally have $B_3\cong G_3\times M$ where $G_3$ is parameterized by 
\begin{align}\label{G3}
\left[\begin{array}{cccc}t^2&0&0&0\\c^1& t\e^{\im r}&0&0\\c^2&0&t\e^{\im s}&0\\c^3&\tfrac{\im}{t}\e^{\im r}c^2&\tfrac{\im}{t}\e^{\im s}c^1&\e^{\im(r+s)}\end{array}\right];\hspace{1cm} r,s,0\neq t\in\mathbb{R}; c^j\in\mathbb{C}.
\end{align}

If $\iota_3:B_3\hookrightarrow B_2$ is the inclusion map, then we let 
\begin{align*}
&F^1:=\iota_3^*t^1_{2\overline{1}},
&F^2:=\iota_3^*t^2_{1\overline{2}}.
\end{align*}
Aside from this relabelling, we maintain the names of every one-form that we pull back along $\iota_3$, so that the structure equations are the same except that $\beta_1,\beta_2$ are now semibasic. Thus, on $B_3$ we have 
\begin{align}\label{B3SE}
 \dd \left[\begin{array}{c}\eta^0\\\eta^1\\\eta^2\\\eta^3\end{array}\right]&=
 -\left[\begin{array}{cccc}2\tau&0&0&0\\\gamma^1&\tau+\im\varrho&0&0\\\gamma^2&0&\tau+\im\varsigma&0\\\gamma^3&\im\gamma^2&\im\gamma^1&\im\varrho+\im\varsigma\end{array}\right]\wedge
 \left[\begin{array}{c}\eta^0\\\eta^1\\\eta^2\\\eta^3\end{array}\right]
 +
 \left[\begin{array}{c}
 \im\eta^1\wedge\eta^{\overline{1}}+\epsilon\im\eta^2\wedge\eta^{\overline{2}}\\
 \epsilon\eta^3\wedge\eta^{\overline{2}}+F^1\eta^{\overline{1}}\wedge\eta^2 \\
 \eta^3\wedge\eta^{\overline{1}}+F^2\eta^{\overline{2}}\wedge\eta^1 \\
 \beta_1\wedge\eta^1+\beta_2\wedge\eta^2\end{array}\right].
\end{align}

We use \eqref{betamod} to expand $\beta_1$ and $\beta_2$, implicitly using that we can absorb $\eta^0$ coefficients into $\gamma^3$.
\begin{align*}
 &\beta_1=f_{11}\eta^1+t_{1\overline{1}}\eta^{\overline{1}}+f_{12}\eta^2+t_{1\overline{2}}\eta^{\overline{2}}+f_{13}\eta^3,
 &\beta_2=f_{21}\eta^1+t_{2\overline{1}}\eta^{\overline{1}}+f_{22}\eta^2+t_{2\overline{2}}\eta^{\overline{2}}+f_{23}\eta^3,
\end{align*}
for some new functions $f,t\in C^\infty(B_3,\mathbb{C})$. 

We now seek to normalize $t_{1\overline{1}}$ and $t_{2\overline{2}}$ to zero. This will require us to collect a few identities. First differentiate $\dd\eta^0$.
\begin{align*}
0&=\dd(\dd\eta^0)\\
&=(-2\dd\tau+\im\gamma^1\wedge\eta^{\overline{1}}-\im\gamma^{\overline{1}}\wedge\eta^1+\epsilon\im\gamma^2\wedge\eta^{\overline{2}}-\epsilon\im\gamma^{\overline{2}}\wedge\eta^2)\wedge\eta^0,
\end{align*}
whence
\begin{align}\label{B3dtau}
2\dd\tau\equiv\im\gamma^1\wedge\eta^{\overline{1}}-\im\gamma^{\overline{1}}\wedge\eta^1+\epsilon\im\gamma^2\wedge\eta^{\overline{2}}-\epsilon\im\gamma^{\overline{2}}\wedge\eta^2\mod\{\eta^0\}.
\end{align}
Now differentiate $\dd\eta^1$.
\begin{equation}\label{d1}
\begin{aligned}
0&=\dd(\dd\eta^1)\\
&=(-\dd\gamma^1+(\tau-\im\varrho)\wedge\gamma^1-\epsilon\gamma^{\overline{2}}\wedge\eta^3+F^1\gamma^{\overline{1}}\wedge\eta^2-F^1\gamma^2\wedge\eta^{\overline{1}}+\epsilon\gamma^3\wedge\eta^{\overline{2}})\wedge\eta^0\\
&+(-\dd\tau-\im\dd\varrho-\im\gamma^1\wedge\eta^{\overline{1}}+\epsilon\im\gamma^2\wedge\eta^{\overline{2}}+\epsilon\eta^{\overline{3}}\wedge\eta^3+F^1F^2\eta^{\overline{2}}\wedge\eta^{\overline{1}}+|F^1|^2\eta^{\overline{2}}\wedge\eta^2)\wedge\eta^1\\
&+(\dd F^1-F^1(\tau-2\im\varrho+\im\varsigma)+\epsilon\overline{F}^2\eta^3)\wedge\eta^{\overline{1}}\wedge\eta^2+\epsilon\beta_1\wedge\eta^1\wedge\eta^{\overline{2}}+\epsilon\beta_2\wedge\eta^2\wedge\eta^{\overline{2}}.
\end{aligned}
\end{equation}
If we reduce this modulo $\eta^0,\eta^1,\eta^{\overline{1}}$, we see that $f_{23}=0$ in the expansion of $\beta_2$. Furthermore, if we reduce modulo $\eta^1,\eta^2$, then by the top line we conclude 
\begin{align}\label{B3dg1}
 \dd\gamma^1\equiv(\tau-\im\varrho)\wedge\gamma^1-\epsilon\gamma^{\overline{2}}\wedge\eta^3-F^1\gamma^2\wedge\eta^{\overline{1}}+\epsilon\gamma^3\wedge\eta^{\overline{2}}\mod\{\eta^0,\eta^1,\eta^2\}.
\end{align}
Next, differentiate $\dd\eta^2$.
\begin{equation}\label{d2}
\begin{aligned}
0&=\dd(\dd\eta^2)\\
&=(-\dd\gamma^2+(\tau-\im\varsigma)\wedge\gamma^2-\gamma^{\overline{1}}\wedge\eta^3 +F^2\gamma^{\overline{2}}\wedge\eta^1-F^2\gamma^1\wedge\eta^{\overline{2}}+\gamma^3\wedge\eta^{\overline{1}})\wedge\eta^0\\
&+(-\dd\tau-\im\dd\varsigma-\epsilon\im\gamma^2\wedge\eta^{\overline{2}}+\im\gamma^1\wedge\eta^{\overline{1}}+\epsilon\eta^{\overline{3}}\wedge\eta^3+F^2F^1\eta^{\overline{1}}\wedge\eta^{\overline{2}}+|F^2|^2\eta^{\overline{1}}\wedge\eta^1)\wedge\eta^2\\
&+(\dd F^2-F^2(\tau+\im\varrho-2\im\varsigma)+\overline{F}^1\eta^3)\wedge\eta^{\overline{2}}\wedge\eta^1+\beta_1\wedge\eta^1\wedge\eta^{\overline{1}}+\beta_2\wedge\eta^2\wedge\eta^{\overline{1}}.
\end{aligned}
\end{equation}
Reducing modulo $\eta^0,\eta^2,\eta^{\overline{2}}$ shows $f_{13}=0$ in the expansion of $\beta_1$. Reducing mod $\eta^1,\eta^2$ then gives
\begin{align}\label{B3dg2}
 \dd\gamma^2\equiv(\tau-\im\varsigma)\wedge\gamma^2-\gamma^{\overline{1}}\wedge\eta^3 -F^2\gamma^1\wedge\eta^{\overline{2}}+\gamma^3\wedge\eta^{\overline{1}}\mod\{\eta^0,\eta^1,\eta^2\}.
\end{align}
Finally, we differentiate $\dd\eta^3$.
\begin{equation}\label{d3}
\begin{aligned}
0&=\dd(\dd\eta^3)\\
&=-(\dd\gamma^3+\gamma^3\wedge(2\tau-\im\varrho-\im\varsigma)+\gamma^1\wedge\beta_1+\gamma^2\wedge\beta_2)\wedge\eta^0\\
&-\im(\dd\gamma^2+\gamma^2\wedge(\tau-\im\varsigma)-F^2\gamma^1\wedge\eta^{\overline{2}}+\gamma^3\wedge\eta^{\overline{1}})\wedge\eta^1\\
&-\im(\dd\gamma^1+\gamma^1\wedge(\tau-\im\varrho)-F^1\gamma^2\wedge\eta^{\overline{1}}+\epsilon\gamma^3\wedge\eta^{\overline{2}})\wedge\eta^2\\
&+(-\im\dd\varrho-\im\dd\varsigma-\epsilon\im\gamma^2\wedge\eta^{\overline{2}}-\im\gamma^1\wedge\eta^{\overline{1}}+\epsilon\beta_1\wedge\eta^{\overline{2}}+\beta_2\wedge\eta^{\overline{1}})\wedge\eta^3\\
&+(\dd\beta_1-(\tau-\im\varsigma)\wedge\beta_1-F^2\beta_2\wedge\eta^{\overline{2}})\wedge\eta^1+(\dd\beta_2-(\tau-\im\varrho)\wedge\beta_2-F^1\beta_1\wedge\eta^{\overline{1}})\wedge\eta^2.
\end{aligned}
\end{equation}
For later use, we note that by reducing modulo $\eta^0,\eta^1,\eta^2$, we get 
\begin{align}\label{B3idrds}
 \im\dd\varrho+\im\dd\varsigma\equiv -\epsilon\im\gamma^2\wedge\eta^{\overline{2}}-\im\gamma^1\wedge\eta^{\overline{1}}+\epsilon\beta_1\wedge\eta^{\overline{2}}+\beta_2\wedge\eta^{\overline{1}}\mod\{\eta^0,\eta^1,\eta^2,\eta^3\}.
\end{align}
Returning to the unreduced equation \eqref{d3}, if we reduce modulo $\eta^0,\eta^1,\eta^3$, plug in the identity for $\dd\gamma^1$ from \eqref{B3dg1}, and expand $\beta_1$ and $\beta_2$, then we have  
\begin{align*}
0&\equiv
(\dd t_{2\overline{1}}-2t_{2\overline{1}}(\tau-\im\varrho)+2\im F^1\gamma^2)\wedge\eta^{\overline{1}}\wedge\eta^2-F^1t_{1\overline{2}}\eta^{\overline{2}}\wedge\eta^{\overline{1}}\wedge\eta^2\\
&+(\dd t_{2\overline{2}}-t_{2\overline{2}}(2\tau-\im\varrho-\im\varsigma)-\epsilon 2\im\gamma^3)\wedge\eta^{\overline{2}}\wedge\eta^2
&\mod\{\eta^0,\eta^1,\eta^3\}.
\end{align*}
If we instead reduce modulo $\eta^0,\eta^2,\eta^3$ and plug in $\dd\gamma^2$ from \eqref{B3dg2}, we see 
\begin{align*}
0&\equiv
(\dd t_{1\overline{1}}-t_{1\overline{1}}(2\tau-\im\varrho-\im\varsigma)-2\im \gamma^3)\wedge\eta^{\overline{1}}\wedge\eta^1-F^2t_{2\overline{1}}\eta^{\overline{1}}\wedge\eta^{\overline{2}}\wedge\eta^1\\
&+(\dd t_{1\overline{2}}-2t_{1\overline{2}}(\tau-\im\varsigma)+2\im F^2\gamma^1)\wedge\eta^{\overline{2}}\wedge\eta^1&\mod\{\eta^0,\eta^2,\eta^3\}.
\end{align*}
The two together show 
\begin{align}\label{dts2}
\left.\begin{array}{l}
  \dd t_{2\overline{2}}\equiv t_{2\overline{2}}(2\tau-\im\varrho-\im\varsigma)+\epsilon 2\im\gamma^3\\
  \dd t_{1\overline{1}}\equiv t_{1\overline{1}}(2\tau-\im\varrho-\im\varsigma)+2\im \gamma^3
      \end{array}\right\}\mod\{\eta^0,\eta^1,\eta^2,\eta^3,\eta^{\overline{1}},\eta^{\overline{2}}\}.
\end{align}

These imply that we can find a subbundle where one of $t_{1\overline{1}},t_{2\overline{2}}$ vanishes identically, but it is not yet clear that there are any coframings on which both vanish. To show this, we revisit the equations \eqref{d1},\eqref{d2}. For the former, we wedge the right side of the equation with $\eta^2$.
\begin{align*}
0&=(\dd^2\eta^1)\wedge\eta^2\\
&=(-\dd\gamma^1+(\tau-\im\varrho)\wedge\gamma^1-\epsilon\gamma^{\overline{2}}\wedge\eta^3-F^1\gamma^2\wedge\eta^{\overline{1}}+\epsilon\gamma^3\wedge\eta^{\overline{2}})\wedge\eta^0\wedge\eta^2\\
&+(-\dd\tau-\im\dd\varrho-\im\gamma^1\wedge\eta^{\overline{1}}+\epsilon\im\gamma^2\wedge\eta^{\overline{2}}+\epsilon\eta^{\overline{3}}\wedge\eta^3)\wedge\eta^1\wedge\eta^2\\
&+F^1F^2\eta^{\overline{2}}\wedge\eta^{\overline{1}}\wedge\eta^1\wedge\eta^2+\epsilon t_{1\overline{1}}\eta^{\overline{1}}\wedge\eta^1\wedge\eta^{\overline{2}}\wedge\eta^2.
\end{align*}
Similarly, wedge the right side of the identity for $\dd(\dd\eta^2)$ with $\eta^1$. 
\begin{align*}
0&=(\dd^2\eta^2)\wedge\eta^1\\
&=(-\dd\gamma^2+(\tau-\im\varsigma)\wedge\gamma^2-\gamma^{\overline{1}}\wedge\eta^3 -F^2\gamma^1\wedge\eta^{\overline{2}}+\gamma^3\wedge\eta^{\overline{1}})\wedge\eta^0\wedge\eta^1\\
&+(-\dd\tau-\im\dd\varsigma-\epsilon\im\gamma^2\wedge\eta^{\overline{2}}+\im\gamma^1\wedge\eta^{\overline{1}}+\epsilon\eta^{\overline{3}}\wedge\eta^3)\wedge\eta^2\wedge\eta^1\\
&+F^2F^1\eta^{\overline{1}}\wedge\eta^{\overline{2}}\wedge\eta^2\wedge\eta^1+t_{2\overline{2}}\eta^{\overline{2}}\wedge\eta^2\wedge\eta^{\overline{1}}\wedge\eta^1.
\end{align*} 
Now subtract the latter from the former, reduce modulo $\eta^0,\eta^3$, and plug in $2\dd\tau$ and $\im\dd\varrho+\im\dd\varsigma$ from \eqref{B3dtau} and \eqref{B3idrds}.
\begin{align*}
0&=(\dd^2\eta^1)\wedge\eta^2-(\dd^2\eta^2)\wedge\eta^1\\
&\equiv -(2\dd\tau+\im\dd\varrho+\im\dd\varsigma)\wedge\eta^1\wedge\eta^2+(\epsilon t_{1\overline{1}}-t_{2\overline{2}})\eta^{\overline{2}}\wedge\eta^2\wedge\eta^{\overline{1}}\wedge\eta^1&&\mod\{\eta^0,\eta^3\}\\
&\equiv 2(\epsilon t_{1\overline{1}}-t_{2\overline{2}})\eta^{\overline{2}}\wedge\eta^2\wedge\eta^{\overline{1}}\wedge\eta^1&&\mod\{\eta^0,\eta^3\}.
\end{align*}

Thus we see that $\epsilon t_{1\overline{1}}=t_{2\overline{2}}$, and by \eqref{dts2} there exists a subbundle $B_4\subset B_3$ of \emph{4-adapted coframes} on which $t_{1\overline{1}}=t_{2\overline{2}}=0$. We also see from \eqref{dts2} that when restricted to $B_4$,
\begin{align}\label{g3mod}
 \gamma^3\equiv0\mod\{\eta^0,\eta^1,\eta^2,\eta^3,\eta^{\overline{1}},\eta^{\overline{2}}\}.
\end{align}
Fix a new 4-adapted coframing $\theta_{\mathbbm{1}}$ in order to locally trivialize $B_4$. As with $G_3$, we seek a parameterization of the structure group $G_4\subset G_3$ of $B_4$ by taking $g^{-1}\in C^\infty(B_3,G_3)$ to be the matrix \eqref{G3} and solving the differential equation $\gamma^3=0$ in 
\begin{align*}
 g^{-1}\dd g=\left[\begin{array}{cccc}2\tau&0&0&0\\\gamma^1&\tau+\im\varrho&0&0\\\gamma^2&0&\tau+\im\varsigma&0\\\gamma^3&\im\gamma^2&\im\gamma^1&\im\varrho+\im\varsigma\end{array}\right].
\end{align*}
The result is that we locally have $B_4\cong G_4\times M$ where $G_4$ is all matrices of the form 
\begin{align}\label{G4}
\left[\begin{array}{cccc}t^2&0&0&0\\c^1& t\e^{\im r}&0&0\\c^2&0&t\e^{\im s}&0\\\tfrac{\im}{t^2}c^1c^2&\tfrac{\im}{t}\e^{\im r}c^2&\tfrac{\im}{t}\e^{\im s}c^1&\e^{\im(r+s)}\end{array}\right];\hspace{1cm} r,s,0\neq t\in\mathbb{R}; c^1,c^2\in\mathbb{C}.
\end{align}

Pulling back along $\iota_4:B_4\hookrightarrow B_3$, we keep the names of all the forms, and relabel  
\begin{align*}
 &T^3:=\iota_4^*(f_{21}-f_{12}),
 &F^3_1:=\iota_4^*t_{1\overline{2}},
 &&F^3_2:=\iota_4^*t_{2\overline{1}},
\end{align*}
so that the structure equations \eqref{B3SE} pull back to 
\begin{equation}\label{B4SEunabsorbed}
{\footnotesize 
\begin{aligned}
 \dd \left[\begin{array}{c}\eta^0\\\eta^1\\\eta^2\\\eta^3\end{array}\right]&=
 -\left[\begin{array}{cccc}2\tau&0&0&0\\\gamma^1&\tau+\im\varrho&0&0\\\gamma^2&0&\tau+\im\varsigma&0\\0&\im\gamma^2&\im\gamma^1&\im\varrho+\im\varsigma\end{array}\right]\wedge
 \left[\begin{array}{c}\eta^0\\\eta^1\\\eta^2\\\eta^3\end{array}\right]
 +
 \left[\begin{array}{c}
 \im\eta^1\wedge\eta^{\overline{1}}+\epsilon\im\eta^2\wedge\eta^{\overline{2}}\\
 \epsilon\eta^3\wedge\eta^{\overline{2}}+F^1\eta^{\overline{1}}\wedge\eta^2 \\
 \eta^3\wedge\eta^{\overline{1}}+F^2\eta^{\overline{2}}\wedge\eta^1 \\
 -\gamma^3\wedge\eta^0+T^3\eta^1\wedge\eta^2+F^3_1\eta^{\overline{2}}\wedge\eta^1+F^3_2\eta^{\overline{1}}\wedge\eta^2\end{array}\right].
\end{aligned} }
\end{equation}
We absorb the real part of $T^3$ as follows. As in \S\ref{absorption}, we focus only on the relevant two-forms.
\begin{align*}
\dd\eta^3&=-\im\gamma^2\wedge\eta^1-\im\gamma^1\wedge\eta^2-(\im\varrho+\im\varsigma)\wedge\eta^3+T^3\eta^1\wedge\eta^2+\dots\\
&=-\im(\gamma^2-\im\tfrac{1}{2}\text{Re}T^3\eta^2)\wedge\eta^1-\im(\gamma^1+\im\tfrac{1}{2}\text{Re}T^3\eta^1)\wedge\eta^2\\
&-(\im\varrho+\im\tfrac{1}{2}\text{Re}T^3\eta^0+\im\varsigma-\im\tfrac{1}{2}\text{Re}T^3\eta^0)\wedge\eta^3+\im\text{Im}T^3\eta^1\wedge\eta^2+\dots
\end{align*}
so let
\begin{align*}
&\im\hat{\varrho}:=\im\varrho+\im\tfrac{1}{2}\text{Re}T^3\eta^0,
&\im\hat{\varsigma}:=\im\varsigma-\im\tfrac{1}{2}\text{Re}T^3\eta^0,
&&\hat{\gamma}^1:=\gamma^1+\im\tfrac{1}{2}\text{Re}T^3\eta^1, 
&&\hat{\gamma}^2:=\gamma^2-\im\tfrac{1}{2}\text{Re}T^3\eta^2,
\end{align*}
and note that these choices leave the structure equations for $\dd\eta^1,\dd\eta^2$ unaltered. We drop the hats as we prepare to absorb new torsion introduced by the pullback along $\iota_4$ of $\gamma^3$. According to \eqref{g3mod}, we expand
\begin{align*}
 \gamma^3=-f^3_0\eta^0-f^3_1\eta^1-T^3_{\overline{1}}\eta^{\overline{1}}-f^3_2\eta^2-T^3_{\overline{2}}\eta^{\overline{2}}-f^3_3\eta^3,
\end{align*}
for some functions $f,T\in C^\infty(B_4,\mathbb{C})$. We absorb the $f^3_1$ and $f^3_2$ terms via
\begin{align*}
 &\im\hat{\gamma}^2:=\im\gamma^2-f^3_1\eta^0,
 &\im\hat{\gamma}^1:=\im\gamma^1-f^3_2\eta^0.
\end{align*}

Now drop the hats for one final absorption -- the imaginary part of $f^3_3$ -- which will proceed in a similar manner to how we treated the real part of $T^3$ above. Notably, we modify forms so that the equations for $\dd\eta^1,\dd\eta^2$ remain unaffected. We have
\begin{align*}
 \dd\eta^3&=-\im\gamma^2\wedge\eta^1-\im\gamma^1\wedge\eta^2-(\im\varrho+\im\varsigma)\wedge\eta^3+f^3_3\eta^3\wedge\eta^0+\dots \\
 &=-\im(\gamma^2+\im\tfrac{1}{2}\text{Im}(f^3_3)\eta^2)\wedge\eta^1-\im(\gamma^1+\im\tfrac{1}{2}\text{Im}(f^3_3)\eta^1)\wedge\eta^2\\
 &-(\im\varrho+\im\varsigma+\im\text{Im}(f^3_3)\eta^0)\wedge\eta^3+\text{Re}(f^3_3)\eta^3\wedge\eta^0\dots, 
\end{align*}
so we define
\begin{align*}
 &\im\hat{\varrho}:=\im\varrho+\im\tfrac{1}{2}\text{Im}(f^3_3)\eta^0,
 &\im\hat{\varsigma}:=\im\varsigma+\im\tfrac{1}{2}\text{Im}(f^3_3)\eta^0,
 &&\hat{\gamma}^1:=\gamma^1+\im\tfrac{1}{2}\text{Im}(f^3_3)\eta^1,
 &&\hat{\gamma}^2:=\gamma^2+\im\tfrac{1}{2}\text{Im}(f^3_3)\eta^2.
\end{align*}
Let us drop the hats and rename
\begin{align*}
 &f^3:=\text{Re}(f^3_3), &\im t^3:=\im\text{Im}T^3.
\end{align*}

By arranging for these torsion coefficients to be purely real and imaginary, we have exhausted the ambiguity in the pseudoconnection forms $\gamma^1,\gamma^2,\im\varrho,\im\varsigma\in\Omega^1(B_4,\mathbb{C})$ which is associated with Lie-algebra compatible additions of semibasic, $\im\mathbb{R}$-valued forms to $\im\varrho$ and $\im\varsigma$. In particular, $\im\varrho$ and $\im\varsigma$ are now completely and intrinsically determined by our choices of torsion normalization, manifested in the structure equations  

\begin{equation}\label{B4SE}
\begin{aligned}
 \dd \left[\begin{array}{c}\eta^0\\\eta^1\\\eta^2\\\eta^3\end{array}\right]&=
 -\left[\begin{array}{cccc}2\tau&0&0&0\\\gamma^1&\tau+\im\varrho&0&0\\\gamma^2&0&\tau+\im\varsigma&0\\0&\im\gamma^2&\im\gamma^1&\im\varrho+\im\varsigma\end{array}\right]\wedge
 \left[\begin{array}{c}\eta^0\\\eta^1\\\eta^2\\\eta^3\end{array}\right]\\
& +
 \left[\begin{array}{c}
 \im\eta^1\wedge\eta^{\overline{1}}+\epsilon\im\eta^2\wedge\eta^{\overline{2}}\\
 \epsilon\eta^3\wedge\eta^{\overline{2}}+F^1\eta^{\overline{1}}\wedge\eta^2 \\
 \eta^3\wedge\eta^{\overline{1}}+F^2\eta^{\overline{2}}\wedge\eta^1 \\
 f^3\eta^3\wedge\eta^0+\im t^3\eta^1\wedge\eta^2+T^3_{\overline{1}}\eta^{\overline{1}}\wedge\eta^0+T^3_{\overline{2}}\eta^{\overline{2}}\wedge\eta^0+F^3_1\eta^{\overline{2}}\wedge\eta^1+F^3_2\eta^{\overline{1}}\wedge\eta^2\end{array}\right].
\end{aligned} 
\end{equation}

In contrast to $\im\varrho$ and $\im\varsigma$, the pseudoconnection forms $\tau$, $\gamma^1$, and $\gamma^2$ are not uniquely determined by the structure equations \eqref{B4SE}, as they are only determined up to permissible additions of semibasic, $\mathbb{R}$-valued one-forms to $\tau$. Specifically, these structure equations are unaltered if we replace 
\begin{align}\label{prolongambiguity}
 \left[\begin{array}{c}\hat{\tau}\\\hat{\gamma}^1\\\hat{\gamma}^2\end{array}\right]:=
 \left[\begin{array}{c}\tau\\\gamma^1\\\gamma^2\end{array}\right]+
 \left[\begin{array}{ccc}y&0&0\\0&y&0\\0&0&y\end{array}\right]
 \left[\begin{array}{c}\eta^0\\\eta^1\\\eta^2\end{array}\right]; \ y\in C^\infty(B_4,\mathbb{R}).
\end{align}
The new variable $y$ fully parameterizes the remaining ambiguity in our pseudoconnection forms; i.e., adding any other combination of semibasic forms to $\tau,\gamma^1,\gamma^2$ will not preserve the structure equations.

%%%%%%%%%%%%%%%%%%%%%%%%%%%%%%%%%%%%%%%%%%%%%%%%%%%%%%%%%%%%%%%%%%%%%%%%%%%%%%%%%%
\vspace{\baselineskip}\subsection{Prolongation}\label{prolongation}
%%%%%%%%%%%%%%%%%%%%%%%%%%%%%%%%%%%%%%%%%%%%%%%%%%%%%%%%%%%%%%%%%%%%%%%%%%%%%%%%%%

 The collection of all choices \eqref{prolongambiguity} of $\hat{\tau},\hat{\gamma}^1,\hat{\gamma}^2$ preserving \eqref{B4SE} defines an affine, real line bundle $\hat{\pi}:B_4^{(1)}\to B_4$ with $y$ as a fiber coordinate. $B_4^{(1)}$ is the \emph{prolongation} of our $G_4$-structure $\underline{\pi}:B_4\to M$, and may be interpreted as the bundle of coframes on $B_4$ which are adapted to the structure equations, so that we are essentially starting over the method of equivalence. We commit our usual notational abuse of recycling names as we recursively define the following global, tautological one-forms on $B_4^{(1)}$.  
\begin{align}\label{prolongtaut}
 \left[\begin{array}{c}
\eta^0 \\\eta^1 \\\eta^2 \\\eta^3 \\\varrho \\\varsigma \\\tau \\\gamma^1 \\\gamma^2 \end{array}\right]:=
\left[\begin{array}{ccccccccc}
1&0&0&0&0&0&0&0&0\\       
0&1&0&0&0&0&0&0&0\\ 
0&0&1&0&0&0&0&0&0\\ 
0&0&0&1&0&0&0&0&0\\ 
0&0&0&0&1&0&0&0&0\\ 
0&0&0&0&0&1&0&0&0\\ 
y&0&0&0&0&0&1&0&0\\ 
0&y&0&0&0&0&0&1&0\\ 
0&0&y&0&0&0&0&0&1\end{array}\right] 
\hat{\pi}^*\left[\begin{array}{c}
\eta^0\\\eta^1\\\eta^2\\\eta^3\\\varrho\\\varsigma\\\tau\\\gamma^1\\\gamma^2\end{array}\right].
\end{align}

These four $\mathbb{R}$-valued forms, along with the real and imaginary parts of these five $\mathbb{C}$-valued forms, are one real dimension shy of a full, global coframing of $B_4^{(1)}$. As usual, we find the missing one-form by differentiating the tautological forms and normalizing torsion until the resulting pseudoconnection form is uniquely (hence, globally) defined. From \eqref{prolongtaut} we see that if we maintain the names of our torsion coefficients after pulling back along $\hat{\pi}$, the structure equations \eqref{B4SE} still hold on $B_4^{(1)}$:
\begin{equation}\label{prodeta}
\begin{aligned}
 \dd\eta^0&=-2\tau\wedge\eta^0+\im\eta^1\wedge\eta^{\overline{1}}+\epsilon\im\eta^2\wedge\eta^{\overline{2}},\\
 \dd\eta^1&=-\gamma^1\wedge\eta^0-(\tau+\im\varrho)\wedge\eta^1+\epsilon\eta^3\wedge\eta^{\overline{2}}+F^1\eta^{\overline{1}}\wedge\eta^2,\\
 \dd\eta^2&=-\gamma^2\wedge\eta^0-(\tau+\im\varsigma)\wedge\eta^2+\eta^3\wedge\eta^{\overline{1}}+F^2\eta^{\overline{2}}\wedge\eta^1,\\
 \dd\eta^3&=-\im\gamma^2\wedge\eta^1-\im\gamma^1\wedge\eta^2-(\im\varrho+\im\varsigma)\wedge\eta^3+f^3\eta^3\wedge\eta^0+\im t^3\eta^1\wedge\eta^2\\
 &+T^3_{\overline{1}}\eta^{\overline{1}}\wedge\eta^0+T^3_{\overline{2}}\eta^{\overline{2}}\wedge\eta^0+F^3_1\eta^{\overline{2}}\wedge\eta^1+F^3_2\eta^{\overline{1}}\wedge\eta^2.
\end{aligned}
\end{equation}
For the remaining tautological forms, we have in analogy with \eqref{B0SElong},
\begin{align}\label{proSEshort}
 \dd\left[\begin{array}{c}\im\varrho\\\im\varsigma\\\tau\\\gamma^1\\\gamma^2\end{array}\right]=
 -\left[\begin{array}{ccccc}0&0&0&0&0\\0&0&0&0&0\\0&0&\psi&0&0\\0&0&0&\psi&0\\0&0&0&0&\psi\end{array}\right]\wedge \left[\begin{array}{c}0\\0\\\eta^0\\\eta^1\\\eta^2\end{array}\right]
 +
 \left[\begin{array}{c}\Xi^\varrho\\\Xi^\varsigma\\\Xi^\tau\\\Xi^1\\\Xi^2\end{array}\right],
\end{align}
where $\psi\in\Omega^1(B_4^{(1)})$ is our new pseudoconnection form and the $\Xi\in\Omega^2(B_4^{(1)},\mathbb{C})$ are $\hat{\pi}$-semibasic, apparent torsion two-forms. As always, we discover explicit expressions for our $\Xi$'s by differentiating the known structure equations \eqref{prodeta}. Differentiating the equation for $\dd\eta^0$ yields something familiar:  
\begin{align*}
0&=\dd(\dd\eta^0)\\
&=(-2\dd\tau+\im\gamma^1\wedge\eta^{\overline{1}}-\im\gamma^{\overline{1}}\wedge\eta^1+\epsilon\im\gamma^2\wedge\eta^{\overline{2}}-\epsilon\im\gamma^{\overline{2}}\wedge\eta^2)\wedge\eta^0,
\end{align*} 
whence we conclude
\begin{align}\label{dtau}
2\dd\tau= \im\gamma^1\wedge\eta^{\overline{1}}-\im\gamma^{\overline{1}}\wedge\eta^1+\epsilon\im\gamma^2\wedge\eta^{\overline{2}}-\epsilon\im\gamma^{\overline{2}}\wedge\eta^2+2\zeta_0\wedge\eta^0,
\end{align}
for some $\mathbb{R}$-valued $\zeta_0\in\Omega^1(B_4^{(1)})$. Using the equation for $\dd\eta^1$, we find  
\begin{align*}
0&=\dd(\dd\eta^1)\\
&=(-\dd\gamma^1+(\tau-\im\varrho)\wedge\gamma^1-\epsilon\gamma^{\overline{2}}\wedge\eta^3+F^1\gamma^{\overline{1}}\wedge\eta^2-F^1\gamma^2\wedge\eta^{\overline{1}}-\epsilon T^3_{\overline{1}}\eta^{\overline{1}}\wedge\eta^{\overline{2}}-\epsilon f^3\eta^3\wedge\eta^{\overline{2}})\wedge\eta^0\\
&+(-\dd\tau-\im\dd\varrho-\im\gamma^1\wedge\eta^{\overline{1}}+\epsilon\im\gamma^2\wedge\eta^{\overline{2}}+\epsilon\eta^{\overline{3}}\wedge\eta^3+\epsilon\im t^3\eta^2\wedge\eta^{\overline{2}}+F^1F^2\eta^{\overline{2}}\wedge\eta^{\overline{1}}+|F^1|^2\eta^{\overline{2}}\wedge\eta^2)\wedge\eta^1\\
&+(\dd F^1-F^1(\tau-2\im\varrho+\im\varsigma)+\epsilon\overline{F}^2\eta^3+\epsilon F^3_2\eta^{\overline{2}})\wedge\eta^{\overline{1}}\wedge\eta^2,
\end{align*}
which by Cartan's lemma yields 
\begin{equation}\label{dg1idr}
\begin{aligned}
 \left[\begin{array}{c}
-\dd\gamma^1+(\tau-\im\varrho)\wedge\gamma^1-\epsilon\gamma^{\overline{2}}\wedge\eta^3+F^1\gamma^{\overline{1}}\wedge\eta^2-F^1\gamma^2\wedge\eta^{\overline{1}}-\epsilon T^3_{\overline{1}}\eta^{\overline{1}}\wedge\eta^{\overline{2}}-\epsilon f^3\eta^3\wedge\eta^{\overline{2}}\\
-\dd\tau-\im\dd\varrho-\im\gamma^1\wedge\eta^{\overline{1}}+\epsilon\im\gamma^2\wedge\eta^{\overline{2}}+\epsilon\eta^{\overline{3}}\wedge\eta^3+\epsilon\im t^3\eta^2\wedge\eta^{\overline{2}}+F^1F^2\eta^{\overline{2}}\wedge\eta^{\overline{1}}+|F^1|^2\eta^{\overline{2}}\wedge\eta^2\\
(\dd F^1-F^1(\tau-2\im\varrho+\im\varsigma)+\epsilon\overline{F}^2\eta^3+\epsilon F^3_2\eta^{\overline{2}})\wedge\eta^{\overline{1}}\end{array}\right]\\
=
-\left[\begin{array}{ccc}   
 \zeta^1_0&\zeta^1_1&\xi^1_2\\
\zeta^1_1&\zeta^\varrho_1&\xi^\varrho_2\\
\xi^1_2&\xi^\varrho_2&\zeta^1\end{array}\right]
\wedge
\left[\begin{array}{c}\eta^0\\\eta^1\\\eta^2\end{array}\right],
\end{aligned}
\end{equation}
for some $\xi,\zeta\in\Omega^1(B_4^{(1)},\mathbb{C})$. Plugging this back into the same equation $0=\dd(\dd\eta^1)$ reduced by $\eta^{\overline{1}}$ shows
\begin{equation}\label{xi12xir2mod}
\begin{aligned}
 0&\equiv 
\xi^1_2\wedge\eta^2\wedge\eta^0+\xi^\varrho_2\wedge\eta^2\wedge\eta^1&\mod\{\eta^{\overline{1}}\}\\
\Rightarrow 0&\equiv \xi^1_2,\xi^\varrho_2&\mod\{\eta^0,\eta^1,\eta^2,\eta^{\overline{1}}\}.
\end{aligned}
\end{equation}
Moving on to $\dd\eta^2$,
\begin{align*}
0&=\dd(\dd\eta^2)\\
&=(-\dd\gamma^2+(\tau-\im\varsigma)\wedge\gamma^2-\gamma^{\overline{1}}\wedge\eta^3-F^2\gamma^1\wedge\eta^{\overline{2}}+F^2\gamma^{\overline{2}}\wedge\eta^1-T^3_{\overline{2}}\eta^{\overline{2}}\wedge\eta^{\overline{1}}-f^3\eta^3\wedge\eta^{\overline{1}})\wedge\eta^0\\
&+(-\dd\tau-\im\dd\varsigma+\im\gamma^1\wedge\eta^{\overline{1}}-\epsilon\im\gamma^2\wedge\eta^{\overline{2}}+\epsilon\eta^{\overline{3}}\wedge\eta^3-\im t^3\eta^1\wedge\eta^{\overline{1}}+F^2F^1\eta^{\overline{1}}\wedge\eta^{\overline{2}}+|F^2|^2\eta^{\overline{1}}\wedge\eta^1)\wedge\eta^2\\
&+(\dd F^2-F^2(\tau+\im\varrho-2\im\varsigma)+\overline{F}^1\eta^3+F^3_1\eta^{\overline{1}})\wedge\eta^{\overline{2}}\wedge\eta^1.
\end{align*}
By the same argument,
\begin{equation}\label{dg2ids}
\begin{aligned}
 \left[\begin{array}{c}
-\dd\gamma^2+(\tau-\im\varsigma)\wedge\gamma^2-\gamma^{\overline{1}}\wedge\eta^3-F^2\gamma^1\wedge\eta^{\overline{2}}+F^2\gamma^{\overline{2}}\wedge\eta^1-T^3_{\overline{2}}\eta^{\overline{2}}\wedge\eta^{\overline{1}}-f^3\eta^3\wedge\eta^{\overline{1}}\\
-\dd\tau-\im\dd\varsigma+\im\gamma^1\wedge\eta^{\overline{1}}-\epsilon\im\gamma^2\wedge\eta^{\overline{2}}+\epsilon\eta^{\overline{3}}\wedge\eta^3-\im t^3\eta^1\wedge\eta^{\overline{1}}+F^2F^1\eta^{\overline{1}}\wedge\eta^{\overline{2}}+|F^2|^2\eta^{\overline{1}}\wedge\eta^1\\
(\dd F^2-F^2(\tau+\im\varrho-2\im\varsigma)+\overline{F}^1\eta^3+F^3_1\eta^{\overline{1}})\wedge\eta^{\overline{2}}\end{array}\right]\\
=
-\left[\begin{array}{ccc}   
 \zeta^2_0&\zeta^2_2&\xi^2_1\\
\zeta^2_2&\zeta^\varsigma_2&\xi^\varsigma_1\\
\xi^2_1&\xi^\varsigma_1&\zeta^2\end{array}\right]
\wedge
\left[\begin{array}{c}\eta^0\\\eta^2\\\eta^1\end{array}\right],
\end{aligned}
\end{equation}
for more, yet-unknown $\xi,\zeta\in\Omega^1(B_4^{(1)},\mathbb{C})$ which satisfy
\begin{equation}\label{xi21xis1mod}
\begin{aligned}
0&\equiv\xi^2_1\wedge\eta^1\wedge\eta^0+\xi^\varsigma_1\wedge\eta^1\wedge\eta^2&\mod\{\eta^{\overline{2}}\}\\
\Rightarrow 0&\equiv \xi^2_1,\xi^\varsigma_1&\mod\{\eta^0,\eta^1,\eta^2,\eta^{\overline{2}}\}.
\end{aligned}
\end{equation}

From \eqref{dtau},\eqref{dg1idr}, and \eqref{dg2ids} we have gleaned 
\begin{equation}\label{proSExtend}
\begin{aligned}
 \dd\tau&=\tfrac{\im}{2}\gamma^1\wedge\eta^{\overline{1}}-\tfrac{\im}{2}\gamma^{\overline{1}}\wedge\eta^1+\epsilon\tfrac{\im}{2}\gamma^2\wedge\eta^{\overline{2}}-\epsilon\tfrac{\im}{2}\gamma^{\overline{2}}\wedge\eta^2+\zeta_0\wedge\eta^0,  \\
\\
\im\dd\varrho&=-\tfrac{3\im}{2}\gamma^1\wedge\eta^{\overline{1}}+\tfrac{\im}{2}\gamma^{\overline{1}}\wedge\eta^1+\epsilon\tfrac{\im}{2}\gamma^2\wedge\eta^{\overline{2}}+\epsilon\tfrac{\im}{2}\gamma^{\overline{2}}\wedge\eta^2+\epsilon\eta^{\overline{3}}\wedge\eta^3+\epsilon\im t^3\eta^2\wedge\eta^{\overline{2}}\\
&+F^1F^2\eta^{\overline{2}}\wedge\eta^{\overline{1}}+|F^1|^2\eta^{\overline{2}}\wedge\eta^2+(\zeta^1_1-\zeta_0)\wedge\eta^0+\zeta^\varrho_1\wedge\eta^1+\xi^\varrho_2\wedge\eta^2,    \\ 
\\
 \im\dd\varsigma&=\tfrac{\im}{2}\gamma^1\wedge\eta^{\overline{1}}+\tfrac{\im}{2}\gamma^{\overline{1}}\wedge\eta^1-\epsilon\tfrac{3\im}{2}\gamma^2\wedge\eta^{\overline{2}}+\epsilon\tfrac{\im}{2}\gamma^{\overline{2}}\wedge\eta^2+\epsilon\eta^{\overline{3}}\wedge\eta^3 -\im t^3\eta^1\wedge\eta^{\overline{1}}\\
 &+F^2F^1\eta^{\overline{1}}\wedge\eta^{\overline{2}}+|F^2|^2\eta^{\overline{1}}\wedge\eta^1+(\zeta^2_2-\zeta_0)\wedge\eta^0+\zeta^\varsigma_2\wedge\eta^2+\xi^\varsigma_1\wedge\eta^1,   \\
\\
\dd\gamma^1&=(\tau-\im\varrho)\wedge\gamma^1-\epsilon\gamma^{\overline{2}}\wedge\eta^3+F^1\gamma^{\overline{1}}\wedge\eta^2-F^1\gamma^2\wedge\eta^{\overline{1}}-\epsilon T^3_{\overline{1}}\eta^{\overline{1}}\wedge\eta^{\overline{2}}-\epsilon f^3\eta^3\wedge\eta^{\overline{2}} \\
& + \zeta^1_0\wedge\eta^0+\zeta^1_1\wedge\eta^1+\xi^1_2\wedge\eta^2,  \\
\\
 \dd\gamma^2&=(\tau-\im\varsigma)\wedge\gamma^2-\gamma^{\overline{1}}\wedge\eta^3-F^2\gamma^1\wedge\eta^{\overline{2}}+F^2\gamma^{\overline{2}}\wedge\eta^1-T^3_{\overline{2}}\eta^{\overline{2}}\wedge\eta^{\overline{1}}-f^3\eta^3\wedge\eta^{\overline{1}}\\
& +\zeta^2_0\wedge\eta^0+\zeta^2_2\wedge\eta^2+\xi^2_1\wedge\eta^1. 
\end{aligned}
\end{equation}
We learn a bit more about the $\xi$'s and $\zeta$'s by differentiating the final equation from \eqref{prodeta}.
\begin{align*}
0&=\dd(\dd\eta^3)\\
&=\im(-\dd\gamma^2+(\tau-\im\varsigma)\wedge\gamma^2+ F^2\gamma^1\wedge\eta^{\overline{2}}+T^3_{\overline{2}}\eta^{\overline{2}}\wedge\eta^{\overline{1}})\wedge\eta^1\\
&+\im(-\dd\gamma^1+(\tau-\im\varrho)\wedge\gamma^1+ F^1\gamma^2\wedge\eta^{\overline{1}}+\epsilon T^3_{\overline{1}}\eta^{\overline{1}}\wedge\eta^{\overline{2}})\wedge\eta^2\\
&+(-\im\dd\varrho-\im\dd\varsigma-\im\gamma^1\wedge\eta^{\overline{1}}-\epsilon\im\gamma^2\wedge\eta^{\overline{2}}+\epsilon\im (t^3+f^3)\eta^{\overline{2}}\wedge\eta^2+\im (t^3-f^3)\eta^1\wedge\eta^{\overline{1}})\wedge\eta^3\\
&+(F^3_2\gamma^{\overline{1}}+\im (t^3-f^3)\gamma^1+\epsilon T^3_{\overline{1}}\eta^{\overline{3}})\wedge\eta^2\wedge\eta^0\\
&+(F^3_1\gamma^{\overline{2}}-\im (t^3+f^3)\gamma^2+T^3_{\overline{2}}\eta^{\overline{3}})\wedge\eta^1\wedge\eta^0\\
&+(\dd T^3_{\overline{1}}-T^3_{\overline{1}}(3\tau-2\im\varrho-\im\varsigma)-F^3_2\gamma^2+(T^3_{\overline{2}}\overline{F}^2-f^3F^3_2)\eta^{2})\wedge\eta^{\overline{1}}\wedge\eta^0\\
&+(\dd T^3_{\overline{2}}-T^3_{\overline{2}}(3\tau-\im\varrho-2\im\varsigma)-F^3_1\gamma^1+(T^3_{\overline{1}}\overline{F}^1-f^3F^3_1)\eta^{1})\wedge\eta^{\overline{2}}\wedge\eta^0\\
&+(\dd F^3_1-2F^3_1(\tau-\im\varsigma)-F^3_2\overline{F}^1\eta^2-F^3_2F^2\eta^{\overline{1}})\wedge\eta^{\overline{2}}\wedge\eta^1\\
&+(\dd F^3_2-2F^3_2(\tau-\im\varrho)-F^3_1\overline{F}^2\eta^1-F^3_1F^1\eta^{\overline{2}})\wedge\eta^{\overline{1}}\wedge\eta^2 \\
&+\im(\dd t^3- 2t^3\tau+ f^3t^3\eta^0)\wedge\eta^1\wedge\eta^2+(\dd f^3-2f^3\tau)\wedge\eta^3\wedge\eta^0.
\end{align*}
 After plugging in \eqref{proSExtend}, this becomes 
\begin{equation}\label{d2eta3xiszetas}
\begin{aligned}
0&=(\zeta^\varrho_1+\xi^\varsigma_1+2\im\gamma^{\overline{1}}+(|F^2|^2+2\im(t^3-f^3))\eta^{\overline{1}})\wedge\eta^3\wedge\eta^1 \\
&+(\xi^\varrho_2+\zeta^\varsigma_2+\epsilon2\im\gamma^{\overline{2}}+(|F^1|^2-\epsilon2\im(t^3+f^3))\eta^{\overline{2}})\wedge\eta^3\wedge\eta^2 \\
&+(\im\zeta^1_0+F^3_2\gamma^{\overline{1}}+\im (t^3-f^3)\gamma^1+\epsilon T^3_{\overline{1}}\eta^{\overline{3}})\wedge\eta^2\wedge\eta^0\\
&+(\im\zeta^2_0+F^3_1\gamma^{\overline{2}}-\im (t^3+f^3)\gamma^2+T^3_{\overline{2}}\eta^{\overline{3}})\wedge\eta^1\wedge\eta^0\\
&+(\dd T^3_{\overline{1}}-T^3_{\overline{1}}(3\tau-2\im\varrho-\im\varsigma)-F^3_2\gamma^2+(T^3_{\overline{2}}\overline{F}^2-f^3F^3_2)\eta^{2})\wedge\eta^{\overline{1}}\wedge\eta^0\\
&+(\dd T^3_{\overline{2}}-T^3_{\overline{2}}(3\tau-\im\varrho-2\im\varsigma)-F^3_1\gamma^1+(T^3_{\overline{1}}\overline{F}^1-f^3F^3_1)\eta^{1})\wedge\eta^{\overline{2}}\wedge\eta^0\\
&+(\dd F^3_1-2F^3_1(\tau-\im\varsigma)+ 2\im F^2\gamma^1-F^3_2\overline{F}^1\eta^2-(F^3_2F^2+2\im T^3_{\overline{2}})\eta^{\overline{1}})\wedge\eta^{\overline{2}}\wedge\eta^1\\
&+(\dd F^3_2-2F^3_2(\tau-\im\varrho)+2\im F^1\gamma^2-F^3_1\overline{F}^2\eta^1-(F^3_1F^1+ \epsilon2\im T^3_{\overline{1}})\eta^{\overline{2}})\wedge\eta^{\overline{1}}\wedge\eta^2 \\
&+\im(\dd t^3- 2t^3\tau+ f^3t^3\eta^0-\zeta^1_1+\zeta^2_2)\wedge\eta^1\wedge\eta^2+(\dd f^3-2f^3\tau+\zeta^1_1+\zeta^2_2-2\zeta_0)\wedge\eta^3\wedge\eta^0.
\end{aligned}
\end{equation}
For later use, we observe that if we reduce by $\{\eta^0,\eta^3,\eta^{\overline{1}},\eta^{\overline{2}}\}$ or $\{\eta^1,\eta^2,\eta^{\overline{1}},\eta^{\overline{2}}\}$, respectively, then we can say 
\begin{align}\label{dtdf}
 \left.\begin{array}{l}
   0\equiv \dd t^3- 2t^3\tau+ f^3t^3\eta^0-\zeta^1_1+\zeta^2_2\\    
   0\equiv \dd f^3-2f^3\tau+\zeta^1_1+\zeta^2_2-2\zeta_0\end{array}\right\}\mod\{\eta^0,\eta^1,\eta^2,\eta^3,\eta^{\overline{1}},\eta^{\overline{2}}\}.
\end{align}

Now we return to the unreduced equation \eqref{d2eta3xiszetas}. With \eqref{xi12xir2mod} and \eqref{xi21xis1mod} in mind, we see that reduction modulo $\{\eta^0,\eta^1,\eta^{\overline{1}}\}$, $\{\eta^0,\eta^2,\eta^{\overline{2}}\}$, $\{\eta^1,\eta^3,\eta^{\overline{1}},\eta^{\overline{2}}\}$, $\{\eta^2,\eta^3,\eta^{\overline{1}},\eta^{\overline{2}}\}$, respectively yields
\begin{equation}\label{zetasmod}
\begin{aligned}
 \zeta^\varsigma_2&\equiv -\epsilon 2\im\gamma^{\overline{2}}-(|F^1|^2-\epsilon 2\im (t^3+f^3))\eta^{\overline{2}} &\mod\{\eta^0,\eta^1,\eta^{\overline{1}},\eta^2,\eta^3\},\\
 \zeta^\varrho_1&\equiv -2\im\gamma^{\overline{1}}-(|F^2|^2+2\im(t^3-f^3))\eta^{\overline{1}}  &\mod\{\eta^0,\eta^2,\eta^{\overline{2}},\eta^1,\eta^3\},\\
 \zeta^1_0&\equiv \im F^3_2\gamma^{\overline{1}}- (t^3-f^3)\gamma^1+\epsilon\im T^3_{\overline{1}}\eta^{\overline{3}} &\mod\{\eta^1,\eta^3,\eta^{\overline{1}},\eta^{\overline{2}},\eta^2,\eta^0\},\\
 \zeta^2_0&\equiv \im F^3_1\gamma^{\overline{2}}+ (t^3+f^3)\gamma^2+\im T^3_{\overline{2}}\eta^{\overline{3}} &\mod\{\eta^2,\eta^3,\eta^{\overline{1}},\eta^{\overline{2}},\eta^1,\eta^0\}.
\end{aligned}
\end{equation}
Thus, if we define 
\begin{align*}
 \xi^\varsigma_2&:=\zeta^\varsigma_2+\epsilon 2\im\gamma^{\overline{2}}+(|F^1|^2-\epsilon 2\im (t^3+f^3))\eta^{\overline{2}},\\
 \xi^\varrho_1&:=\zeta^\varrho_1+2\im\gamma^{\overline{1}}+(|F^2|^2+2\im(t^3-f^3))\eta^{\overline{1}},\\
 \xi^1_0&:=\zeta^1_0-\im F^3_2\gamma^{\overline{1}}+ (t^3-f^3)\gamma^1-\epsilon\im T^3_{\overline{1}}\eta^{\overline{3}},\\ 
 \xi^2_0&:=\zeta^2_0-\im F^3_1\gamma^{\overline{2}}- (t^3+f^3)\gamma^2-\im T^3_{\overline{2}}\eta^{\overline{3}},\\
 \xi_0&:=\zeta_0+\psi,\\
 \xi^1_1&:=\zeta^1_1+\psi,\\
 \xi^2_2&:=\zeta^2_2+\psi,
\end{align*}
then we are left with an expression for each of the $\Xi$'s in the structure equations \eqref{proSEshort} of $B_4^{(1)}$:
\begin{equation}\label{Xixi}
\begin{aligned}
\Xi^\tau&=\tfrac{\im}{2}\gamma^1\wedge\eta^{\overline{1}}-\tfrac{\im}{2}\gamma^{\overline{1}}\wedge\eta^1+\epsilon \tfrac{\im}{2}\gamma^2\wedge\eta^{\overline{2}}-\epsilon \tfrac{\im}{2}\gamma^{\overline{2}}\wedge\eta^2+\xi_0\wedge\eta^0,\\
\\
\Xi^\varrho&=-\tfrac{3\im}{2}\gamma^1\wedge\eta^{\overline{1}}-\tfrac{3\im}{2}\gamma^{\overline{1}}\wedge\eta^1+\epsilon\tfrac{\im}{2}\gamma^2\wedge\eta^{\overline{2}}+\epsilon\tfrac{\im}{2}\gamma^{\overline{2}}\wedge\eta^2+\epsilon \eta^{\overline{3}}\wedge\eta^3+F^1F^2\eta^{\overline{2}}\wedge\eta^{\overline{1}}\\
&+(|F^1|^2-\epsilon \im t^3)\eta^{\overline{2}}\wedge\eta^2-(|F^2|^2+2\im(t^3-f^3))\eta^{\overline{1}}\wedge\eta^1+(\xi^1_1-\xi_0)\wedge\eta^0+\xi^\varrho_1\wedge\eta^1+\xi^\varrho_2\wedge\eta^2,\\
\\
\Xi^\varsigma&=\tfrac{\im}{2}\gamma^1\wedge\eta^{\overline{1}}+\tfrac{\im}{2}\gamma^{\overline{1}}\wedge\eta^1-\epsilon \tfrac{3\im}{2}\gamma^2\wedge\eta^{\overline{2}}-\epsilon \tfrac{3\im}{2}\gamma^{\overline{2}}\wedge\eta^2+\epsilon\eta^{\overline{3}}\wedge\eta^3+F^2F^1\eta^{\overline{1}}\wedge\eta^{\overline{2}}\\
&+(|F^2|^2+\im t^3)\eta^{\overline{1}}\wedge\eta^1-(|F^1|^2-\epsilon 2\im(f^3+t^3))\eta^{\overline{2}}\wedge\eta^2+(\xi^2_2-\xi_0)\wedge\eta^0+\xi^\varsigma_1\wedge\eta^1+\xi^\varsigma_2\wedge\eta^2,\\
\\
\Xi^1&=(\tau-\im\varrho)\wedge\gamma^1-\epsilon\gamma^{\overline{2}}\wedge\eta^3+\im F^3_2\gamma^{\overline{1}}\wedge\eta^0+F^1\gamma^{\overline{1}}\wedge\eta^2-F^1\gamma^2\wedge\eta^{\overline{1}}+ (t^3-f^3)\gamma^1\wedge\eta^0  \\
 &-\epsilon f^3\eta^3\wedge\eta^{\overline{2}} -\epsilon T^3_{\overline{1}}\eta^{\overline{1}}\wedge\eta^{\overline{2}} +\epsilon\im T^3_{\overline{1}}\eta^{\overline{3}}\wedge\eta^0+ \xi^1_0\wedge\eta^0+\xi^1_1\wedge\eta^1+\xi^1_2\wedge\eta^2,  \\
\\
\Xi^2&=(\tau-\im\varsigma)\wedge\gamma^2-\gamma^{\overline{1}}\wedge\eta^3+\im F^3_1\gamma^{\overline{2}}\wedge\eta^0-F^2\gamma^1\wedge\eta^{\overline{2}}+F^2\gamma^{\overline{2}}\wedge\eta^1+ (t^3+f^3)\gamma^2\wedge\eta^0\\
& -f^3\eta^3\wedge\eta^{\overline{1}}-T^3_{\overline{2}}\eta^{\overline{2}}\wedge\eta^{\overline{1}}+\im T^3_{\overline{2}}\eta^{\overline{3}}\wedge\eta^0+\xi^2_0\wedge\eta^0+\xi^2_2\wedge\eta^2+\xi^2_1\wedge\eta^1, 
\end{aligned}
\end{equation}
where, by \eqref{xi12xir2mod},\eqref{xi21xis1mod}, and \eqref{zetasmod}, we now have 
\begin{align}\label{xismod3}
 0\equiv\left\{\begin{array}{ll}
       \xi^1_2, \xi^\varrho_2 & \mod\{\eta^0,\eta^1,\eta^2,\eta^{\overline{1}}\}, \\
       \xi^2_1,\xi^\varsigma_1&\mod\{\eta^0,\eta^1,\eta^2,\eta^{\overline{2}}\}, \\
       \xi^\varrho_1  & \mod\{\eta^0,\eta^2,\eta^{\overline{2}},\eta^1,\eta^3\},\\
      \xi^\varsigma_2 & \mod\{\eta^0,\eta^1,\eta^{\overline{1}},\eta^2,\eta^3\},\\
      \xi^1_0,\xi^2_0 & \mod\{\eta^0,\eta^1,\eta^{\overline{1}},\eta^2,\eta^{\overline{2}},\eta^3\}.\end{array}\right.
\end{align}
Using the fact that $\im\dd\varrho$ is $\im\mathbb{R}$-valued, we can write 
\begin{equation}\label{Xirimag}
\begin{aligned}
 0&=\Xi^\varrho+\overline{\Xi}^\varrho\\
&=(\xi^1_1+\xi^{\overline{1}}_{\overline{1}}-2\xi_0)\wedge\eta^0+\xi^\varrho_1\wedge\eta^1+\xi^{\overline{\varrho}}_{\overline{1}}\wedge\eta^{\overline{1}}+\epsilon2\im t^3\eta^{2}\wedge\eta^{\overline{2}}+4\im(t^3-f^3)\eta^{1}\wedge\eta^{\overline{1}}\\
&+(\xi^\varrho_2-\overline{F}^1\overline{F}^2\eta^{1})\wedge\eta^2+(\xi^{\overline{\varrho}}_{\overline{2}}-F^1F^2\eta^{\overline{1}})\wedge\eta^{\overline{2}},
\end{aligned}
\end{equation}
which along with \eqref{xi12xir2mod} shows that $t^3=f^3=0$. Plugging these zeros into \eqref{dtdf} yields 
\begin{align*}
 \left.\begin{array}{l}
   0\equiv -\xi^1_1+\xi^2_2\\    
   0\equiv \xi^1_1+\xi^2_2-2\xi_0\end{array}\right\}\mod\{\eta^0,\eta^1,\eta^2,\eta^3,\eta^{\overline{1}},\eta^{\overline{2}}\},
\end{align*}
so in particular,
\begin{align}\label{xi1xi2mod}
\left.\begin{array}{l}
  \xi^1:=\xi^1_1-\xi_0\\    
  \xi^2:=\xi^2_2-\xi_0\end{array}\right\}\equiv0\mod\{\eta^0,\eta^1,\eta^2,\eta^3,\eta^{\overline{1}},\eta^{\overline{2}}\}.
\end{align}

We know that $\xi_0$ is $\mathbb{R}$-valued, so we can replace $\psi$ with $\hat{\psi}=\psi-\xi_0$, which has the effect of removing the $\xi_0$ term in the equation for $\dd\tau$ and replacing $\xi^i_i$ with $\xi^i:=\xi^i_i-\xi_0$ ($i=1,2$) in the equation for $\dd\gamma^i$. We therefore update our structure equations
\begin{equation}\label{proSExis}
\begin{aligned}
\dd\tau&=-\hat{\psi}\wedge\eta^0+\tfrac{\im}{2}\gamma^1\wedge\eta^{\overline{1}}-\tfrac{\im}{2}\gamma^{\overline{1}}\wedge\eta^1+\epsilon\tfrac{\im}{2}\gamma^2\wedge\eta^{\overline{2}}-\epsilon\tfrac{\im}{2}\gamma^{\overline{2}}\wedge\eta^2,\\
\\
\im\dd\varrho&=-\tfrac{3\im}{2}\gamma^1\wedge\eta^{\overline{1}}-\tfrac{3\im}{2}\gamma^{\overline{1}}\wedge\eta^1+\epsilon\tfrac{\im}{2}\gamma^2\wedge\eta^{\overline{2}}+\epsilon\tfrac{\im}{2}\gamma^{\overline{2}}\wedge\eta^2+\epsilon\eta^{\overline{3}}\wedge\eta^3+F^1F^2\eta^{\overline{2}}\wedge\eta^{\overline{1}}\\
&+|F^1|^2\eta^{\overline{2}}\wedge\eta^2-|F^2|^2\eta^{\overline{1}}\wedge\eta^1+\xi^1\wedge\eta^0+\xi^\varrho_1\wedge\eta^1+\xi^\varrho_2\wedge\eta^2,\\
\\
\im\dd\varsigma&=\tfrac{\im}{2}\gamma^1\wedge\eta^{\overline{1}}+\tfrac{\im}{2}\gamma^{\overline{1}}\wedge\eta^1-\epsilon\tfrac{3\im}{2}\gamma^2\wedge\eta^{\overline{2}}-\epsilon\tfrac{3\im}{2}\gamma^{\overline{2}}\wedge\eta^2+\epsilon\eta^{\overline{3}}\wedge\eta^3+F^2F^1\eta^{\overline{1}}\wedge\eta^{\overline{2}}\\
&+|F^2|^2\eta^{\overline{1}}\wedge\eta^1-|F^1|^2\eta^{\overline{2}}\wedge\eta^2+\xi^2\wedge\eta^0+\xi^\varsigma_1\wedge\eta^1+\xi^\varsigma_2\wedge\eta^2,\\
\\
\dd\gamma^1&=-\hat{\psi}\wedge\eta^1+(\tau-\im\varrho)\wedge\gamma^1-\epsilon\gamma^{\overline{2}}\wedge\eta^3+\im F^3_2\gamma^{\overline{1}}\wedge\eta^0+F^1\gamma^{\overline{1}}\wedge\eta^2-F^1\gamma^2\wedge\eta^{\overline{1}}  \\
 & -\epsilon T^3_{\overline{1}}\eta^{\overline{1}}\wedge\eta^{\overline{2}} +\epsilon\im T^3_{\overline{1}}\eta^{\overline{3}}\wedge\eta^0+ \xi^1_0\wedge\eta^0+\xi^1\wedge\eta^1+\xi^1_2\wedge\eta^2,  \\
\\
\dd\gamma^2&=-\hat{\psi}\wedge\eta^2+(\tau-\im\varsigma)\wedge\gamma^2-\gamma^{\overline{1}}\wedge\eta^3+\im F^3_1\gamma^{\overline{2}}\wedge\eta^0-F^2\gamma^1\wedge\eta^{\overline{2}}+F^2\gamma^{\overline{2}}\wedge\eta^1\\
&-T^3_{\overline{2}}\eta^{\overline{2}}\wedge\eta^{\overline{1}}+\im T^3_{\overline{2}}\eta^{\overline{3}}\wedge\eta^0+\xi^2_0\wedge\eta^0+\xi^2_1\wedge\eta^1+\xi^2\wedge\eta^2,
\end{aligned}
\end{equation}
where by \eqref{xismod3} and \eqref{xi1xi2mod} we can say
\begin{align*}
 0\equiv\left\{\begin{array}{ll}
       \xi^1_2, \xi^\varrho_2 & \mod\{\eta^0,\eta^1,\eta^2,\eta^{\overline{1}}\}, \\
       \xi^2_1,\xi^\varsigma_1&\mod\{\eta^0,\eta^1,\eta^2,\eta^{\overline{2}}\}, \\
       \xi^\varrho_1  & \mod\{\eta^0,\eta^2,\eta^{\overline{2}},\eta^1,\eta^3\},\\
      \xi^\varsigma_2 & \mod\{\eta^0,\eta^1,\eta^{\overline{1}},\eta^2,\eta^3\},\\
      \xi^1_0,\xi^2_0,\xi^1,\xi^2 & \mod\{\eta^0,\eta^1,\eta^{\overline{1}},\eta^2,\eta^{\overline{2}},\eta^3\}.\end{array}\right.
\end{align*}
By collecting coefficients of redundant two-forms and suppressing forms which are only wedged against themselves in all of the equations, we may more specifically assume
\begin{align*}
& 0\equiv\left\{\begin{array}{ll}
       \xi^1_2,\xi^\varrho_2 & \mod\{\eta^0,\eta^1,\eta^{\overline{1}}\},\\
       \xi^2_1,\xi^\varsigma_1&\mod\{\eta^0,\eta^2,\eta^{\overline{2}}\} ,\\
       \xi^\varrho_1  & \mod\{\eta^0,\eta^{\overline{2}},\eta^3\},\\
       \xi^\varsigma_2 & \mod\{\eta^0,\eta^{\overline{1}},\eta^3\}\\
       \xi^1_0 & \mod\{\eta^1,\eta^{\overline{1}},\eta^{\overline{2}},\eta^3\},\\
      \xi^2_0 & \mod\{\eta^{\overline{1}},\eta^2,\eta^{\overline{2}},\eta^3\},\\
      \xi^1,\xi^2 & \mod\{\eta^{\overline{1}},\eta^{\overline{2}},\eta^3\}.\end{array}\right.
\end{align*}
Let us therefore expand 
\begin{align*}
\begin{array}{l}
\xi^1_{0} = P^1_{01}\eta^1+P^1_{0\overline{1}}\eta^{\overline{1}}+P^1_{0\overline{2}}\eta^{\overline{2}}+P^1_{03}\eta^3,\\
\xi^1_{2} = P^1_{20}\eta^0+P^1_{21}\eta^1+P^1_{2\overline{1}}\eta^{\overline{1}},\\
\xi^1 = Q^1_{\overline{1}}\eta^{\overline{1}}+Q^1_{\overline{2}}\eta^{\overline{2}}+Q^1_{3}\eta^3,\\
\xi^\varrho_{1}  = R_{10}\eta^0+R_{1\overline{2}}\eta^{\overline{2}}+R_{13}\eta^3,\\
\xi^\varsigma_{1}=S_{10}\eta^0+S_{12}\eta^2+S_{1\overline{2}}\eta^{\overline{2}},
\end{array}
&&\begin{array}{l}
\xi^2_{0} = P^2_{0\overline{1}}\eta^{\overline{1}}+P^2_{02}\eta^2+P^2_{0\overline{2}}\eta^{\overline{2}}+P^2_{03}\eta^3,\\
\xi^2_{1}=P^2_{10}\eta^0+P^2_{12}\eta^2+P^2_{1\overline{2}}\eta^{\overline{2}},\\
\xi^2 = Q^2_{\overline{1}}\eta^{\overline{1}}+Q^2_{\overline{2}}\eta^{\overline{2}}+Q^2_{3}\eta^3,\\
\xi^\varrho_{2} = R_{20}\eta^0+R_{21}\eta^1+R_{2\overline{1}}\eta^{\overline{1}},\\
\xi^\varsigma_{2} = S_{20}\eta^0+S_{2\overline{1}}\eta^{\overline{1}}+S_{23}\eta^3,
\end{array}
\end{align*}
for some functions $P,Q,R,S\in C^\infty(B_4^{(1)},\mathbb{C})$. With these in hand, we return to our argument about the imaginary value of $\im\dd\varrho$ from \eqref{Xirimag}.
\begin{align*}
 0&=\Xi^\varrho+\overline{\Xi}^\varrho\\
&=(\xi^1+\xi^{\overline{1}})\wedge\eta^0+\xi^\varrho_1\wedge\eta^1+\xi^{\overline{\varrho}}_{\overline{1}}\wedge\eta^{\overline{1}}+(\xi^\varrho_2-\overline{F}^1\overline{F}^2\eta^{1})\wedge\eta^2+(\xi^{\overline{\varrho}}_{\overline{2}}-F^1F^2\eta^{\overline{1}})\wedge\eta^{\overline{2}}\\
&=(Q^1_{\overline{1}}\eta^{\overline{1}}+Q^1_{\overline{2}}\eta^{\overline{2}}+Q^1_{3}\eta^3+\overline{Q}^1_{\overline{1}}\eta^{1}+\overline{Q}^1_{\overline{2}}\eta^{2}+\overline{Q}^1_{3}\eta^{\overline{3}})\wedge\eta^0+(R_{10}\eta^0+R_{1\overline{2}}\eta^{\overline{2}}+R_{13}\eta^3)\wedge\eta^1\\
&+(\overline{R}_{10}\eta^0+\overline{R}_{1\overline{2}}\eta^{2}+\overline{R}_{13}\eta^{\overline{3}})\wedge\eta^{\overline{1}}+(R_{20}\eta^0+R_{21}\eta^1+R_{2\overline{1}}\eta^{\overline{1}}-\overline{F}^1\overline{F}^2\eta^{1})\wedge\eta^2\\
&+(\overline{R}_{20}\eta^0+\overline{R}_{21}\eta^{\overline{1}}+\overline{R}_{2\overline{1}}\eta^{1}-F^1F^2\eta^{\overline{1}})\wedge\eta^{\overline{2}}.
\end{align*}
Thus we see that 
\begin{align*}
&Q^1_{3}=R_{13}=0,
&R_{21}=\overline{F}^1\overline{F}^2,
&&R_{10}=\overline{Q}^1_{\overline{1}},
&&R_{20}=\overline{Q}^1_{\overline{2}},
&&&R_{1\overline{2}}=\overline{R}_{2\overline{1}}.
\end{align*}
Similarly, $\im\varsigma$ is $\im\mathbb{R}$-valued, and we have 
\begin{align*}
 0&=\Xi^\varsigma+\overline{\Xi}^\varsigma\\
&=(\xi^2+\xi^{\overline{2}})\wedge\eta^0+(\xi^\varsigma_1-\overline{F}^2\overline{F}^1\eta^{2})\wedge\eta^1+(\xi^{\overline{\varsigma}}_{\overline{1}}-F^2F^1\eta^{\overline{2}})\wedge\eta^{\overline{1}}+\xi^\varsigma_2\wedge\eta^2+\xi^{\overline{\varsigma}}_{\overline{2}}\wedge\eta^{\overline{2}}\\
&=(Q^2_{\overline{1}}\eta^{\overline{1}}+Q^2_{\overline{2}}\eta^{\overline{2}}+Q^2_{3}\eta^3+\overline{Q}^2_{\overline{1}}\eta^{1}+\overline{Q}^2_{\overline{2}}\eta^{2}+\overline{Q}^2_{3}\eta^{\overline{3}})\wedge\eta^0+(S_{10}\eta^0+S_{12}\eta^2+S_{1\overline{2}}\eta^{\overline{2}}-\overline{F}^2\overline{F}^1\eta^{2})\wedge\eta^1\\
&+(\overline{S}_{10}\eta^0+\overline{S}_{12}\eta^{\overline{2}}+\overline{S}_{1\overline{2}}\eta^{2}-F^2F^1\eta^{\overline{2}})\wedge\eta^{\overline{1}}+(S_{20}\eta^0+S_{2\overline{1}}\eta^{\overline{1}}+S_{23}\eta^3)\wedge\eta^2\\
&+(\overline{S}_{20}\eta^0+\overline{S}_{2\overline{1}}\eta^{1}+\overline{S}_{23}\eta^{\overline{3}})\wedge\eta^{\overline{2}},
\end{align*}
whence 
\begin{align*}
&Q^2_{3}=S_{23}=0,
&S_{12}=\overline{F}^1\overline{F}^2,
&&S_{10}=\overline{Q}^2_{\overline{1}},
&&S_{20}=\overline{Q}^2_{\overline{2}},
&&&S_{1\overline{2}}=\overline{S}_{2\overline{1}}.
\end{align*}
We reveal a few more relations by revisiting our original structure equations.
\begin{align*}
 0&\equiv \dd^2\eta^1&\mod\{\eta^{\overline{1}}\}\\
&\equiv\xi^1_2\wedge\eta^2\wedge\eta^0+\xi^\varrho_2\wedge\eta^2\wedge\eta^1&\mod\{\eta^{\overline{1}}\}\\
&\equiv P^1_{21}\eta^1\wedge\eta^2\wedge\eta^0+\overline{Q}^1_{\overline{2}}\eta^0\wedge\eta^2\wedge\eta^1&\mod\{\eta^{\overline{1}}\}\\
\\
\Rightarrow & P^1_{21}=\overline{Q}^1_{\overline{2}}.
\end{align*}
Similarly, 
\begin{align*}
0&\equiv\dd^2\eta^2&\mod\{\eta^{\overline{2}}\}\\
&\equiv\xi^2_1\wedge\eta^1\wedge\eta^0+\xi^\varsigma_1\wedge\eta^1\wedge\eta^2&\mod\{\eta^{\overline{2}}\}\\
&\equiv P^2_{12}\eta^2\wedge\eta^1\wedge\eta^0+\overline{Q}^2_{\overline{1}}\eta^0\wedge\eta^1\wedge\eta^2&\mod\{\eta^{\overline{2}}\}\\
\\
\Rightarrow & P^2_{12}=\overline{Q}^2_{\overline{1}}.
\end{align*}
And finally,
\begin{align*}
0&\equiv \dd^2\eta^3 &\mod\{\eta^{\overline{1}},\eta^{\overline{2}}\}\\
&\equiv(\xi^\varrho_1+\xi^\varsigma_1)\wedge\eta^3\wedge\eta^1 +(\xi^\varrho_2+\xi^\varsigma_2)\wedge\eta^3\wedge\eta^2 +\im\xi^1_0\wedge\eta^2\wedge\eta^0+\im\xi^2_0\wedge\eta^1\wedge\eta^0\\
&+\im(-\xi^1+\xi^2)\wedge\eta^1\wedge\eta^2+(\xi^1+\xi^2)\wedge\eta^3\wedge\eta^0&\mod\{\eta^{\overline{1}},\eta^{\overline{2}}\}\\
&\equiv(\overline{Q}^1_{\overline{1}}+\overline{Q}^2_{\overline{1}})\eta^0\wedge\eta^3\wedge\eta^1 +(\overline{Q}^1_{\overline{2}}+\overline{Q}^2_{\overline{2}})\eta^0\wedge\eta^3\wedge\eta^2\\
&+\im(P^1_{01}\eta^1+P^1_{03}\eta^3)\wedge\eta^2\wedge\eta^0+\im(P^2_{02}\eta^2+P^2_{03}\eta^3)\wedge\eta^1\wedge\eta^0&\mod\{\eta^{\overline{1}},\eta^{\overline{2}}\}\\
\\
\Rightarrow &P^1_{03}=\im(\overline{Q}^1_{\overline{2}}+\overline{Q}^2_{\overline{2}}),\hspace{0.5cm} P^2_{03}=\im(\overline{Q}^1_{\overline{1}}+\overline{Q}^2_{\overline{1}}),\hspace{0.5cm} P^1_{01}=P^2_{02}.
\end{align*}
We give preference to the $Q$'s in our notation, so we can rename the only remaining $R:=R_{1\overline{2}}$ and $S:=S_{1\overline{2}}$. We also rename $P_0:=P^1_{01}=P^2_{02}$ to emphasize that the equations for $\dd\gamma^1$ and $\dd\gamma^2$ have this term in common. Dropping the hat off of $\psi$ in \eqref{proSExis}, we summarize our results so far.
\begin{equation}\label{proSElong}
\begin{aligned}
\dd\tau&=-\psi\wedge\eta^0+\tfrac{\im}{2}\gamma^1\wedge\eta^{\overline{1}}-\tfrac{\im}{2}\gamma^{\overline{1}}\wedge\eta^1+\epsilon\tfrac{\im}{2}\gamma^2\wedge\eta^{\overline{2}}-\epsilon\tfrac{\im}{2}\gamma^{\overline{2}}\wedge\eta^2,\\
&\\
\im\dd\varrho&=-\tfrac{3\im}{2}\gamma^1\wedge\eta^{\overline{1}}-\tfrac{3\im}{2}\gamma^{\overline{1}}\wedge\eta^1+\epsilon\tfrac{\im}{2}\gamma^2\wedge\eta^{\overline{2}}+\epsilon\tfrac{\im}{2}\gamma^{\overline{2}}\wedge\eta^2+\epsilon\eta^{\overline{3}}\wedge\eta^3+F^1F^2\eta^{\overline{2}}\wedge\eta^{\overline{1}}+\overline{F}^1\overline{F}^2\eta^1\wedge\eta^2\\
&+|F^1|^2\eta^{\overline{2}}\wedge\eta^2-|F^2|^2\eta^{\overline{1}}\wedge\eta^1+(Q^1_{\overline{1}}\eta^{\overline{1}}-\overline{Q}^1_{\overline{1}}\eta^1+Q^1_{\overline{2}}\eta^{\overline{2}}-\overline{Q}^1_{\overline{2}}\eta^2)\wedge\eta^0\\
&+R\eta^{\overline{2}}\wedge\eta^1+\overline{R}\eta^{\overline{1}}\wedge\eta^2,\\
&\\
\im\dd\varsigma&=\tfrac{\im}{2}\gamma^1\wedge\eta^{\overline{1}}+\tfrac{\im}{2}\gamma^{\overline{1}}\wedge\eta^1-\epsilon\tfrac{3\im}{2}\gamma^2\wedge\eta^{\overline{2}}-\epsilon\tfrac{3\im}{2}\gamma^{\overline{2}}\wedge\eta^2+\epsilon\eta^{\overline{3}}\wedge\eta^3+F^1F^2\eta^{\overline{1}}\wedge\eta^{\overline{2}}+\overline{F}^1\overline{F}^2\eta^2\wedge\eta^1\\
&+|F^2|^2\eta^{\overline{1}}\wedge\eta^1-|F^1|^2\eta^{\overline{2}}\wedge\eta^2+(Q^2_{\overline{1}}\eta^{\overline{1}}-\overline{Q}^2_{\overline{1}}\eta^1+Q^2_{\overline{2}}\eta^{\overline{2}}-\overline{Q}^2_{\overline{2}}\eta^2)\wedge\eta^0\\
&+S\eta^{\overline{2}}\wedge\eta^1+\overline{S}\eta^{\overline{1}}\wedge\eta^2,\\
&\\
\dd\gamma^1&=-\psi\wedge\eta^1+(\tau-\im\varrho)\wedge\gamma^1-\epsilon\gamma^{\overline{2}}\wedge\eta^3+\im F^3_2\gamma^{\overline{1}}\wedge\eta^0+F^1\gamma^{\overline{1}}\wedge\eta^2-F^1\gamma^2\wedge\eta^{\overline{1}}  \\
 & -\epsilon T^3_{\overline{1}}\eta^{\overline{1}}\wedge\eta^{\overline{2}} +\epsilon \im T^3_{\overline{1}}\eta^{\overline{3}}\wedge\eta^0+ (P_{0}\eta^1+P^1_{0\overline{1}}\eta^{\overline{1}}+P^1_{0\overline{2}}\eta^{\overline{2}}+\im(\overline{Q}^1_{\overline{2}}+\overline{Q}^2_{\overline{2}})\eta^3)\wedge\eta^0\\
 &+(Q^1_{\overline{1}}\eta^{\overline{1}}+Q^1_{\overline{2}}\eta^{\overline{2}}-\overline{Q}^1_{\overline{2}}\eta^2)\wedge\eta^1+(P^1_{20}\eta^0+P^1_{2\overline{1}}\eta^{\overline{1}})\wedge\eta^2,  \\
&\\
\dd\gamma^2&=-\psi\wedge\eta^2+(\tau-\im\varsigma)\wedge\gamma^2-\gamma^{\overline{1}}\wedge\eta^3+\im F^3_1\gamma^{\overline{2}}\wedge\eta^0-F^2\gamma^1\wedge\eta^{\overline{2}}+F^2\gamma^{\overline{2}}\wedge\eta^1\\
&-T^3_{\overline{2}}\eta^{\overline{2}}\wedge\eta^{\overline{1}}+\im T^3_{\overline{2}}\eta^{\overline{3}}\wedge\eta^0+(P^2_{0\overline{1}}\eta^{\overline{1}}+P_{0}\eta^2+P^2_{0\overline{2}}\eta^{\overline{2}}+\im(\overline{Q}^1_{\overline{1}}+\overline{Q}^2_{\overline{1}})\eta^3)\wedge\eta^0\\
&+(Q^2_{\overline{1}}\eta^{\overline{1}}-\overline{Q}^2_{\overline{1}}\eta^1+Q^2_{\overline{2}}\eta^{\overline{2}})\wedge\eta^2+(P^2_{10}\eta^0+P^2_{1\overline{2}}\eta^{\overline{2}})\wedge\eta^1.     
\end{aligned}
\end{equation}

By replacing $\hat{\psi}:=\psi+\tfrac{1}{2}(P_0+\overline{P}_0)\eta^0$, we absorb the real part of $P_0$ in the equations for $\dd\gamma^1$ and $\dd\gamma^2$ without affecting the equation for $\dd\tau$. After this absorption (and dropping the hat), $\psi$ is uniquely and globally determined, and we may replace $P_0$ in our equations with $\im p_0$ where $p_0\in C^\infty(B_4^{(1)})$ is the $\mathbb{R}$-valued $-\tfrac{\im}{2}(P_0-\overline{P}_0)$. 

Note that our equations are now free of any unknown one-forms, which is just in time for us to introduce the last one we will need. It shows up in the equation for $\dd\psi$, which we obtain by differentiating $2\dd\tau$.
\begin{align*}
0
&=\dd(2\dd\tau)\\
&=\Big(-2\dd\psi-4\psi\wedge\tau+2\im\gamma^1\wedge\gamma^{\overline{1}}+\epsilon2\im\gamma^2\wedge\gamma^{\overline{2}} \dots \\
&+\im (P^1_{0\overline{2}}-\epsilon P^2_{0\overline{1}})\eta^{\overline{1}}\wedge\eta^{\overline{2}}+\im(\overline{P}^1_{0\overline{2}}-\epsilon \overline{P}^2_{0\overline{1}})\eta^{2}\wedge\eta^1+\im(P^1_{20}+\epsilon\overline{P}^2_{10})\eta^2\wedge\eta^{\overline{1}}+\im (\epsilon P^2_{10}+\overline{P}^1_{20})\eta^1\wedge\eta^{\overline{2}}\dots\\
&+\epsilon(Q^1_{\overline{1}}+Q^2_{\overline{1}})\eta^{\overline{3}}\wedge\eta^2+\epsilon(\overline{Q}^1_{\overline{1}}+\overline{Q}^2_{\overline{1}})\eta^3\wedge\eta^{\overline{2}}+(\overline{Q}^1_{\overline{2}}+\overline{Q}^2_{\overline{2}})\eta^3\wedge\eta^{\overline{1}}+(Q^1_{\overline{2}}+Q^2_{\overline{2}})\eta^{\overline{3}}\wedge\eta^1\dots\\
&+\overline{F}^3_2\gamma^{1}\wedge\eta^1+F^3_2\gamma^{\overline{1}}\wedge\eta^{\overline{1}}+\epsilon\overline{F}^3_1\gamma^{2}\wedge\eta^2+\epsilon F^3_1\gamma^{\overline{2}}\wedge\eta^{\overline{2}}\dots\\
&+\epsilon\overline{T}^3_{\overline{1}}\eta^{3}\wedge\eta^1+\epsilon T^3_{\overline{1}}\eta^{\overline{3}}\wedge\eta^{\overline{1}}+\epsilon\overline{T}^3_{\overline{2}}\eta^{3}\wedge\eta^2+\epsilon T^3_{\overline{2}}\eta^{\overline{3}}\wedge\eta^{\overline{2}}\Big)\wedge\eta^0.
\end{align*}
Thus, for some $\mathbb{R}$-valued $\zeta\in\Omega^1(B_4^{(1)})$, we have a final structure equation 
\begin{equation}\label{dpsizeta}
\begin{aligned}
\dd\psi&=-2\psi\wedge\tau+\im\gamma^1\wedge\gamma^{\overline{1}}+\epsilon\im\gamma^2\wedge\gamma^{\overline{2}}+\zeta\wedge\eta^0  \\
&+\tfrac{\im}{2} (P^1_{0\overline{2}}-\epsilon P^2_{0\overline{1}})\eta^{\overline{1}}\wedge\eta^{\overline{2}}+\tfrac{\im}{2}(\overline{P}^1_{0\overline{2}}-\epsilon\overline{P}^2_{0\overline{1}})\eta^{2}\wedge\eta^1+\tfrac{\im}{2}(P^1_{20}+\epsilon\overline{P}^2_{10})\eta^2\wedge\eta^{\overline{1}}+\tfrac{\im}{2} (\epsilon P^2_{10}+\overline{P}^1_{20})\eta^1\wedge\eta^{\overline{2}}\\
&+\epsilon\tfrac{1}{2}(Q^1_{\overline{1}}+Q^2_{\overline{1}})\eta^{\overline{3}}\wedge\eta^2+\epsilon\tfrac{1}{2}(\overline{Q}^1_{\overline{1}}+\overline{Q}^2_{\overline{1}})\eta^3\wedge\eta^{\overline{2}}+\tfrac{1}{2}(\overline{Q}^1_{\overline{2}}+\overline{Q}^2_{\overline{2}})\eta^3\wedge\eta^{\overline{1}}+\tfrac{1}{2}(Q^1_{\overline{2}}+Q^2_{\overline{2}})\eta^{\overline{3}}\wedge\eta^1\\
&+\tfrac{1}{2}\overline{F}^3_2\gamma^{1}\wedge\eta^1+\tfrac{1}{2}F^3_2\gamma^{\overline{1}}\wedge\eta^{\overline{1}}+\epsilon\tfrac{1}{2}\overline{F}^3_1\gamma^{2}\wedge\eta^2+\epsilon\tfrac{1}{2}F^3_1\gamma^{\overline{2}}\wedge\eta^{\overline{2}}\\
&+\epsilon\tfrac{1}{2}\overline{T}^3_{\overline{1}}\eta^{3}\wedge\eta^1+\epsilon\tfrac{1}{2}T^3_{\overline{1}}\eta^{\overline{3}}\wedge\eta^{\overline{1}}+\epsilon\tfrac{1}{2}\overline{T}^3_{\overline{2}}\eta^{3}\wedge\eta^2+\epsilon\tfrac{1}{2}T^3_{\overline{2}}\eta^{\overline{3}}\wedge\eta^{\overline{2}}.
\end{aligned}
\end{equation}
In order to expand $\zeta$, we first revisit 
\begin{equation*}
\begin{aligned}
0&=\dd^2\eta^3\\
&\equiv(\dd F^3_2-2F^3_2(\tau-\im\varrho)+2\im F^1\gamma^2+(\im(Q^1_{\overline{1}}-Q^2_{\overline{1}})-F^3_1\overline{F}^2)\eta^1)\wedge\eta^{\overline{1}}\wedge\eta^2&\mod\{\eta^0,\eta^3,\eta^{\overline{2}}\},
\end{aligned}
\end{equation*}
which implies 
\begin{equation}\label{dF32mod}
\begin{aligned}
\dd F^3_2&\equiv2F^3_2(\tau-\im\varrho)-2\im F^1\gamma^2-(\im(Q^1_{\overline{1}}-Q^2_{\overline{1}})-F^3_1\overline{F}^2)\eta^1&&\mod\{\eta^0,\eta^2,\eta^3,\eta^{\overline{1}},\eta^{\overline{2}}\}.\\
\end{aligned}
\end{equation}
Now differentiate $\dd\gamma^1$ and reduce by all of the $\eta$'s except $\eta^0,\eta^1$.
\begin{align*}
0&=\dd^2\gamma^1\\
&\equiv-\zeta\wedge\eta^0\wedge\eta^1+\im(\dd F^3_2-2F^3_2(\tau-\im\varrho)+2\im F^1\gamma^2)\wedge\gamma^{\overline{1}}\wedge\eta^0 \\
&+(\im\dd p_{0}-4\im p_{0}\tau+Q^1_{\overline{1}}\gamma^{\overline{1}}+\overline{Q}^1_{\overline{1}}\gamma^1+\overline{Q}^2_{\overline{2}}\gamma^2+(Q^1_{\overline{2}}-\im\overline{F}^1F^3_2)\gamma^{\overline{2}})\wedge\eta^1\wedge\eta^0&&\mod\{\eta^2,\eta^3,\eta^{\overline{1}},\eta^{\overline{2}},\eta^{\overline{3}}\}.
\end{align*}
Plugging in \eqref{dF32mod} then yields 
\begin{align*} 
\dd^2\gamma^1&\equiv(-\im\dd p_0+4\im p_0\tau-\zeta- \overline{Q}^1_{\overline{1}}\gamma^{1}-(Q^2_{\overline{1}}-\im F^3_1\overline{F}^2)\gamma^{\overline{1}}- \overline{Q}^2_{\overline{2}}\gamma^{2}-(Q^1_{\overline{2}}-\im \overline{F}^1 F^3_2)\gamma^{\overline{2}})\wedge\eta^{0}\wedge\eta^{1} \\
&\mod\{\eta^2,\eta^3,\eta^{\overline{1}},\eta^{\overline{2}},\eta^{\overline{3}}\},\\
\Rightarrow\dd^2\gamma^{\overline{1}}&\equiv(\im\dd p_0-4\im p_0\tau-\zeta- Q^1_{\overline{1}}\gamma^{\overline{1}}-(\overline{Q}^2_{\overline{1}}+\im\overline{F}^3_1F^2)\gamma^{1}- Q^2_{\overline{2}}\gamma^{\overline{2}}-(\overline{Q}^1_{\overline{2}}+\im F^1\overline{F}^3_2)\gamma^{2})\wedge\eta^{0}\wedge\eta^{\overline{1}} \\
&\mod\{\eta^1,\eta^2,\eta^3,\eta^{\overline{2}},\eta^{\overline{3}}\},
\end{align*}
where we have used the fact that $\zeta$ and $p_0$ are $\mathbb{R}$-valued. We exploit this further to calculate 
\begin{align*}
 0&\equiv \dd^2\gamma^1\wedge\eta^{\overline{1}}-\dd^2\gamma^{\overline{1}}\wedge\eta^1&&\mod\{\eta^2,\eta^3,\eta^{\overline{2}},\eta^{\overline{3}}\}\\
 &\equiv \Big(-2\zeta-(\overline{Q}^1_{\overline{1}}+\overline{Q}^2_{\overline{1}}+\im\overline{F}^3_1F^2)\gamma^{1}-(\overline{Q}^2_{\overline{2}}+\overline{Q}^1_{\overline{2}}+\im F^1\overline{F}^3_2)\gamma^{2}\dots\\
 &-(Q^1_{\overline{1}}+Q^2_{\overline{1}}-\im F^3_1\overline{F}^2)\gamma^{\overline{1}}-(Q^2_{\overline{2}}+Q^1_{\overline{2}}-\im\overline{F}^1F^3_2)\gamma^{\overline{2}}\Big)\wedge\eta^0\wedge\eta^1\wedge\eta^{\overline{1}}&&\mod\{\eta^0,\eta^1,\eta^2,\eta^3,\eta^{\overline{1}},\eta^{\overline{2}},\eta^{\overline{3}}\},
\end{align*}
by which we find 
\begin{equation}\label{zeta}
\begin{aligned}
 \zeta&\equiv-\tfrac{1}{2}(\overline{Q}^1_{\overline{1}}+\overline{Q}^2_{\overline{1}}+\im\overline{F}^3_1F^2)\gamma^{1}-\tfrac{1}{2}(\overline{Q}^2_{\overline{2}}+\overline{Q}^1_{\overline{2}}+\im F^1\overline{F}^3_2)\gamma^{2}\\
 &-\tfrac{1}{2}(Q^1_{\overline{1}}+Q^2_{\overline{1}}-\im F^3_1\overline{F}^2)\gamma^{\overline{1}}-\tfrac{1}{2}(Q^2_{\overline{2}}+Q^1_{\overline{2}}-\im\overline{F}^1F^3_2)\gamma^{\overline{2}}&&\mod\{\eta^0,\eta^1,\eta^2,\eta^3,\eta^{\overline{1}},\eta^{\overline{2}},\eta^{\overline{3}}\}.
\end{aligned}
\end{equation}
Thus, if we define $\xi\in\Omega^1(B_4^{(1)})$ to be
\begin{align*}
 \xi:=\zeta&+\tfrac{1}{2}(\overline{Q}^1_{\overline{1}}+\overline{Q}^2_{\overline{1}}+\im\overline{F}^3_1F^2)\gamma^{1}+\tfrac{1}{2}(\overline{Q}^2_{\overline{2}}+\overline{Q}^1_{\overline{2}}+\im F^1\overline{F}^3_2)\gamma^{2}\\
 &+\tfrac{1}{2}(Q^1_{\overline{1}}+Q^2_{\overline{1}}-\im F^3_1\overline{F}^2)\gamma^{\overline{1}}+\tfrac{1}{2}(Q^2_{\overline{2}}+Q^1_{\overline{2}}-\im\overline{F}^1F^3_2)\gamma^{\overline{2}},
\end{align*}
then by \eqref{zeta} we know 
\begin{align*}
 \xi&\equiv0\mod\{\eta^0,\eta^1,\eta^2,\eta^3,\eta^{\overline{1}},\eta^{\overline{2}},\eta^{\overline{3}}\},
\end{align*}
which along with the fact that $\xi$ is $\mathbb{R}$-valued (and wedged against $\eta^0$) means we can expand 
\begin{equation}\label{xi}
\begin{aligned}
 \xi=O_1\eta^1+\overline{O}_1\eta^{\overline{1}}+O_2\eta^2+\overline{O}_2\eta^{\overline{2}}+O_3\eta^3+\overline{O}_3\eta^{\overline{3}},
\end{aligned}
\end{equation}
for some $O\in C^\infty(B_4^{(1)},\mathbb{C})$. We incorporate the expressions \eqref{zeta} and \eqref{xi} into our equation \eqref{dpsizeta} for $\dd\psi$, which we append to our list of completely determined structure equations
\begin{equation}\label{SE}
\begin{aligned}
 \dd\eta^0&=-2\tau\wedge\eta^0+\im\eta^1\wedge\eta^{\overline{1}}+\epsilon\im\eta^2\wedge\eta^{\overline{2}},\\
 \dd\eta^1&=-\gamma^1\wedge\eta^0-(\tau+\im\varrho)\wedge\eta^1+\epsilon\eta^3\wedge\eta^{\overline{2}}+F^1\eta^{\overline{1}}\wedge\eta^2,\\
 \dd\eta^2&=-\gamma^2\wedge\eta^0-(\tau+\im\varsigma)\wedge\eta^2+\eta^3\wedge\eta^{\overline{1}}+F^2\eta^{\overline{2}}\wedge\eta^1,\\
 \dd\eta^3&=-\im\gamma^2\wedge\eta^1-\im\gamma^1\wedge\eta^2-(\im\varrho+\im\varsigma)\wedge\eta^3+T^3_{\overline{1}}\eta^{\overline{1}}\wedge\eta^0+T^3_{\overline{2}}\eta^{\overline{2}}\wedge\eta^0+F^3_1\eta^{\overline{2}}\wedge\eta^1+F^3_2\eta^{\overline{1}}\wedge\eta^2  ,  \\
\\
\dd\tau&=-\psi\wedge\eta^0+\tfrac{\im}{2}\gamma^1\wedge\eta^{\overline{1}}-\tfrac{\im}{2}\gamma^{\overline{1}}\wedge\eta^1+\epsilon\tfrac{\im}{2}\gamma^2\wedge\eta^{\overline{2}}-\epsilon\tfrac{\im}{2}\gamma^{\overline{2}}\wedge\eta^2,\\
\\
\im\dd\varrho&=-\tfrac{3\im}{2}\gamma^1\wedge\eta^{\overline{1}}-\tfrac{3\im}{2}\gamma^{\overline{1}}\wedge\eta^1+\epsilon\tfrac{\im}{2}\gamma^2\wedge\eta^{\overline{2}}+\epsilon\tfrac{\im}{2}\gamma^{\overline{2}}\wedge\eta^2+\epsilon\eta^{\overline{3}}\wedge\eta^3+F^1F^2\eta^{\overline{2}}\wedge\eta^{\overline{1}}+\overline{F}^1\overline{F}^2\eta^1\wedge\eta^2\\
&+|F^1|^2\eta^{\overline{2}}\wedge\eta^2-|F^2|^2\eta^{\overline{1}}\wedge\eta^1+(Q^1_{\overline{1}}\eta^{\overline{1}}-\overline{Q}^1_{\overline{1}}\eta^1+Q^1_{\overline{2}}\eta^{\overline{2}}-\overline{Q}^1_{\overline{2}}\eta^2)\wedge\eta^0\\
&+R\eta^{\overline{2}}\wedge\eta^1+\overline{R}\eta^{\overline{1}}\wedge\eta^2,\\
\\
\im\dd\varsigma&=\tfrac{\im}{2}\gamma^1\wedge\eta^{\overline{1}}+\tfrac{\im}{2}\gamma^{\overline{1}}\wedge\eta^1-\epsilon\tfrac{3\im}{2}\gamma^2\wedge\eta^{\overline{2}}-\epsilon\tfrac{3\im}{2}\gamma^{\overline{2}}\wedge\eta^2+\epsilon\eta^{\overline{3}}\wedge\eta^3+F^1F^2\eta^{\overline{1}}\wedge\eta^{\overline{2}}+\overline{F}^1\overline{F}^2\eta^2\wedge\eta^1\\
&+|F^2|^2\eta^{\overline{1}}\wedge\eta^1-|F^1|^2\eta^{\overline{2}}\wedge\eta^2+(Q^2_{\overline{1}}\eta^{\overline{1}}-\overline{Q}^2_{\overline{1}}\eta^1+Q^2_{\overline{2}}\eta^{\overline{2}}-\overline{Q}^2_{\overline{2}}\eta^2)\wedge\eta^0\\
&+S\eta^{\overline{2}}\wedge\eta^1+\overline{S}\eta^{\overline{1}}\wedge\eta^2,\\
\\
\dd\gamma^1&=-\psi\wedge\eta^1+(\tau-\im\varrho)\wedge\gamma^1-\epsilon\gamma^{\overline{2}}\wedge\eta^3+\im F^3_2\gamma^{\overline{1}}\wedge\eta^0+F^1\gamma^{\overline{1}}\wedge\eta^2-F^1\gamma^2\wedge\eta^{\overline{1}}  \\
 & -\epsilon T^3_{\overline{1}}\eta^{\overline{1}}\wedge\eta^{\overline{2}} +\epsilon\im T^3_{\overline{1}}\eta^{\overline{3}}\wedge\eta^0+ (\im p_{0}\eta^1+P^1_{0\overline{1}}\eta^{\overline{1}}+P^1_{0\overline{2}}\eta^{\overline{2}}+\im(\overline{Q}^1_{\overline{2}}+\overline{Q}^2_{\overline{2}})\eta^3)\wedge\eta^0\\
 &+(Q^1_{\overline{1}}\eta^{\overline{1}}+Q^1_{\overline{2}}\eta^{\overline{2}}-\overline{Q}^1_{\overline{2}}\eta^2)\wedge\eta^1+(P^1_{20}\eta^0+P^1_{2\overline{1}}\eta^{\overline{1}})\wedge\eta^2,  \\
\\
\dd\gamma^2&=-\psi\wedge\eta^2+(\tau-\im\varsigma)\wedge\gamma^2-\gamma^{\overline{1}}\wedge\eta^3+\im F^3_1\gamma^{\overline{2}}\wedge\eta^0-F^2\gamma^1\wedge\eta^{\overline{2}}+F^2\gamma^{\overline{2}}\wedge\eta^1\\
&-T^3_{\overline{2}}\eta^{\overline{2}}\wedge\eta^{\overline{1}}+\im T^3_{\overline{2}}\eta^{\overline{3}}\wedge\eta^0+(P^2_{0\overline{1}}\eta^{\overline{1}}+\im p_{0}\eta^2+P^2_{0\overline{2}}\eta^{\overline{2}}+\im(\overline{Q}^1_{\overline{1}}+\overline{Q}^2_{\overline{1}})\eta^3)\wedge\eta^0\\
&+(Q^2_{\overline{1}}\eta^{\overline{1}}-\overline{Q}^2_{\overline{1}}\eta^1+Q^2_{\overline{2}}\eta^{\overline{2}})\wedge\eta^2+(P^2_{10}\eta^0+P^2_{1\overline{2}}\eta^{\overline{2}})\wedge\eta^1,     \\
\\
\dd\psi&=-2\psi\wedge\tau+\im\gamma^1\wedge\gamma^{\overline{1}}+\epsilon\im\gamma^2\wedge\gamma^{\overline{2}}+(O_1\eta^1+\overline{O}_1\eta^{\overline{1}}+O_2\eta^2+\overline{O}_2\eta^{\overline{2}}+O_3\eta^3+\overline{O}_3\eta^{\overline{3}})\wedge\eta^0  \\
&-\tfrac{1}{2}(\overline{Q}^1_{\overline{1}}+\overline{Q}^2_{\overline{1}}+\im\overline{F}^3_1F^2)\gamma^{1}\wedge\eta^0-\tfrac{1}{2}(\overline{Q}^2_{\overline{2}}+\overline{Q}^1_{\overline{2}}+\im F^1\overline{F}^3_2)\gamma^{2}\wedge\eta^0\\
&-\tfrac{1}{2}(Q^1_{\overline{1}}+Q^2_{\overline{1}}-\im F^3_1\overline{F}^2)\gamma^{\overline{1}}\wedge\eta^0-\tfrac{1}{2}(Q^2_{\overline{2}}+Q^1_{\overline{2}}-\im\overline{F}^1F^3_2)\gamma^{\overline{2}}\wedge\eta^0\\
&+\tfrac{\im}{2} (P^1_{0\overline{2}}-\epsilon P^2_{0\overline{1}})\eta^{\overline{1}}\wedge\eta^{\overline{2}}+\tfrac{\im}{2}(\overline{P}^1_{0\overline{2}}-\epsilon\overline{P}^2_{0\overline{1}})\eta^{2}\wedge\eta^1+\tfrac{\im}{2}(P^1_{20}+\epsilon\overline{P}^2_{10})\eta^2\wedge\eta^{\overline{1}}+\tfrac{\im}{2} (\epsilon P^2_{10}+\overline{P}^1_{20})\eta^1\wedge\eta^{\overline{2}}\\
&+\epsilon\tfrac{1}{2}(Q^1_{\overline{1}}+Q^2_{\overline{1}})\eta^{\overline{3}}\wedge\eta^2+\epsilon\tfrac{1}{2}(\overline{Q}^1_{\overline{1}}+\overline{Q}^2_{\overline{1}})\eta^3\wedge\eta^{\overline{2}}+\tfrac{1}{2}(\overline{Q}^1_{\overline{2}}+\overline{Q}^2_{\overline{2}})\eta^3\wedge\eta^{\overline{1}}+\tfrac{1}{2}(Q^1_{\overline{2}}+Q^2_{\overline{2}})\eta^{\overline{3}}\wedge\eta^1\\
&+\tfrac{1}{2}\overline{F}^3_2\gamma^{1}\wedge\eta^1+\tfrac{1}{2}F^3_2\gamma^{\overline{1}}\wedge\eta^{\overline{1}}+\epsilon\tfrac{1}{2}\overline{F}^3_1\gamma^{2}\wedge\eta^2+\epsilon\tfrac{1}{2}F^3_1\gamma^{\overline{2}}\wedge\eta^{\overline{2}}\\
&+\epsilon\tfrac{1}{2}\overline{T}^3_{\overline{1}}\eta^{3}\wedge\eta^1+\epsilon\tfrac{1}{2}T^3_{\overline{1}}\eta^{\overline{3}}\wedge\eta^{\overline{1}}+\epsilon\tfrac{1}{2}\overline{T}^3_{\overline{2}}\eta^{3}\wedge\eta^2+\epsilon\tfrac{1}{2}T^3_{\overline{2}}\eta^{\overline{3}}\wedge\eta^{\overline{2}}.
\end{aligned}
\end{equation}

Let $\pi:=\underline{\pi}\circ\hat{\pi}$ so we have the bundle $\pi:B_4^{(1)}\to M$. At this point, the coframing of $B_4^{(1)}$ given by the five $\mathbb{R}$-valued forms $\eta^0,\tau,\varrho,\varsigma,\psi$ and the real and imaginary parts of the five $\mathbb{C}$-valued forms $\eta^1,\eta^2,\eta^3,\gamma^1,\gamma^2$ is uniquely and globally determined by the structure equations \eqref{SE}. Thus, this coframing constitutes a solution in the sense of E. Cartan to the equivalence problem for 7-dimensional, 2-nondegenerate CR manifolds whose cubic form is of conformal unitary type.

\vspace{\baselineskip}\section{The Parallelism}\label{parallelism}

%%%%%%%%%%%%%%%%%%%%%%%%%%%%%%%%%%%%%%%%%%%%%%%%%%%%%%%%%%%%%%%%%%%%%%%%%%%%%%%%%%
\subsection{Homogeneous Model}\label{homogmodel}
%%%%%%%%%%%%%%%%%%%%%%%%%%%%%%%%%%%%%%%%%%%%%%%%%%%%%%%%%%%%%%%%%%%%%%%%%%%%%%%%%%

Consider $\mathbb{C}^4$ with its standard basis $\underline{\tv}=(\underline{\tv}_1,\underline{\tv}_2,\underline{\tv}_3,\underline{\tv}_4)$ of column vectors and corresponding complex, linear coordinates $z^1,z^2,z^3,z^4$. A basis $\tv=(\tv_1,\tv_2,\tv_3,\tv_4)$ of column vectors for $\mathbb{C}^4$ will be called an \emph{oriented frame} if 
\begin{align}\label{orframe}
 \tv_1\wedge \tv_2\wedge \tv_3\wedge \tv_4=\underline{\tv}_1\wedge \underline{\tv}_2\wedge \underline{\tv}_3\wedge \underline{\tv}_4.
\end{align}
Let $B_\mathbb{C}^{(1)}$ denote the set of oriented frames, and observe that fixing an identity element $\underline{\tv}$ determines an isomorphism $B_\mathbb{C}^{(1)}\cong SL_4\mathbb{C}$ whereby the oriented frame $\tv$ is identified with the $4\times4$ matrix $[\tv_1,\tv_2,\tv_3,\tv_4]$. If $Gr(2,4)\subset\mathbb{P}(\Lambda^2\mathbb{C}^4)$ denotes the Grassmannian manifold of 2-planes in $\mathbb{C}^4$, then $B_\mathbb{C}^{(1)}$ fibers over $Gr(2,4)$ via the projection map 
\begin{align*}
 \pi(\tv)=\llbracket \tv_1\wedge \tv_2\rrbracket, 
\end{align*}
where the bold brackets denote the projective equivalence class \`{a} la Pl\"{u}cker embedding. This fibration exhibits $Gr(2,4)$ as the homogeneous quotient of $SL_4\mathbb{C}$ by the parabolic subgroup $P\subset SL_4\mathbb{C}$ represented as all matrices of the form 
\begin{align*}
 P=\left[\begin{array}{cccc}*&*&*&*\\ *&*&*&*\\0&0&*&*\\0&0&*&*\end{array}\right],
\end{align*}
i.e., the stabilizer subgroup of the plane spanned by $\underline{\tv}_1,\underline{\tv}_2$.

Let $\epsilon,\delta_\epsilon$ be as in \eqref{epsilon_delta}, and introduce a Hermitian inner product $h$ of signature $(2+\delta_\epsilon,2-\delta_\epsilon)$ on $\mathbb{C}^4$ given in our linear coordinates by 
\begin{align*}
 h(z,w)=z^1\overline{w}^4+z^4\overline{w}^1-\epsilon z^2\overline{w}^2+ z^3\overline{w}^3.
\end{align*}
Now $SU_\star:=SU(2+\delta_\epsilon,2-\delta_\epsilon)\subset SL_4\mathbb{C}$ denotes the subgroup $\{A\in SL_4\mathbb{C}\ |\ h(Az,Aw)=h(z,w)\ \forall z,w\in\mathbb{C}^4\}$, and $Gr(2,4)$ decomposes into $SU_\star$ orbits as follows. Let $\Pi\in Gr(2,4)$. In the $SU(2,2)$ case, $h|_\Pi$ has one of the signatures $(2,0)$,$(0,2)$,$(1,1)$,$(1,0)$,$(0,1)$,$(0,0)$. In the $SU(3,1)$ case, $h|_\Pi$ has one of the signatures $(2,0)$,$(1,1)$,$(1,0)$. In both cases, we let $M_\star$ denote $SU_\star \cdot\llbracket\underline{\tv}_1\wedge\underline{\tv}_2\rrbracket$, which is an orbit of codimension-one in $Gr(2,4)$ where $h|_\Pi$ has signature $(1,0)$. 

An oriented frame $\tv\in B_\mathbb{C}^{(1)}$ will be called a \emph{Hermitian frame} if 
\begin{align}\label{hermframe}
 [h(\tv_i,\tv_j)]_{i,j=1}^4=\left[\begin{array}{crcc}0&0&0&1\\0&-\epsilon&0&0\\0&0&1&0\\1&0&0&0\end{array}\right].
\end{align}
In particular, $\underline{\tv}$ is a Hermitian frame. Let $B^{(1)}\subset B_\mathbb{C}^{(1)}$ be the subset of Hermitian frames, and note that fixing $\underline{\tv}$ once again determines an isomorphism $B^{(1)}\cong SU_\star $ in the same manner as before. The most general transformation of $\underline{\tv}$ which preserves the 2-plane $\llbracket\underline{\tv}_1\wedge\underline{\tv}_2\rrbracket\in Gr(2,4)$ and yields a new Hermitian frame $\tv$ is given by 
\begin{align*}
\tv_1&=\tfrac{1}{t}\e^{\nicefrac{\im}{4}(-r+s)}\underline{\tv}_1,\\
\tv_2&=c^2\e^{-\nicefrac{\im}{4}(r+3s)}\underline{\tv}_1+\e^{-\nicefrac{\im}{4}(r+3s)}\underline{\tv}_2,\\
\tv_3&=-\overline{c}^1\e^{\nicefrac{\im}{4}(3r+s)}\underline{\tv}_1+\e^{\nicefrac{\im}{4}(3r+s)}\underline{\tv}_3,\\
\tv_4&=t\e^{\nicefrac{\im}{4}(-r+s)}(\im y-\tfrac{1}{2}(|c^1|^2-\epsilon|c^2|^2))\underline{\tv}_1+\epsilon\overline{c}^2t\e^{\nicefrac{\im}{4}(-r+s)}\underline{\tv}_2+c^1t\e^{\nicefrac{\im}{4}(-r+s)}\underline{\tv}_3+t\e^{\nicefrac{\im}{4}(-r+s)}\underline{\tv}_4,
\end{align*}
for $r,s,t,y\in\mathbb{R}$ ($t\neq0$) and $c^1,c^2\in\mathbb{C}$. Thus we see that the eight-dimensional Lie group $P_\star:=P\cap SU_\star $ is parameterized by
\begin{align}\label{Pstar}
 \left[\begin{array}{cccc}
\tfrac{1}{t}\e^{\nicefrac{\im}{4}(-r+s)}&c^2\e^{-\nicefrac{\im}{4}(r+3s)}&-\overline{c}^1\e^{\nicefrac{\im}{4}(3r+s)}&t\e^{\nicefrac{\im}{4}(-r+s)}(\im y-\tfrac{1}{2}(|c^1|^2-\epsilon|c^2|^2))\\&&&\\
0&\e^{-\nicefrac{\im}{4}(r+3s)}&0 &\epsilon\overline{c}^2t\e^{\nicefrac{\im}{4}(-r+s)}\\&&&\\
0&0&\e^{\nicefrac{\im}{4}(3r+s)}&c^1t\e^{\nicefrac{\im}{4}(-r+s)}\\&&&\\
0&0&0&t\e^{\nicefrac{\im}{4}(-r+s)}\end{array}\right].
\end{align}
The restriction of the projection $\pi$ to $B^{(1)}$ now determines a fibration over our model space $M_\star$ by which we realize $M_\star$ as the homogeneous quotient $SU_\star /P_\star$. Observe that our parameterization of $P_\star$ may be decomposed into the product $P_\star=P_\star^2P_\star^1P_\star^0$ where the factors are matrices of the form  
\begin{equation}\label{Pgroups}
\begin{aligned}
&P_\star^2=\left[\begin{array}{cccc}
1&0&0&\im y\\
0&1&0 &0\\
0&0&1&0\\
0&0&0&1\end{array}\right],
&P_\star^1=\left[\begin{array}{cccc}
1&c^2&-\overline{c}^1&-\tfrac{1}{2}(|c^1|^2-\epsilon|c^2|^2)\\
0&1&0 &\epsilon\overline{c}^2\\
0&0&1&c^1\\
0&0&0&1\end{array}\right],\\
& &P_\star^0=\left[\begin{array}{cccc}
\tfrac{1}{t}\e^{\nicefrac{\im}{4}(-r+s)}&0&0&0\\
0&\e^{-\nicefrac{\im}{4}(r+3s)}&0 &0\\
0&0&\e^{\nicefrac{\im}{4}(3r+s)}&0\\
0&0&0&t\e^{\nicefrac{\im}{4}(-r+s)}\end{array}\right],
\end{aligned}
\end{equation}
with matrix entries as above. Each of $P_\star^0,P_\star^2$, and the product $P_\star^2P_\star^1$ define subgroups of $SU_\star$, and there is a corresponding tower of fibrations 
\begin{align}\label{SUfibers}
\xymatrix{P_\star^2\ar@{->}[r] & SU_\star \ar@{->}[d]\\
          (P_\star^2P_\star^1)/P_\star^2\ar@{->}[r] & SU_\star /P_\star^2\ar@{->}[d]\\
          P_\star^0\ar@{->}[r] & SU_\star /(P_\star^2P_\star^1)\ar@{->}[d]\\
          & SU_\star /P_\star}.
\end{align}

The four vector-valued functions $B^{(1)}\to\mathbb{C}^4$ given by $\tv\mapsto \tv_j$ ($1\leq j\leq4)$ may be differentiated to obtain one-forms $\omega^i_j\in\Omega^1(B^{(1)},\mathbb{C})$ which we express by 
\begin{align*}
 \dd \tv_j=\tv_i\omega^i_j,
\end{align*}
so that $\omega:=[\omega^i_j]$ is the Maurer-Cartan form of $SU_\star $. Differentiating \eqref{orframe} will show that $\text{trace}(\omega)=0$, while differentiating \eqref{hermframe} reveals
\begin{align*}
\left[\begin{array}{cccc}
 \omega^{\overline{4}}_{\overline{1}}& -\epsilon\omega^{\overline{2}}_{\overline{1}}& \omega^{\overline{3}}_{\overline{1}}& \omega^{\overline{1}}_{\overline{1}}\\ &&&\\
 \omega^{\overline{4}}_{\overline{2}}& -\epsilon\omega^{\overline{2}}_{\overline{2}}& \omega^{\overline{3}}_{\overline{2}}& \omega^{\overline{1}}_{\overline{2}}\\&&&\\
 \omega^{\overline{4}}_{\overline{3}}& -\epsilon\omega^{\overline{2}}_{\overline{3}}& \omega^{\overline{3}}_{\overline{3}}& \omega^{\overline{1}}_{\overline{3}}\\&&&\\
 \omega^{\overline{4}}_{\overline{4}}& -\epsilon\omega^{\overline{2}}_{\overline{4}}& \omega^{\overline{3}}_{\overline{4}}& \omega^{\overline{1}}_{\overline{4}}\end{array}\right]
 +
\left[\begin{array}{rrrr} 
 \omega^4_1&    \omega^4_2&    \omega^4_3&    \omega^4_4\\&&&\\
 -\epsilon\omega^2_1&   -\epsilon\omega^2_2&    -\epsilon\omega^2_3&    -\epsilon\omega^2_4\\&&&\\
 \omega^3_1& \omega^3_2& \omega^3_3& \omega^3_4\\&&&\\
    \omega^1_1&    \omega^1_2&    \omega^1_3&    \omega^1_4\end{array}\right]=0,
\end{align*}
which is simply to say that $\omega$ takes values in the Lie algebra $\mathfrak{su}_\star$ of $SU_\star$. These conditions show that if we let 
\begin{align*}
&\eta^0:=-\text{Im}(\omega^4_1),
&&\eta^1:=\omega^3_1,
&&&\eta^2:=\omega^4_2,
&&&&\eta^3:=\omega^3_2,
&&&&&\tau:=\text{Re}(\omega^1_1),\\
&\im\varrho:=\tfrac{1}{2}(3\omega^3_3+\omega^2_2),
&&\im\varsigma:=-\tfrac{1}{2}(3\omega^2_2+\omega^3_3),
&&&\im\gamma^1:=\omega_4^3,
&&&&-\im\gamma^2:=\omega^1_2,
&&&&&\psi:=-\text{Im}(\omega^1_4),
\end{align*}
then we can write 
\begin{align}\label{omega}
\omega=\left[\begin{array}{cccc}
 -\tau-\im\tfrac{1}{4}\varrho+\im\tfrac{1}{4}\varsigma        &-\im\gamma^2       &-\im\gamma^{\overline{1}}        &-\im\psi\\ 
 & & & \\
 -\epsilon\eta^{\overline{2}} &-\im\tfrac{1}{4}\varrho-\im\tfrac{3}{4}\varsigma  &\epsilon\eta^{\overline{3}}  &-\epsilon\im\gamma^{\overline{2}}\\
 & & & \\
 \eta^1        &\eta^3      &\im\tfrac{3}{4}\varrho+\im\tfrac{1}{4}\varsigma    &\im\gamma^1\\
 & & & \\
 -\im\eta^0       &\eta^2      &\eta^{\overline{1}}  &\tau-\im\tfrac{1}{4}\varrho+\im\tfrac{1}{4}\varsigma
 \end{array}\right],
\end{align}
and the $SU_\star $ Maurer-Cartan equations $\dd\omega+\omega\wedge\omega=0$ read 
\begin{equation}\label{MC}
\begin{aligned}
\dd\eta^0&=-2\tau\wedge\eta^0+\im\eta^1\wedge\eta^{\overline{1}}+\epsilon\im\eta^2\wedge\eta^{\overline{2}},\\
 \dd\eta^1&=-\gamma^1\wedge\eta^0-(\tau+\im\varrho)\wedge\eta^1+\epsilon\eta^3\wedge\eta^{\overline{2}},\\
 \dd\eta^2&=-\gamma^2\wedge\eta^0-(\tau+\im\varsigma)\wedge\eta^2+\eta^3\wedge\eta^{\overline{1}},\\
 \dd\eta^3&=-\im\gamma^2\wedge\eta^1-\im\gamma^1\wedge\eta^2-(\im\varrho+\im\varsigma)\wedge\eta^3,\\
\dd\tau&=-\psi\wedge\eta^0+\tfrac{\im}{2}\gamma^1\wedge\eta^{\overline{1}}-\tfrac{\im}{2}\gamma^{\overline{1}}\wedge\eta^1+\epsilon\tfrac{\im}{2}\gamma^2\wedge\eta^{\overline{2}}-\epsilon\tfrac{\im}{2}\gamma^{\overline{2}}\wedge\eta^2,\\
\im\dd\varrho&=-\tfrac{3\im}{2}\gamma^1\wedge\eta^{\overline{1}}-\tfrac{3\im}{2}\gamma^{\overline{1}}\wedge\eta^1+\epsilon\tfrac{\im}{2}\gamma^2\wedge\eta^{\overline{2}}+\epsilon\tfrac{\im}{2}\gamma^{\overline{2}}\wedge\eta^2+\epsilon\eta^{\overline{3}}\wedge\eta^3,\\
\im\dd\varsigma&=\tfrac{\im}{2}\gamma^1\wedge\eta^{\overline{1}}+\tfrac{\im}{2}\gamma^{\overline{1}}\wedge\eta^1-\epsilon\tfrac{3\im}{2}\gamma^2\wedge\eta^{\overline{2}}-\epsilon\tfrac{3\im}{2}\gamma^{\overline{2}}\wedge\eta^2+\epsilon\eta^{\overline{3}}\wedge\eta^3,\\
\dd\gamma^1&=-\psi\wedge\eta^1+(\tau-\im\varrho)\wedge\gamma^1-\epsilon\gamma^{\overline{2}}\wedge\eta^3,\\
\dd\gamma^2&=-\psi\wedge\eta^2+(\tau-\im\varsigma)\wedge\gamma^2-\gamma^{\overline{1}}\wedge\eta^3,\\
\dd\psi&=-2\psi\wedge\tau+\im\gamma^1\wedge\gamma^{\overline{1}}+\epsilon\im\gamma^2\wedge\gamma^{\overline{2}}.
\end{aligned}
\end{equation}
Observe that the equations \eqref{MC} show 
\begin{align*}
 \dd(\psi-2\tau+\eta^0)&=(\psi-2\tau+\eta^0)\wedge(\eta^0-\psi)+\im(\gamma^1-\eta^1)\wedge(\gamma^{\overline{1}}-\eta^{\overline{1}})+\epsilon\im(\gamma^2-\eta^2)\wedge(\gamma^{\overline{2}}-\eta^{\overline{2}}),\\
 \dd(\gamma^1-\eta^1)&=-(\psi-2\tau+\eta^0)\wedge\eta^1+(\gamma^1-\eta^1)\wedge\eta^0+(\tau-\im\varrho)\wedge(\gamma^1-\eta^1)-\epsilon(\gamma^{\overline{2}}-\eta^{\overline{2}})\wedge\eta^3,\\
 \dd(\gamma^2-\eta^2)&=-(\psi-2\tau+\eta^0)\wedge\eta^2+(\gamma^2-\eta^2)\wedge\eta^0+(\tau-\im\varsigma)\wedge(\gamma^2-\eta^2)-(\gamma^{\overline{1}}-\eta^{\overline{1}})\wedge\eta^3,
\end{align*}
which proves that the Pfaffian system $\mathcal{I}:=\{\psi-2\tau+\eta^0,\gamma^1-\eta^1,\gamma^2-\eta^2,\gamma^{\overline{1}}-\eta^{\overline{1}},\gamma^{\overline{2}}-\eta^{\overline{2}}\}$ on $B^{(1)}$ is Frobenius. We let $B_\mathcal{I}$ denote the maximal integral manifold of $\mathcal{I}$ that contains $\underline{\tv}$, with $\iota:B_\mathcal{I}\hookrightarrow B^{(1)}$ as the inclusion. Then $\omega\in\Omega^1(B^{(1)},\mathfrak{su}_\star)$ pulls back to 
\begin{align*}
 \iota^*\omega=\iota^*\left[\begin{array}{cccc}
 -\tau-\im\tfrac{1}{4}\varrho+\im\tfrac{1}{4}\varsigma        &-\im\eta^2       &-\im\eta^{\overline{1}}        &-\im(2\tau-\eta^0)\\ 
 & & & \\
 -\epsilon\eta^{\overline{2}} &-\im\tfrac{1}{4}\varrho-\im\tfrac{3}{4}\varsigma  &\epsilon\eta^{\overline{3}}  &-\epsilon\im\eta^{\overline{2}}\\
 & & & \\
 \eta^1        &\eta^3      &\im\tfrac{3}{4}\varrho+\im\tfrac{1}{4}\varsigma    &\im\eta^1\\
 & & & \\
 -\im\eta^0       &\eta^2      &\eta^{\overline{1}}  &\tau-\im\tfrac{1}{4}\varrho+\im\tfrac{1}{4}\varsigma
 \end{array}\right]\in\Omega^1(B_\mathcal{I},\mathfrak{su}_\star),
\end{align*}
and in particular on $B_\mathcal{I}$ we have 
\begin{equation}\label{pfaffMC}
\iota^*\dd\omega+\iota^*\omega\wedge\iota^*\omega=0. 
\end{equation}
Moreover, when restricted to the fibers of $\pi|_{B_{\mathcal{I}}}:B_\mathcal{I}\to M_\star$ (where the pullbacks of the $\eta$'s vanish), \eqref{pfaffMC} is exactly the Maurer-Cartan equations of the abelian subgroup $P_\star^0\subset SU_\star$. By a theorem of E. Cartan (\cite[Thm 1.6.10]{CfB}), there exist local lifts $B_\mathcal{I}\to SU_\star$ by which the fibers of $B_\mathcal{I}$ are diffeomorphic to $P_\star^0$, and the fibration 
\begin{align*}
\xymatrix{P_\star^0\ar@{->}[r] & B_\mathcal{I}\ar@{->}[d]\\
          & M_\star}
\end{align*}
corresponds to the lowest level of the tower \eqref{SUfibers}.

Using our identifications $B^{(1)}\cong SU_\star$ and $B_\mathcal{I}\cong SU_\star/(P_\star^2P_\star^1)$, we see that $B^{(1)}$ fibers over $B_\mathcal{I}$ as the $P_\star^2P_\star^1$-orbits of Hermitian frames in $B_\mathcal{I}$. We therefore identify an intermediate bundle $B\cong SU_\star/P_\star^2$ as the $(P_\star^2P_\star^1)/P_\star^2$-orbits ($P_\star^2$ is normal in $P_\star^2P_\star^1$). The significance of $B$ is that it corresponds to the bundle $B_4$ constructed in \S\ref{eqprob} when $M=M_\star$.

%%%%%%%%%%%%%%%%%%%%%%%%%%%%%%%%%%%%%%%%%%%%%%%%%%%%%%%%%%%%%%%%%%%%%%%%%%%%%%%%%%
\vspace{\baselineskip}\subsection{Bianchi Identities, Fundamental Invariants}\label{FundInv}
%%%%%%%%%%%%%%%%%%%%%%%%%%%%%%%%%%%%%%%%%%%%%%%%%%%%%%%%%%%%%%%%%%%%%%%%%%%%%%%%%%

We return to the bundle $\pi:B_4^{(1)}\to M$ as in \S\ref{eqprob}. The coframing constructed therein is interpreted as a parallelism $\omega\in\Omega^1(B_4^{(1)},\mathfrak{su}_\star)$ by writing $\omega$ as in \eqref{omega}. The structure equations \eqref{SE} on $B_4^{(1)}$ are now summarized 
\begin{align*}
 \dd\omega=-\omega\wedge\omega+C
\end{align*}
where the curvature tensor $C\in\Omega^2(B_4^{(1)},\mathfrak{su}_\star)$ may be written
\begin{align}\label{curvature}
 C=\left[\begin{array}{cccc}
      C^1_1&-\im C^1_2&-\im\overline{C}^3_4&-\im C^1_4\\
      -\epsilon\overline{F}^2\eta^{2}\wedge\eta^{\overline{1}}&C^2_2&\epsilon\overline{C}^3_2&-\epsilon\im\overline{C}^1_2\\
      F^1\eta^{\overline{1}}\wedge\eta^2&C^3_2&C^3_3&\im C^3_4\\
      0&F^2\eta^{\overline{2}}\wedge\eta^1&\overline{F}^1\eta^{1}\wedge\eta^{\overline{2}}&C^1_1\end{array}\right],
\end{align}
for $C^i_j\in\Omega^2(B_4^{(1)},\mathbb{C})$ given by

\begin{align*}
 C^3_2&=T^3_{\overline{1}}\eta^{\overline{1}}\wedge\eta^0+T^3_{\overline{2}}\eta^{\overline{2}}\wedge\eta^0+F^3_1\eta^{\overline{2}}\wedge\eta^1+F^3_2\eta^{\overline{1}}\wedge\eta^2,\\
\\
 C^1_1&=\tfrac{1}{4}(Q^1_{\overline{1}}-Q^2_{\overline{1}})\eta^0\wedge\eta^{\overline{1}}+\tfrac{1}{4}(Q^1_{\overline{2}}-Q^2_{\overline{2}})\eta^0\wedge\eta^{\overline{2}}
 +\tfrac{1}{4}(\overline{Q}^1_{\overline{1}}-\overline{Q}^2_{\overline{1}})\eta^1\wedge\eta^0+\tfrac{1}{4}(\overline{Q}^1_{\overline{2}}-\overline{Q}^2_{\overline{2}})\eta^2\wedge\eta^0\\
 &+\tfrac{1}{2}F^1F^2\eta^{\overline{1}}\wedge\eta^{\overline{2}}-\tfrac{1}{2}\overline{F}^1\overline{F}^2\eta^1\wedge\eta^2+\tfrac{1}{2}|F^1|^2\eta^2\wedge\eta^{\overline{2}}-\tfrac{1}{2}|F^2|^2\eta^1\wedge\eta^{\overline{1}}\\
 &+\tfrac{1}{4}(R-S)\eta^1\wedge\eta^{\overline{2}}+\tfrac{1}{4}(\overline{R}-\overline{S})\eta^2\wedge\eta^{\overline{1}},\\
\\
 C^2_2&=\tfrac{1}{4}(Q^1_{\overline{1}}+3Q^2_{\overline{1}})\eta^0\wedge\eta^{\overline{1}}+\tfrac{1}{4}(Q^1_{\overline{2}}+3Q^2_{\overline{2}})\eta^0\wedge\eta^{\overline{2}}
 +\tfrac{1}{4}(\overline{Q}^1_{\overline{1}}+3\overline{Q}^2_{\overline{1}})\eta^1\wedge\eta^0+\tfrac{1}{4}(\overline{Q}^1_{\overline{2}}+3\overline{Q}^2_{\overline{2}})\eta^2\wedge\eta^0\\
 &-\tfrac{1}{2}F^1F^2\eta^{\overline{1}}\wedge\eta^{\overline{2}}+\tfrac{1}{2}\overline{F}^1\overline{F}^2\eta^1\wedge\eta^2-\tfrac{1}{2}|F^1|^2\eta^2\wedge\eta^{\overline{2}}+\tfrac{1}{2}|F^2|^2\eta^1\wedge\eta^{\overline{1}}\\
 &+\tfrac{1}{4}(R+3S)\eta^1\wedge\eta^{\overline{2}}+\tfrac{1}{4}(\overline{R}+3\overline{S})\eta^2\wedge\eta^{\overline{1}},\\
\\
 C^3_3&=-\tfrac{1}{4}(3Q^1_{\overline{1}}+Q^2_{\overline{1}})\eta^0\wedge\eta^{\overline{1}}-\tfrac{1}{4}(3Q^1_{\overline{2}}+Q^2_{\overline{2}})\eta^0\wedge\eta^{\overline{2}}
 +\tfrac{1}{4}(3\overline{Q}^1_{\overline{1}}+\overline{Q}^2_{\overline{1}})\eta^0\wedge\eta^1+\tfrac{1}{4}(3\overline{Q}^1_{\overline{2}}+\overline{Q}^2_{\overline{2}})\eta^0\wedge\eta^2\\
 &-\tfrac{1}{2}F^1F^2\eta^{\overline{1}}\wedge\eta^{\overline{2}}+\tfrac{1}{2}\overline{F}^1\overline{F}^2\eta^1\wedge\eta^2-\tfrac{1}{2}|F^1|^2\eta^2\wedge\eta^{\overline{2}}+\tfrac{1}{2}|F^2|^2\eta^1\wedge\eta^{\overline{1}}\\
 &-\tfrac{1}{4}(3R+S)\eta^1\wedge\eta^{\overline{2}}-\tfrac{1}{4}(3\overline{R}+\overline{S})\eta^2\wedge\eta^{\overline{1}},\\
\\
 C^1_2&=\im F^3_1\gamma^{\overline{2}}\wedge\eta^0-F^2\gamma^1\wedge\eta^{\overline{2}}+F^2\gamma^{\overline{2}}\wedge\eta^1-T^3_{\overline{2}}\eta^{\overline{2}}\wedge\eta^{\overline{1}}+\im T^3_{\overline{2}}\eta^{\overline{3}}\wedge\eta^0\\
 &+(P^2_{0\overline{1}}\eta^{\overline{1}}+\im p_{0}\eta^2+P^2_{0\overline{2}}\eta^{\overline{2}}+\im(\overline{Q}^1_{\overline{1}}+\overline{Q}^2_{\overline{1}})\eta^3)\wedge\eta^0+(P^2_{10}\eta^0+P^2_{1\overline{2}}\eta^{\overline{2}})\wedge\eta^1\\
&+(Q^2_{\overline{1}}\eta^{\overline{1}}-\overline{Q}^2_{\overline{1}}\eta^1+Q^2_{\overline{2}}\eta^{\overline{2}})\wedge\eta^2,\\
\\
C^3_4&=\im F^3_2\gamma^{\overline{1}}\wedge\eta^0+F^1\gamma^{\overline{1}}\wedge\eta^2-F^1\gamma^2\wedge\eta^{\overline{1}}-\epsilon T^3_{\overline{1}}\eta^{\overline{1}}\wedge\eta^{\overline{2}} +\epsilon \im T^3_{\overline{1}}\eta^{\overline{3}}\wedge\eta^0\\
&+ (\im p_{0}\eta^1+P^1_{0\overline{1}}\eta^{\overline{1}}+P^1_{0\overline{2}}\eta^{\overline{2}}+\im(\overline{Q}^1_{\overline{2}}+\overline{Q}^2_{\overline{2}})\eta^3)\wedge\eta^0+(P^1_{20}\eta^0+P^1_{2\overline{1}}\eta^{\overline{1}})\wedge\eta^2\\
 &+(Q^1_{\overline{1}}\eta^{\overline{1}}+Q^1_{\overline{2}}\eta^{\overline{2}}-\overline{Q}^1_{\overline{2}}\eta^2)\wedge\eta^1,\\
\\
C^1_4&=(O_1\eta^1+\overline{O}_1\eta^{\overline{1}}+O_2\eta^2+\overline{O}_2\eta^{\overline{2}}+O_3\eta^3+\overline{O}_3\eta^{\overline{3}})\wedge\eta^0 -\tfrac{1}{2}(\overline{Q}^1_{\overline{1}}+\overline{Q}^2_{\overline{1}}+\im\overline{F}^3_1F^2)\gamma^{1}\wedge\eta^0\\
&-\tfrac{1}{2}(\overline{Q}^2_{\overline{2}}+\overline{Q}^1_{\overline{2}}+\im F^1\overline{F}^3_2)\gamma^{2}\wedge\eta^0-\tfrac{1}{2}(Q^1_{\overline{1}}+Q^2_{\overline{1}}-\im F^3_1\overline{F}^2)\gamma^{\overline{1}}\wedge\eta^0-\tfrac{1}{2}(Q^2_{\overline{2}}+Q^1_{\overline{2}}-\im\overline{F}^1F^3_2)\gamma^{\overline{2}}\wedge\eta^0\\
&+\tfrac{\im}{2} (P^1_{0\overline{2}}-\epsilon P^2_{0\overline{1}})\eta^{\overline{1}}\wedge\eta^{\overline{2}}+\tfrac{\im}{2}(\overline{P}^1_{0\overline{2}}-\epsilon\overline{P}^2_{0\overline{1}})\eta^{2}\wedge\eta^1+\tfrac{\im}{2}(P^1_{20}+\epsilon\overline{P}^2_{10})\eta^2\wedge\eta^{\overline{1}}+\tfrac{\im}{2} (\epsilon P^2_{10}+\overline{P}^1_{20})\eta^1\wedge\eta^{\overline{2}}\\
&+\epsilon\tfrac{1}{2}(Q^1_{\overline{1}}+Q^2_{\overline{1}})\eta^{\overline{3}}\wedge\eta^2+\epsilon\tfrac{1}{2}(\overline{Q}^1_{\overline{1}}+\overline{Q}^2_{\overline{1}})\eta^3\wedge\eta^{\overline{2}}+\tfrac{1}{2}(\overline{Q}^1_{\overline{2}}+\overline{Q}^2_{\overline{2}})\eta^3\wedge\eta^{\overline{1}}+\tfrac{1}{2}(Q^1_{\overline{2}}+Q^2_{\overline{2}})\eta^{\overline{3}}\wedge\eta^1\\
&+\tfrac{1}{2}\overline{F}^3_2\gamma^{1}\wedge\eta^1+\tfrac{1}{2}F^3_2\gamma^{\overline{1}}\wedge\eta^{\overline{1}}+\epsilon\tfrac{1}{2}\overline{F}^3_1\gamma^{2}\wedge\eta^2+\epsilon\tfrac{1}{2}F^3_1\gamma^{\overline{2}}\wedge\eta^{\overline{2}}\\
&+\epsilon\tfrac{1}{2}\overline{T}^3_{\overline{1}}\eta^{3}\wedge\eta^1+\epsilon\tfrac{1}{2}T^3_{\overline{1}}\eta^{\overline{3}}\wedge\eta^{\overline{1}}+\epsilon\tfrac{1}{2}\overline{T}^3_{\overline{2}}\eta^{3}\wedge\eta^2+\epsilon\tfrac{1}{2}T^3_{\overline{2}}\eta^{\overline{3}}\wedge\eta^{\overline{2}}.
\end{align*}

The coefficients which appear at lowest order are $F^1,F^2$. We find how they vary on $B_4^{(1)}$ by differentiating the structure equations 
\begin{equation*}
\begin{aligned}
0&=\dd(\dd\eta^1)\\
&=(\dd F^1-F^1(\tau-2\im\varrho+\im\varsigma)+\epsilon\overline{F}^2\eta^3+\epsilon F^3_2\eta^{\overline{2}}-\overline{R}\eta^1-P^1_{2\overline{1}}\eta^0)\wedge\eta^{\overline{1}}\wedge\eta^2,
\end{aligned}
\end{equation*}
and similarly,
\begin{equation*}
\begin{aligned}
0&=\dd(\dd\eta^2)\\
&=(\dd F^2-F^2(\tau+\im\varrho-2\im\varsigma)+\overline{F}^1\eta^3+F^3_1\eta^{\overline{1}}-S\eta^2-P^2_{1\overline{2}}\eta^0)\wedge\eta^{\overline{2}}\wedge\eta^1.
\end{aligned}
\end{equation*}
Therefore, for some functions $f^1_{\overline{1}},f^1_2,f^2_1,f^2_{\overline{2}}\in C^\infty(B_4^{(1)},\mathbb{C})$ we can write 
\begin{equation}\label{dFs}
 \begin{aligned}
\dd F^1&= F^1(\tau-2\im\varrho+\im\varsigma)-\epsilon\overline{F}^2\eta^3-\epsilon F^3_2\eta^{\overline{2}}+\overline{R}\eta^1+P^1_{2\overline{1}}\eta^0+f^1_{\overline{1}}\eta^{\overline{1}}+f^1_2\eta^2,\\
\dd F^2&= F^2(\tau+\im\varrho-2\im\varsigma)-\overline{F}^1\eta^3-F^3_1\eta^{\overline{1}}+S\eta^2+P^2_{1\overline{2}}\eta^0+ f^2_1\eta^1+f^2_{\overline{2}}\eta^{\overline{2}}.
 \end{aligned}
\end{equation}
Recall (\cite[Prop B.3.3]{CfB}) that a form $\alpha\in\Omega^\bullet(B_4^{(1)},\mathbb{C})$ is $\pi$-basic if and only if $\alpha$ and $\dd\alpha$ are $\pi$-semibasic. We consider the $\mathbb{R}$-valued semibasic forms 
\begin{align}\label{fundinv}
 &|F^1|^2\eta^0,
 &|F^2|^2\eta^0,
\end{align}
and use \eqref{dFs} to calculate 
\begin{align*}
\dd(|F^1|^2\eta^0)&=-(\overline{F}^1\overline{R}+\overline{f}^1_{\overline{1}}F^1)\eta^0\wedge\eta^1-(F^1R+f^1_{\overline{1}}\overline{F}^1)\eta^0\wedge\eta^{\overline{1}}+\im|F^1|^2\eta^1\wedge\eta^{\overline{1}}+\epsilon\overline{F}^1\overline{F}^2\eta^0\wedge\eta^3\\
&-(\overline{F}^1f^1_2-\epsilon\overline{F}^3_2 F^1)\eta^0\wedge\eta^2-(F^1\overline{f}^1_2-\epsilon F^3_2\overline{F}^1)\eta^0\wedge\eta^{\overline{2}}+\epsilon\im|F^1|^2\eta^2\wedge\eta^{\overline{2}}+\epsilon F^1F^2\eta^0\wedge\eta^{\overline{3}},\\
\dd(|F^2|^2\eta^0)&=-(\overline{F}^2\overline{S}+\overline{f}^2_{\overline{1}}F^2)\eta^0\wedge\eta^2-(F^2S+f^2_{\overline{2}}\overline{F}^2)\eta^0\wedge\eta^{\overline{2}}+\im|F^2|^2\eta^1\wedge\eta^{\overline{1}}+\epsilon\overline{F}^1\overline{F}^2\eta^0\wedge\eta^3\\
&-(\overline{F}^2f^2_1-\epsilon\overline{F}^3_1 F^2)\eta^0\wedge\eta^1-(F^2\overline{f}^2_1-\epsilon F^3_1\overline{F}^2)\eta^0\wedge\eta^{\overline{1}}+\epsilon\im|F^2|^2\eta^2\wedge\eta^{\overline{2}}+\epsilon F^1F^2\eta^0\wedge\eta^{\overline{3}}.
\end{align*}
These are semibasic as well, so we've shown that the one-forms \eqref{fundinv} on $B_4^{(1)}$ are the $\pi$-pullbacks of well-defined invariants on $M$.

Let us make a few more observations about the equations \eqref{dFs}. First, they show that if $F^1$ or $F^2$ is locally constant on $B_4^{(1)}$, then they must locally vanish. Second, we see that if either of $F^1$, $F^2$ vanishes identically, the other must as well. By the same token, we will have  
\begin{equation}\label{van1}
F^3_1=F^3_2=R=S=P^1_{2\overline{1}}=P^2_{1\overline{2}}=0
\end{equation}
in this case. In fact, if either of $F^1,F^2=0$, we will show that every coefficient function in the curvature tensor $C$ must vanish too. This will follow by differentiating more of the structure equations. We revisit
{\small
\begin{equation}\label{d2eta3}
\begin{aligned}
0&=\dd^2\eta^3\\
&=(\dd T^3_{\overline{1}}-T^3_{\overline{1}}(3\tau-2\im\varrho-\im\varsigma)-F^3_2\gamma^2-(Q^1_{\overline{1}}+Q^2_{\overline{1}})\eta^3-\im P^2_{0\overline{1}}\eta^1)\wedge\eta^{\overline{1}}\wedge\eta^0\\
&+(\dd T^3_{\overline{2}}-T^3_{\overline{2}}(3\tau-\im\varrho-2\im\varsigma)-F^3_1\gamma^1-(Q^1_{\overline{2}}+Q^2_{\overline{2}})\eta^3-\im P^1_{0\overline{2}}\eta^2)\wedge\eta^{\overline{2}}\wedge\eta^0\\
&+(\dd F^3_1-2F^3_1(\tau-\im\varsigma)+ 2\im F^2\gamma^1-(R+S)\eta^3+(\im(Q^2_{\overline{2}}-Q^1_{\overline{2}})-F^3_2\overline{F}^1)\eta^2-(F^3_2F^2+2\im T^3_{\overline{2}})\eta^{\overline{1}})\wedge\eta^{\overline{2}}\wedge\eta^1\\
&+(\dd F^3_2-2F^3_2(\tau-\im\varrho)+2\im F^1\gamma^2-(\overline{R}+\overline{S})\eta^3+(\im(Q^1_{\overline{1}}-Q^2_{\overline{1}})-F^3_1\overline{F}^2)\eta^1-(F^3_1F^1+ \epsilon2\im T^3_{\overline{1}})\eta^{\overline{2}})\wedge\eta^{\overline{1}}\wedge\eta^2\\
&+(T^3_{\overline{2}}\overline{F}^2-\im P^1_{0\overline{1}})\eta^{2}\wedge\eta^{\overline{1}}\wedge\eta^0+(T^3_{\overline{1}}\overline{F}^1-\im P^2_{0\overline{2}})\eta^{1}\wedge\eta^{\overline{2}}\wedge\eta^0. 
\end{aligned}
\end{equation}}
Reducing \eqref{d2eta3} by $\eta^0$ and plugging in $F^3_1=F^3_2=0$ implies $T^3_{\overline{1}}=T^3_{\overline{2}}=0$ and $Q^1_{\overline{1}}=Q^2_{\overline{1}}$, $Q^2_{\overline{2}}=Q^1_{\overline{2}}$. Then, returning to the unreduced \eqref{d2eta3} and setting $T^3_{\overline{1}}=T^3_{\overline{2}}=0$ will show 
\begin{equation}\label{van2}
T^3_{\overline{1}}=T^3_{\overline{2}}=Q^1_{\overline{1}}=Q^1_{\overline{2}}=Q^2_{\overline{1}}=Q^2_{\overline{2}}=P^1_{0\overline{1}}=P^1_{0\overline{2}}=P^2_{0\overline{1}}=P^2_{0\overline{2}}=0.
\end{equation}
We assume that we have \eqref{van1} and \eqref{van2} as we now differentiate $\im\dd\varrho$ and $\im\dd\varsigma$;
\begin{align*}
 0&=\dd(\im\dd\varrho)\\
 &=-3p_0\eta^0\wedge\eta^1\wedge\eta^{\overline{1}}+\epsilon p_0\eta^0\wedge\eta^2\wedge\eta^{\overline{2}}+\tfrac{\im}{2}(\epsilon P^2_{10}+3\overline{P}^1_{20})\eta^0\wedge\eta^1\wedge\eta^{\overline{2}}-\tfrac{\im}{2}(\epsilon\overline{P}^2_{10}+3P^1_{20})\eta^0\wedge\eta^2\wedge\eta^{\overline{1}},
\end{align*}
and 
\begin{align*}
 0&=\dd(\im\dd\varsigma)\\
 &= p_0\eta^0\wedge\eta^1\wedge\eta^{\overline{1}}-\epsilon3p_0\eta^0\wedge\eta^2\wedge\eta^{\overline{2}}-\tfrac{\im}{2}(\epsilon3 P^2_{10}+\overline{P}^1_{20})\eta^0\wedge\eta^1\wedge\eta^{\overline{2}}+\tfrac{\im}{2}(\epsilon3\overline{P}^2_{10}+P^1_{20})\eta^0\wedge\eta^2\wedge\eta^{\overline{1}},
\end{align*}
which together demonstrate 
\begin{equation}\label{van3}
p_0=P^1_{20}=P^2_{10}=0.
\end{equation}
Finally, we simply state that differentiating $\dd\gamma^1$ and $\dd\gamma^2$ will now show 
\begin{equation}\label{van4}
O_1=O_2=O_3=0.
\end{equation}

By \eqref{van1},\eqref{van2},\eqref{van3}, and \eqref{van4}, we have shown that $C=0$ when one of \eqref{fundinv} vanishes. In this case, the structure equations of $M$ are exactly the Maurer-Cartan equations \eqref{MC}, and $M$ is locally CR-equivalent to the homogeneous model $M_\star$.

%%%%%%%%%%%%%%%%%%%%%%%%%%%%%%%%%%%%%%%%%%%%%%%%%%%%%%%%%%%%%%%%%%%%%%%%%%%%%%%%%%
\vspace{\baselineskip}\subsection{Equivariance}
%%%%%%%%%%%%%%%%%%%%%%%%%%%%%%%%%%%%%%%%%%%%%%%%%%%%%%%%%%%%%%%%%%%%%%%%%%%%%%%%%%

Let us establish some general definitions which we will use to interpret the bundles $\hat{\pi}:B_4^{(1)}\to B_4$ and $\pi:B_4^{(1)}\to M$ constructed in \S\ref{eqprob}. A reference for this material is \cite{CapSlovak}. Let $G$ be a Lie group with Lie algebra $\mathfrak{g}$, $H\subset G$ a Lie subgroup with Lie algebra $\mathfrak{h}\subset\mathfrak{g}$, and $\exp:\mathfrak{h}\to H$ the exponential map. For each $g\in G$, $G$ acts on itself isomorphically by conjugation $a\mapsto gag^{-1}\ \forall a\in G$, which induces the adjoint representation $\Ad_g:\mathfrak{g}\to\mathfrak{g}$ acting automorphically on $\mathfrak{g}$. By restriction of this adjoint action, $\mathfrak{g}$ is a representation of $H$ as well.

Suppose we have a manifold $M$ and a principal bundle $\pi:\mathcal{B}\to M$ with structure group $H$. For $h\in H$, we let $R_h:\mathcal{B}\to\mathcal{B}$ denote the right principal action of $h$ on the fibers of $\mathcal{B}$. In particular, the vertical bundle $\ker\pi_*\subset T\mathcal{B}$ is trivialized by fundamental vector fields $\zeta_X$ associated to $X\in\mathfrak{h}$, where the value at $u\in\mathcal{B}$ of $\zeta_X$ is $\left.\frac{d}{d{\tt t}}\right|_{{\tt t}=0}R_{\exp({\tt t}X)}(u)$. The bundle $\pi:\mathcal{B}\to M$ defines a \emph{Cartan geometry} of type $(G,H)$ if it admits a \emph{Cartan connection}:
\begin{definition}
 A Cartan connection is a $\mathfrak{g}$-valued one form $\omega\in\Omega^1(\mathcal{B},\mathfrak{g})$ which satisfies:
 \begin{itemize}
  \item $\omega:T_u\mathcal{B}\to \mathfrak{g}$ is a linear isomorphism for every $u\in \mathcal{B}$,
  \item $\omega(\zeta_X)=X$ for every $X\in\mathfrak{h}$,
  \item $R_h^*\omega =\Ad_{h^{-1}}\circ \omega$ for every $h\in H$.
 \end{itemize}
\end{definition}
The purpose of this section is to prove the following 
\begin{prop}
For $\mathcal{B}=B_4^{(1)}$ and $G=SU_\star$, the bundles $\hat{\pi}:B_4^{(1)}\to B_4$ and $\pi:B_4^{(1)}\to M$ are principal bundles with structure groups isomorphic to $H=P_\star^2$ and $H=P_\star$, respectively -- c.f. \S\ref{homogmodel}. The $\mathfrak{su}_\star$-valued parallelism $\omega$ constructed in the previous section defines a Cartan connection for the former bundle, but not the latter.
\end{prop}

By construction, $\omega$ satisfies the first property of a Cartan connection, and the fundamental vector fields are spanned by vertical vector fields dual to the pseudoconnection forms that are vertical for $\hat{\pi}$ or $\pi$, so it remains to determine if $\omega$ satisfies the final, equivariancy condition. In the process, we confirm the first statement of the proposition when we realize a local trivialization of the bundle $\pi:B_4^{(1)}\to M$ via those of the bundles $\hat{\pi}:B_4^{(1)}\to B_4$ and $\underline{\pi}:B_4\to M$.

Let $\mathfrak{g}_4$ be the Lie algebra of $G_4$. We know that $G_4\subset GL(V)$, so $\mathfrak{g}_4\subset V\otimes V^*$ and we can define $\mathfrak{g}_4^{(1)}$ to be the kernel in $\mathfrak{g}_4\otimes V^*$ of the skew-symmetrization map $V\otimes V^*\otimes V^*\to V\otimes \Lambda^2V^*$. This abelian group parameterizes the ambiguity in the pseudoconnection forms on $B_4$ (c.f. \cite[\S3.1.2]{BryantGriffithsGrossman}). In particular, if we write $\underline{\eta}\in\Omega^1(B_4,V)$ for the tautological form on $B_4$ and use underlines to indicate a coframing of $B_4$ which satisfies the structure equations \eqref{B4SE}, we have a local trivialization $B_4^{(1)}\cong\mathfrak{g}_4^{(1)}\times B_4$ as all coframings of $B_4$ which satisfy the structure equations:
\begin{align}\label{g41xB4}
\left[\begin{array}{ccccccccc}
1&0&0&0&0&0&0&0&0\\       
0&1&0&0&0&0&0&0&0\\ 
0&0&1&0&0&0&0&0&0\\ 
0&0&0&1&0&0&0&0&0\\ 
0&0&0&0&1&0&0&0&0\\ 
0&0&0&0&0&1&0&0&0\\ 
y&0&0&0&0&0&1&0&0\\ 
0&y&0&0&0&0&0&1&0\\ 
0&0&y&0&0&0&0&0&1\end{array}\right] 
\left[\begin{array}{c}
\underline{\eta}^0\\\underline{\eta}^1\\\underline{\eta}^2\\\underline{\eta}^3\\\underline{\varrho}\\\underline{\varsigma}\\\underline{\tau}\\\underline{\gamma}^1\\\underline{\gamma}^2\end{array}\right].
\end{align} 

We abbreviate the coframing \eqref{g41xB4} by $\eta_y\in B_4^{(1)}$, and we let $\eta_+$ denote the column vector \eqref{prolongtaut} of tautological forms on $B_4^{(1)}$. With this notation we can concisely say 
\begin{align*}
 \eta_+=\hat{\pi}^*\eta_y.
\end{align*}
For fixed $\check{y}\in\mathbb{R}$, let $\check{g}\in\mathfrak{g}_4^{(1)}$ be the group element represented by the matrix \eqref{g41xB4} where the fiber coordinate $y\in C^\infty(B_4^{(1)})$ equals $\check{y}$. The right principal $\mathfrak{g}_4^{(1)}$-action $R_{\check{g}}:B_4^{(1)}\to B_4^{(1)}$ is simply given by matrix multiplication 
\begin{align*}
R_{\check{g}}:\eta_y\mapsto \check{g}^{-1}\eta_y=\eta_{y-\check{y}}.
\end{align*}
Thus, the pullback $R^*_{\check{g}}:T^*_{\eta_{y-\check{y}}}B_4^{(1)}\to T^*_{\eta_{y}}B_4^{(1)}$ of the tautological forms along this principal action is also given by matrix multiplication
\begin{align*}
R^*_{\check{g}}\eta_+= \check{g}^{-1}\eta_+.
\end{align*}
More explicitly, 
\begin{align}\label{g41tautpullback}
R^*_{\check{g}}\left[\begin{array}{c}
\eta^0 \\\eta^1 \\\eta^2 \\\eta^3 \\\varrho \\\varsigma \\\tau \\\gamma^1 \\\gamma^2 \end{array}\right]= 
\left[\begin{array}{c}
\eta^0 \\\eta^1 \\\eta^2 \\\eta^3 \\\varrho \\\varsigma \\\tau-\check{y}\eta^0 \\\gamma^1-\check{y}\eta^1 \\\gamma^2-\check{y}\eta^2 \end{array}\right].
\end{align}

It remains to determine $R^*_{\check{g}}\psi$, for which we enlist the help of the structure equations \eqref{SE} of $B_4^{(1)}$. We differentiate the equation
\begin{align*}
 R_{\check{g}}^*(\tau)=\tau-\check{y}\eta^0
\end{align*}
and use \eqref{g41tautpullback} to conclude  
\begin{align*}
 -R_{\check{g}}^*(\psi)\wedge\eta^0=-(\psi-2\check{y}\tau)\wedge\eta^0,
\end{align*}
whence we see that 
\begin{align*}
 R_{\check{g}}^*(\psi)\equiv\psi-2\check{y}\tau\mod\{\eta^0\}.
\end{align*}
Let us therefore write 
\begin{align*}
R_{\check{g}}^*(\psi)=\psi-2\check{y}\tau+a\eta^0 
\end{align*}
for some $a\in\mathbb{R}$ and differentiate again, this time reducing by $\eta^0,\eta^2,\eta^3,\eta^{\overline{2}},\eta^{\overline{3}}$ to get
\begin{align*}
0&\equiv\tfrac{1}{2}(R_{\check{g}}^*(\overline{F}^3_2)-\overline{F}^3_2)\gamma^1\wedge\eta^1+\tfrac{1}{2}(R_{\check{g}}^*(\overline{F}^3_2)-\overline{F}^3_2)\gamma^{\overline{1}}\wedge\eta^{\overline{1}}-\im(a-\check{y}^2)\eta^1\wedge\eta^{\overline{1}}
&\mod\{\eta^0,\eta^2,\eta^3,\eta^{\overline{2}},\eta^{\overline{3}}\}.
\end{align*}
Thus we conclude 
\begin{align*}
 R_{\check{g}}^*(\psi)=\psi-2\check{y}\tau+\check{y}^2\eta^0,
\end{align*}
which along with \eqref{g41tautpullback} shows 
\begin{align}\label{omegapullback}
R_{\check{g}}^*\omega=\left[\begin{array}{cccc}
 -(\tau-\check{y}\eta^0)-\im\tfrac{1}{4}\varrho+\im\tfrac{1}{4}\varsigma        &-\im(\gamma^2-\check{y}\eta^2)       &-\im(\gamma^{\overline{1}}-\check{y}\eta^{\overline{1}})        &-\im(\psi-2\check{y}\tau+\check{y}^2\eta^0)\\ 
 & & & \\
 -\epsilon\eta^{\overline{2}} &-\im\tfrac{1}{4}\varrho-\im\tfrac{3}{4}\varsigma  &\epsilon\eta^{\overline{3}}  &-\epsilon\im(\gamma^{\overline{2}}-\check{y}\eta^{\overline{2}})\\
 & & & \\
 \eta^1        &\eta^3      &\im\tfrac{3}{4}\varrho+\im\tfrac{1}{4}\varsigma    &\im(\gamma^1-\check{y}\eta^1)\\
 & & & \\
 -\im\eta^0       &\eta^2      &\eta^{\overline{1}}  &(\tau-\check{y}\eta^0)-\im\tfrac{1}{4}\varrho+\im\tfrac{1}{4}\varsigma
 \end{array}\right].
\end{align}

It is clear that $\mathfrak{g}_4^{(1)}$ is isomorphic to $P_\star^2$ as they are both one-dimensional, abelian Lie groups. We formally define an isomorphism $\varphi:\mathfrak{g}_4^{(1)}\to P_\star^2$ by mapping the element represented by the inverse of the matrix \eqref{g41xB4} to the $P_\star^2$ matrix in \eqref{Pgroups}. In particular,
\begin{align*}
 \varphi(\check{g}^{-1})=\left[\begin{array}{cccc}
1&0&0&\im \check{y}\\
0&1&0 &0\\
0&0&1&0\\
0&0&0&1\end{array}\right], 
\end{align*}
so it is straightforward to check that $\Ad_{\varphi(\check{g}^{-1})}\circ\omega$ agrees with the matrix \eqref{omegapullback}. Thus we have shown that $\hat{\pi}:B_4^{(1)}\to B_4$ is a principal $P_\star^2$-bundle for which $\omega\in\Omega^1(B_4^{(1)},\mathfrak{su}_\star)$ is a Cartan connection.

Recall that the bundle $\underline{\pi}:B_4\to M$ from \S\ref{last2} is locally trivialized as $B_4\cong G_4\times M$ by fixing a 4-adapted coframing $\theta_\mathbbm{1}$ of $M$. This trivialization parameterizes local 4-adapted coframings by $g^{-1}\theta_\mathbbm{1}$ where $g^{-1}$ is the matrix \eqref{G4}. Furthermore, the tautological forms on $B_4$ have the local expression 
\begin{align*}
\left[\begin{array}{c}\underline{\eta}^0\\ \underline{\eta}^1\\\underline{\eta}^2\\\underline{\eta}^3\end{array}\right]=
\left[\begin{array}{cccc}t^2&0&0&0\\c^1& t\e^{\im r}&0&0\\c^2&0&t\e^{\im s}&0\\\tfrac{\im}{t^2}c^1c^2&\tfrac{\im}{t}\e^{\im r}c^2&\tfrac{\im}{t}\e^{\im s}c^1&\e^{\im(r+s)}\end{array}\right]
\left[\begin{array}{c}\underline{\pi}^*\theta^0_\mathbbm{1}\\ \underline{\pi}^*\theta^1_\mathbbm{1}\\\underline{\pi}^*\theta^2_\mathbbm{1}\\\underline{\pi}^*\theta^3_\mathbbm{1}\end{array}\right];\hspace{0.5cm} r,s,0\neq t\in C^\infty(B_4); c^1,c^2\in C^\infty(B_4,\mathbb{C}) ,
\end{align*}
As such, the coframing $\eta_y$ of $B_4$ in \eqref{g41xB4} above may be expanded 
\begin{align}\label{G41}
\left[\begin{array}{cccccccccc}t^2&0&0&0&0&0&0&0&0\\
c^1& t\e^{\im r}&0&0&0&0&0&0&0\\
c^2&0&t\e^{\im s}&0&0&0&0&0&0\\
\tfrac{\im}{t^2}c^1c^2&\tfrac{\im}{t}\e^{\im r}c^2&\tfrac{\im}{t}\e^{\im s}c^1&\e^{\im(r+s)}&0&0&0&0&0\\
0&0&0&0&1&0&0&0&0\\
0&0&0&0&0&1&0&0&0\\ 
yt^2&0&0&0&0&0&1&0&0\\
yc^1& yt\e^{\im r}&0&0&0&0&0&1&0\\
yc^2&0&yt\e^{\im s}&0&0&0&0&0&1
\end{array}\right]
\left[\begin{array}{c}
\underline{\pi}^*\theta^0_\mathbbm{1}\\ \underline{\pi}^*\theta^1_\mathbbm{1}\\\underline{\pi}^*\theta^2_\mathbbm{1}\\\underline{\pi}^*\theta^3_\mathbbm{1}\\\underline{\varrho}\\\underline{\varsigma}\\\underline{\tau}\\\underline{\gamma}^1\\\underline{\gamma}^2\end{array}\right],
\end{align}
and this defines a local trivialization of the bundle $\pi:B_4^{(1)}\to M$ as $B_4^{(1)}\cong G_4^{(1)}\times M$ where the structure group $G_4^{(1)}\cong\mathfrak{g}_4^{(1)}\times G_4$ is parameterized as shown. We extend the isomorphism $\varphi$ above to an isomorphism $G_4^{(1)}\to P_\star$ by mapping the inverse of the matrix \eqref{G41} to the matrix \eqref{Pstar}. In this way we realize $\pi:B_4^{(1)}\to M$ as a principal $P_\star$-bundle over $M$.

We need not attempt to verify the equivariancy condition on this bundle; $\omega$ cannot be a Cartan connection for $\pi:B_4^{(1)}\to M$ since the curvature tensor $C$ given by \eqref{curvature} is not $\pi$-semibasic; see \cite[Lem 1.5.1]{CapSlovak}.

%%%%%%%%%%%%%%%%%%%%%%%%%%%%%%%%%%%%%%%%%%%%%%%%%%%%%%%%%%%%%%%%%%%%%%%%%%%%%%%%%%
\vspace{\baselineskip}\subsection{A Non-Flat Example}\label{nonflat}
%%%%%%%%%%%%%%%%%%%%%%%%%%%%%%%%%%%%%%%%%%%%%%%%%%%%%%%%%%%%%%%%%%%%%%%%%%%%%%%%%%

Recall from \S\ref{FundInv} that a necessary and sufficient condition for a 2-non-degenerate CR manifold $M$ to be locally CR equivalent to the homogeneous model $M_\star$ is that the coefficients $F^1,F^2$ of the fundamental invariants \eqref{fundinv} vanish. We saw that this implies the curvature tensor $C$ as in \eqref{curvature} is trivial, and such $M$ is therefore called flat. To demonstrate the existence of non-flat $M$, we consider $\mathbb{C}^4$ with complex coordinates $\{z^i,z^{\overline{i}}\}_{i=1}^4$, and let $M$ be the hypersurface given by the level set $\rho^{-1}(0)$ of a smooth function $\rho:\mathbb{C}^4\to\mathbb{R}$ whose partial derivatives do not all vanish. In this setting, we can take the contact form $\theta^0\in\Omega^1(M)$ to be 
\begin{align}\label{contactformrho}
 \theta^0:=-\im\partial\rho=-\im\frac{\partial\rho}{\partial z^i}\dd z^i.
\end{align}

After a change of coordinates if necessary, the equation $\rho=0$ may be written 
\begin{align*}
 F(z^1,z^2,z^3,z^{\overline{1}},z^{\overline{2}},z^{\overline{3}})=z^4+z^{\overline{4}},
\end{align*}
for $F:\mathbb{C}^3\to\mathbb{R}$, and the forms $\dd z^j,\dd z^{\overline{j}}$ ($1\leq j\leq3$) complete $\theta^0$ to a local coframing of $M$. In the simplified case that $F$ is given by 
\begin{align*}
 F(z^1,z^2,z^3,z^{\overline{1}},z^{\overline{2}},z^{\overline{3}})=f(z^1+z^{\overline{1}},z^2+z^{\overline{2}},z^3+z^{\overline{3}})
\end{align*}
for some $f:\mathbb{R}^3\to\mathbb{R}$, we have 
\begin{align*}
 F_j:=\frac{\partial F}{\partial z^j}=\frac{\partial F}{\partial z^{\overline{j}}}=:F_{\overline{j}},
\end{align*}
and we denote their common expression by $f_j$. Thus, \eqref{contactformrho} may be written  
\begin{align}\label{contactformf}
 \theta^0=-\im f_j\dd z^j+\im\dd z^4.
\end{align}
Second order partial derivatives are indicated by two subscripts, so that differentiating \eqref{contactformf} gives the following matrix representation of the Levi form of $M$ with respect to the coframing $\{\dd z^j,\dd z^{\overline{j}}\}_{j=1}^3$: 
\begin{align*}
 \left[\begin{array}{ccc}f_{11}&f_{12}&f_{13}\\f_{12}&f_{22}&f_{23}\\f_{13}&f_{23}&f_{33}\end{array}\right].
\end{align*}

If we impose the condition that $f_{12}=0$ while all other $f_{jk}$ are nonvanishing, then Levi-degeneracy is equivalent to the partial differential equation 
\begin{align}\label{detzero}
 0=\det(f_{jk})=f_{11}f_{22}f_{33}-f_{11}(f_{23})^2-f_{22}(f_{13})^2,
\end{align}
which is satisfied, for example, when 
\begin{align}\label{detzerosatisfied}
 &(f_{23})^2=\tfrac{1}{2}f_{22}f_{33},
 &(f_{13})^2=\tfrac{1}{2}f_{11}f_{33}.
\end{align}
We further assume that $f_{jj}>0$ for $j=1,2,3$, so that when \eqref{detzerosatisfied} holds, $f_{k3}=\pm\sqrt{\tfrac{1}{2}f_{kk}f_{33}}$ for $k=1,2$, and the coframing given by 
\begin{align}\label{diaglevi}
\left[\begin{array}{cccc}\theta^0\\\theta^1\\\theta^2\\\theta^3\end{array}\right]=
\left[\begin{array}{cccc}1&0&0&0\\0&\sqrt{f_{11}}&0&\pm\sqrt{\tfrac{1}{2}f_{33}}\\0&0&\sqrt{f_{22}}&\pm\sqrt{\tfrac{1}{2}f_{33}}\\0&0&0&1\end{array}\right]
\left[\begin{array}{cccc}\theta^0\\\dd z^1\\\dd z^2\\\dd z^3\end{array}\right]
\end{align}
diagonalizes the Levi form,
\begin{align*}
 \dd \theta^0=\im\theta^1\wedge\theta^{\overline{1}}+\im\theta^2\wedge\theta^{\overline{2}}.
\end{align*}

We will compute the structure equations for a concrete example: let $x_1,x_2,x_3$ be coordinates for $\mathbb{R}^3$ and take $\mathbb{R}_+^3$ to be the subspace where all coordinates are strictly positive. Define 
\begin{align}\label{nonflatf}
 f(x_1,x_2,x_3)=-x_3\ln\left(\frac{x_1x_2}{(x_3)^2}\right).
\end{align}
In the sequel, we will continue to denote $x_j=z^j+z^{\overline{j}}$ in order to compactify notation. Thus, \eqref{contactformf} is given by 
\begin{align*}
 \theta^0=\im\frac{x_3}{x_1}\dd z^1+\im\frac{x_3}{x_2}\dd z^2+\im\left(\ln\left(\frac{x_1x_2}{(x_3)^2}\right)-2\right)\dd z^3+\im\dd z^4,
\end{align*}
and our first approximation \eqref{diaglevi} at an adapted coframing is 
\begin{align}\label{1approx}
\left[\begin{array}{cccc}\theta^0\\\theta^1\\\theta^2\\\theta^3\end{array}\right]=
\left[\begin{array}{cccc}1&0&0&0\\0&\tfrac{\sqrt{x_3}}{x_1}&0&-\tfrac{1}{\sqrt{x_3}}\\0&0&\tfrac{\sqrt{x_3}}{x_2}&-\tfrac{1}{\sqrt{x_3}}\\0&0&0&1\end{array}\right]
\left[\begin{array}{cccc}\theta^0\\\dd z^1\\\dd z^2\\\dd z^3\end{array}\right].
\end{align}
We differentiate to determine the structure equations so far,
\begin{equation}\label{1approxSE}
\begin{aligned}
\dd\theta^0&=\im\theta^1\wedge\theta^{\overline{1}}+\im\theta^2\wedge\theta^{\overline{2}},\\
\dd\theta^1&=\tfrac{1}{x_3}\theta^3\wedge\theta^{\overline{1}}+\tfrac{1}{\sqrt{x_3}}\theta^1\wedge\theta^{\overline{1}}-\tfrac{1}{2x_3}\theta^1\wedge\theta^3+\tfrac{1}{2x_3}\theta^1\wedge\theta^{\overline{3}},\\
\dd\theta^2&=\tfrac{1}{x_3}\theta^3\wedge\theta^{\overline{2}}+\tfrac{1}{\sqrt{x_3}}\theta^2\wedge\theta^{\overline{2}}-\tfrac{1}{2x_3}\theta^2\wedge\theta^3+\tfrac{1}{2x_3}\theta^2\wedge\theta^{\overline{3}},\\
\dd\theta^3&=0.
\end{aligned}
\end{equation}
Recall that the structure group $G_0$ of all 0-adapted coframings is parameterized by \eqref{G0}, and that the subgroup $G_1$ which preserves 1-adaptation is given by the additional conditions \eqref{G1}. The structure equations \eqref{1approxSE} show that our coframing is 1-adapted as in \eqref{MSE1}, and we maintain this property when we submit it to a $G_1$-transformation to get a new coframing   
\begin{align}\label{2approx}
\left[\begin{array}{cccc}\eta^0\\\theta^{1'}\\\theta^{2'}\\\theta^{3'}\end{array}\right]=
\left[\begin{array}{rrrr}2&0&0&0\\0&1&\im&0\\0&1&-\im&0\\0&0&0&\tfrac{1}{x_3}\end{array}\right]
\left[\begin{array}{cccc}\theta^0\\\theta^1\\\theta^2\\\theta^3\end{array}\right].
\end{align}
The new structure equations are 
\begin{equation}\label{2approxSE}
\begin{aligned}
\dd\eta^0&=\im\theta^{1'}\wedge\theta^{\overline{1}'}+\im\theta^{2'}\wedge\theta^{\overline{2}'},\\
\dd\theta^{1'}&=\theta^{3'}\wedge\theta^{\overline{2}'}+\tfrac{1+\im}{4\sqrt{x_3}}\theta^{1'}\wedge\theta^{\overline{1}'}+\tfrac{1-\im}{4\sqrt{x_3}}\theta^{1'}\wedge\theta^{\overline{2}'}+\tfrac{1-\im}{4\sqrt{x_3}}\theta^{2'}\wedge\theta^{\overline{1}'}+\tfrac{1+\im}{4\sqrt{x_3}}\theta^{2'}\wedge\theta^{\overline{2}'}+\tfrac{1}{2}\theta^{1'}\wedge(\theta^{\overline{3}'}-\theta^{3'}),\\
\dd\theta^{2'}&=\theta^{3'}\wedge\theta^{\overline{1}'}+\tfrac{1-\im}{4\sqrt{x_3}}\theta^{1'}\wedge\theta^{\overline{1}'}+\tfrac{1+\im}{4\sqrt{x_3}}\theta^{1'}\wedge\theta^{\overline{2}'}+\tfrac{1+\im}{4\sqrt{x_3}}\theta^{2'}\wedge\theta^{\overline{1}'}+\tfrac{1-\im}{4\sqrt{x_3}}\theta^{2'}\wedge\theta^{\overline{2}'}+\tfrac{1}{2}\theta^{2'}\wedge(\theta^{\overline{3}'}-\theta^{3'}),\\
\dd\theta^{3'}&=\theta^{3'}\wedge\theta^{\overline{3}'},
\end{aligned}
\end{equation}
so our coframing \eqref{2approx} is now 2-adapted according to \eqref{MSE2}. The structure group $G_2$ of the bundle of 2-adapted coframes is parameterized by \eqref{G2}, so our 2-adaptation is preserved when we apply a $G_2$ transformation to get a new coframing  
\begin{align}\label{3approx}
\left[\begin{array}{cccc}\eta^0\\\eta^{1}\\\eta^{2}\\\theta^{3''}\end{array}\right]=
\left[\begin{array}{rrrr}1&0&0&0\\c^1&1&0&0\\c^2&0&1&0\\0&b_1&b_2&1\end{array}\right]
\left[\begin{array}{cccc}\eta^0\\\theta^{1'}\\\theta^{2'}\\\theta^{3'}\end{array}\right],
\end{align}
for some $c^1,c^2,b_1,b_2\in C^\infty(M,\mathbb{C})$. The effect of this transformation on the first three structure equations may be written 
\begin{equation}\label{3approxSEmod}
\begin{aligned}
\dd\eta^0&=\im\eta^1\wedge\eta^{\overline{1}}+\im\eta^{2}\wedge\eta^{\overline{2}}+\im\eta^0\wedge(\overline{c}^1\eta^1-c^1\eta^{\overline{1}}+\overline{c}^2\eta^2-c^2\eta^{\overline{2}}),\\
\dd\eta^{1}&\equiv\eta^{3}\wedge\eta^{\overline{2}}+\tfrac{b_2}{2}\eta^1\wedge\eta^2-\tfrac{2\sqrt{x_3}(\overline{b}_1-2\im c^1)-1-\im}{4\sqrt{x_3}}\eta^1\wedge\eta^{\overline{1}}-\tfrac{2\sqrt{x_3}(2b_1+\overline{b}_2)-1+\im}{4\sqrt{x_3}}\eta^1\wedge\eta^{\overline{2}}\\
&+\tfrac{1-\im}{4\sqrt{x_3}}\eta^2\wedge\eta^{\overline{1}}+\tfrac{4\sqrt{x_3}(\im c^1-b_2)+1+\im}{4\sqrt{x_3}}\eta^2\wedge\eta^{\overline{2}}+\tfrac{1}{2}\eta^{1}\wedge(\theta^{\overline{3}''}-\theta^{3''}) &&\mod\{\eta^0\},\\
\dd\eta^2&\equiv\eta^{3}\wedge\eta^{\overline{1}}-\tfrac{b_1}{2}\eta^1\wedge\eta^2+\tfrac{4\sqrt{x_3}(\im c^2-b_1)+1-\im}{4\sqrt{x_3}}\eta^1\wedge\eta^{\overline{1}}-\tfrac{2\sqrt{x_3}(\overline{b}_1+2b_2)-1-\im}{4\sqrt{x_3}}\eta^2\wedge\eta^{\overline{1}}\\
&+\tfrac{1+\im}{4\sqrt{x_3}}\eta^1\wedge\eta^{\overline{2}}-\tfrac{2\sqrt{x_3}(\overline{b}_2-2\im c^2)-1+\im}{4\sqrt{x_3}}\eta^2\wedge\eta^{\overline{2}}+\tfrac{1}{2}\eta^{2}\wedge(\theta^{\overline{3}''}-\theta^{3''}) &&\mod\{\eta^0\}.
\end{aligned}
\end{equation}
We choose functions $b,c$ that eliminate the coefficients of $\eta^1\wedge\eta^{\overline{1}}$ and $\eta^2\wedge\eta^{\overline{2}}$ in the identities for $\dd\eta^1,\dd\eta^2$ in \eqref{3approxSEmod}. Therefore, set 
\begin{align*}
 &c^1:=\frac{-1+\im}{4\sqrt{x_3}},
 &c^2:=\frac{1+\im}{4\sqrt{x_3}},
 &&b_1=b_2=0.
\end{align*}
Now we have 
\begin{equation}\label{3approxSE}
\begin{aligned}
\dd\eta^0&=\im\eta^1\wedge\eta^{\overline{1}}+\im\eta^{2}\wedge\eta^{\overline{2}}+\tfrac{1}{4\sqrt{x_3}}\eta^0\wedge((1-\im)\eta^1+(1+\im)\eta^{\overline{1}}+(1+\im)\eta^2+(1-\im)\eta^{\overline{2}}),\\
\dd\eta^{1}&=\eta^{3}\wedge\eta^{\overline{2}}+\tfrac{1-\im}{4\sqrt{x_3}}\eta^1\wedge\eta^{\overline{2}}+\tfrac{1-\im}{4\sqrt{x_3}}\eta^2\wedge\eta^{\overline{1}}+\tfrac{1}{2}\eta^{1}\wedge(\theta^{\overline{3}''}-\theta^{3''})-\tfrac{1}{8x_3}\eta^0\wedge(\im\eta^1+\eta^{\overline{1}}+\eta^2+\im\eta^{\overline{2}}) ,\\
\dd\eta^2&=\eta^{3}\wedge\eta^{\overline{1}}+\tfrac{1+\im}{4\sqrt{x_3}}\eta^2\wedge\eta^{\overline{1}}+\tfrac{1+\im}{4\sqrt{x_3}}\eta^1\wedge\eta^{\overline{2}}+\tfrac{1}{2}\eta^{2}\wedge(\theta^{\overline{3}''}-\theta^{3''})+\tfrac{1}{8x_3}\eta^0\wedge(\eta^1-\im\eta^{\overline{1}}-\im\eta^2+\eta^{\overline{2}})  ,\\
\dd\theta^{3''}&=\theta^{3''}\wedge\theta^{\overline{3}''}.
\end{aligned}
\end{equation}
Finally, we apply a $G_3$-transformation -- see \eqref{G3} -- to get 
\begin{align}\label{4approx}
\left[\begin{array}{cccc}\eta^0\\\eta^{1}\\\eta^{2}\\\eta^{3}\end{array}\right]=
\left[\begin{array}{rrrr}1&0&0&0\\0&1&0&0\\0&0&1&0\\c^3&0&0&1\end{array}\right]
\left[\begin{array}{cccc}\eta^0\\\eta^{1}\\\eta^{2}\\\theta^{3''}\end{array}\right],
\end{align}
which effects the following alteration of the latter three structure equations \eqref{3approxSE}
\begin{equation}\label{4approxSEc3}
\begin{aligned}
\dd\eta^{1}&=\eta^{3}\wedge\eta^{\overline{2}}+\tfrac{1-\im}{4\sqrt{x_3}}\eta^1\wedge\eta^{\overline{2}}+\tfrac{1-\im}{4\sqrt{x_3}}\eta^2\wedge\eta^{\overline{1}}+\tfrac{1}{2}\eta^{1}\wedge(\theta^{\overline{3}''}-\theta^{3''})\\
&+\tfrac{1}{8x_3}\eta^0\wedge((4x_3(\overline{c}^3-c^3)-\im)\eta^1-\eta^{\overline{1}}-\eta^2-(8x_3c^3+\im)\eta^{\overline{2}}) ,\\
\dd\eta^2&=\eta^{3}\wedge\eta^{\overline{1}}+\tfrac{1+\im}{4\sqrt{x_3}}\eta^2\wedge\eta^{\overline{1}}+\tfrac{1+\im}{4\sqrt{x_3}}\eta^1\wedge\eta^{\overline{2}}+\tfrac{1}{2}\eta^{2}\wedge(\theta^{\overline{3}''}-\theta^{3''})\\
&+\tfrac{1}{8x_3}\eta^0\wedge(\eta^1-(8x_3c^3+\im)\eta^{\overline{1}}+(4x_3(\overline{c}^3-c^3)-\im)\eta^2+\eta^{\overline{2}})  ,\\
\dd\eta^3&\equiv \eta^3\wedge\eta^{\overline{3}}+\im c^3\eta^{1}\wedge\eta^{\overline{1}}+\im c^3\eta^{2}\wedge\eta^{\overline{2}} &&\mod\{\eta^0\}.
\end{aligned}
\end{equation}
If we take 
\begin{equation*}
 \begin{aligned}
\gamma^1&\equiv \tfrac{1}{8x_3}((4x_3(\overline{c}^3-c^3)-\im)\eta^1-\eta^{\overline{1}}-\eta^2-(8x_3c^3+\im)\eta^{\overline{2}})\mod\{\eta^0\},\\
\gamma^2&\equiv \tfrac{1}{8x_3}(\eta^1-(8x_3c^3+\im)\eta^{\overline{1}}+(4x_3(\overline{c}^3-c^3)-\im)\eta^2+\eta^{\overline{2}})\mod\{\eta^0\},
 \end{aligned}
\end{equation*}
then we can equivalently express \eqref{4approxSEc3} as 
\begin{equation}\label{4approxSEgammas}
\begin{aligned}
\dd\eta^{1}&=-\gamma^1\wedge\eta^0+\eta^{3}\wedge\eta^{\overline{2}}+\tfrac{1-\im}{4\sqrt{x_3}}\eta^1\wedge\eta^{\overline{2}}+\tfrac{1-\im}{4\sqrt{x_3}}\eta^2\wedge\eta^{\overline{1}}+\tfrac{1}{2}\eta^{1}\wedge(\theta^{\overline{3}''}-\theta^{3''}),\\
\dd\eta^2&=-\gamma^2\wedge\eta^0+\eta^{3}\wedge\eta^{\overline{1}}+\tfrac{1+\im}{4\sqrt{x_3}}\eta^2\wedge\eta^{\overline{1}}+\tfrac{1+\im}{4\sqrt{x_3}}\eta^1\wedge\eta^{\overline{2}}+\tfrac{1}{2}\eta^{2}\wedge(\theta^{\overline{3}''}-\theta^{3''}),\\
\dd\eta^3&\equiv -\im\gamma^2\wedge\eta^1-\im\gamma^1\wedge\eta^2+\eta^3\wedge\eta^{\overline{3}}+\tfrac{\im16x_3c^3-1}{8x_3}(\eta^{1}\wedge\eta^{\overline{1}}+\eta^{2}\wedge\eta^{\overline{2}})\\
&-\tfrac{\im}{8x_3}\eta^1\wedge\eta^{\overline{2}}+\tfrac{\im}{8x_3}\eta^2\wedge\eta^{\overline{1}} &&\mod\{\eta^0\}.
\end{aligned}
\end{equation}
We select $c^3$ to eliminate the $\eta^1\wedge\eta^{\overline{1}}$ and $\eta^2\wedge\eta^{\overline{2}}$ terms in the identity \eqref{4approxSEgammas} for $\dd\eta^3$, viz,
\begin{align*}
 c^3:=-\frac{\im}{16x_3}.
\end{align*}

Now the forms $\eta^0,\eta^1,\eta^2,\eta^3$ on $M$ are completely determined. We summarize in terms of our $\mathbb{C}^4$ coordinates $z^1,z^2,z^3,z^4$, whose real parts we assume to be strictly positive (except for $z^4$),   
\begin{align*}
 \eta^0&=2\im\tfrac{z^3+z^{\overline{3}}}{z^1+z^{\overline{1}}}\dd z^1+2\im\tfrac{z^3+z^{\overline{3}}}{z^2+z^{\overline{2}}}\dd z^2+2\im\left(\ln\left(\tfrac{(z^1+z^{\overline{1}})(z^2+z^{\overline{2}})}{(z^3+z^{\overline{3}})^2}\right)-2\right)\dd z^3+2\im\dd z^4,\\
 &\\
 \eta^1&=\tfrac{(1-\im)\sqrt{z^3+z^{\overline{3}}}}{2(z^1+z^{\overline{1}})}\dd z^1-\tfrac{(1-\im)\sqrt{z^3+z^{\overline{3}}}}{2(z^2+z^{\overline{2}})}\dd z^2-\tfrac{(1+\im)}{2\sqrt{z^3+z^{\overline{3}}}}\ln\left(\tfrac{(z^1+z^{\overline{1}})(z^2+z^{\overline{2}})}{(z^3+z^{\overline{3}})^2}\right)\dd z^3-\tfrac{1+\im}{2\sqrt{z^3+z^{\overline{3}}}}\dd z^4,\\
 &\\
 \eta^2&=\tfrac{(1+\im)\sqrt{z^3+z^{\overline{3}}}}{2(z^1+z^{\overline{1}})}\dd z^1-\tfrac{(1+\im)\sqrt{z^3+z^{\overline{3}}}}{2(z^2+z^{\overline{2}})}\dd z^2-\tfrac{(1-\im)}{2\sqrt{z^3+z^{\overline{3}}}}\ln\left(\tfrac{(z^1+z^{\overline{1}})(z^2+z^{\overline{2}})}{(z^3+z^{\overline{3}})^2}\right)\dd z^3-\tfrac{1-\im}{2\sqrt{z^3+z^{\overline{3}}}}\dd z^4,\\
 &\\
 \eta^3&=\tfrac{1}{8(z^1+z^{\overline{1}})}\dd z^1+\tfrac{1}{8(z^2+z^{\overline{2}})}\dd z^2+\frac{6+\ln\left(\tfrac{(z^1+z^{\overline{1}})(z^2+z^{\overline{2}})}{(z^3+z^{\overline{3}})^2}\right)}{8(z^3+z^{\overline{3}})}\dd z^3+\tfrac{1}{8(z^3+z^{\overline{3}})}\dd z^4.
\end{align*}

The structure equations for these forms are 
\begin{equation}\label{4approxSE}
 \begin{aligned}
\dd\eta^0&=\im\eta^1\wedge\eta^{\overline{1}}+\im\eta^{2}\wedge\eta^{\overline{2}}+\tfrac{1}{4\sqrt{x_3}}\eta^0\wedge((1-\im)\eta^1+(1+\im)\eta^{\overline{1}}+(1+\im)\eta^2+(1-\im)\eta^{\overline{2}}),\\ 
\dd\eta^{1}&=\eta^{3}\wedge\eta^{\overline{2}}+\tfrac{1-\im}{4\sqrt{x_3}}\eta^1\wedge\eta^{\overline{2}}+\tfrac{1-\im}{4\sqrt{x_3}}\eta^2\wedge\eta^{\overline{1}}+\tfrac{1}{2}\eta^{1}\wedge(\eta^{\overline{3}}-\eta^{3})-\tfrac{1}{16x_3}\eta^0\wedge(\im\eta^1+2\eta^{\overline{1}}+2\eta^2+\im\eta^{\overline{2}}) ,\\
\dd\eta^2&=\eta^{3}\wedge\eta^{\overline{1}}+\tfrac{1+\im}{4\sqrt{x_3}}\eta^2\wedge\eta^{\overline{1}}+\tfrac{1+\im}{4\sqrt{x_3}}\eta^1\wedge\eta^{\overline{2}}+\tfrac{1}{2}\eta^{2}\wedge(\eta^{\overline{3}}-\eta^{3})+\tfrac{1}{16x_3}\eta^0\wedge(2\eta^1-\im\eta^{\overline{1}}-\im\eta^2+2\eta^{\overline{2}})  ,\\
\dd\eta^3&=\tfrac{1}{64(x_3)^{\nicefrac{3}{2}}}((1+\im)\eta^1-(1-\im)\eta^{\overline{1}}-(1-\im)\eta^2+(1+\im)\eta^{\overline{2}})\wedge\eta^0+\tfrac{1}{16x_3}\eta^1\wedge\eta^{\overline{1}}+\tfrac{1}{16x_3}\eta^2\wedge\eta^{\overline{2}}+\eta^3\wedge\eta^{\overline{3}},
 \end{aligned}
\end{equation}
which shows that the coframing $\eta^0,\eta^1,\eta^2,\eta^3$ of $M$ defines a section of the bundle $B_4\to M$ of 4-adapted coframes. If we denote the pullbacks along this section of the pseudoconnection forms on $B_4$ by their same names, then we write  
\begin{equation}\label{pseudoconnectionB4}
 \begin{aligned}
\tau&=\tfrac{1}{8\sqrt{x_3}}((1-\im)\eta^1+(1+\im)\eta^{\overline{1}}+(1+\im)\eta^2+(1-\im)\eta^{\overline{2}}),\\
\im\varrho&=\tfrac{1}{2}(\eta^{\overline{3}}-\eta^{3})+\tfrac{1}{8\sqrt{x_3}}((1-\im)\eta^1-(1+\im)\eta^{\overline{1}}-(1+\im)\eta^2+(1-\im)\eta^{\overline{2}}),\\
\im\varsigma&=\tfrac{1}{2}(\eta^{\overline{3}}-\eta^{3})-\tfrac{1}{8\sqrt{x_3}}((1-\im)\eta^1-(1+\im)\eta^{\overline{1}}-(1+\im)\eta^2+(1-\im)\eta^{\overline{2}}),\\
\gamma^1&=\tfrac{1+\im}{64(x_3)^{\nicefrac{3}{2}}}\eta^0-\tfrac{1}{16x_3}(\im\eta^1+2\eta^{\overline{1}}+2\eta^2+\im\eta^{\overline{2}}),\\
\gamma^2&=\tfrac{1-\im}{64(x_3)^{\nicefrac{3}{2}}}\eta^0+\tfrac{1}{16x_3}(2\eta^1-\im\eta^{\overline{1}}-\im\eta^2+2\eta^{\overline{2}}),
\end{aligned}
\end{equation}
and the structure equations \eqref{4approxSE} may be written according to \eqref{B4SE}
{\footnotesize
\begin{equation*}
\begin{aligned}
 \dd \left[\begin{array}{c}\eta^0\\\eta^1\\\eta^2\\\eta^3\end{array}\right]&=
 -\left[\begin{array}{cccc}2\tau&0&0&0\\\gamma^1&\tau+\im\varrho&0&0\\\gamma^2&0&\tau+\im\varsigma&0\\0&\im\gamma^2&\im\gamma^1&\im\varrho+\im\varsigma\end{array}\right]\wedge
 \left[\begin{array}{c}\eta^0\\\eta^1\\\eta^2\\\eta^3\end{array}\right] +
 \left[\begin{array}{c}
 \im\eta^1\wedge\eta^{\overline{1}}+\im\eta^2\wedge\eta^{\overline{2}}\\
 \eta^3\wedge\eta^{\overline{2}}+F^1\eta^{\overline{1}}\wedge\eta^2 \\
 \eta^3\wedge\eta^{\overline{1}}+F^2\eta^{\overline{2}}\wedge\eta^1 \\
 T^3_{\overline{1}}\eta^{\overline{1}}\wedge\eta^0+T^3_{\overline{2}}\eta^{\overline{2}}\wedge\eta^0+F^3_1\eta^{\overline{2}}\wedge\eta^1+F^3_2\eta^{\overline{1}}\wedge\eta^2\end{array}\right],
\end{aligned} 
\end{equation*}}
for 
\begin{align}\label{fundcoeff}
&F^1=-\frac{1-\im}{4\sqrt{z^3+z^{\overline{3}}}},
&&F^2=-\frac{1+\im}{4\sqrt{z^3+z^{\overline{3}}}},
\end{align}
\begin{align*}
&T^3_{\overline{1}}=-\frac{1-\im}{64(z^3+z^{\overline{3}})^{\nicefrac{3}{2}}},
&T^3_{\overline{2}}=\frac{1+\im}{64(z^3+z^{\overline{3}})^{\nicefrac{3}{2}}},
&&F^3_1=\frac{\im}{8(z^3+z^{\overline{3}})},
&&F^3_2=-\frac{\im}{8(z^3+z^{\overline{3}})}.
\end{align*}
In particular, the coefficients \eqref{fundcoeff} of the fundamental invariants \eqref{fundinv} are nonvanishing, so $M$ is not locally CR equivalent to the homogeneous model $M_\star$.

At this point, the forms $\eta,\varrho,\varsigma,\tau,\gamma$ on $M$ are adapted to the $B_4$ structure equations, so they define a section of the bundle $B_4^{(1)}\to M$, and they are exactly the pullbacks along this section of the tautological forms with the same names \eqref{prolongtaut} on $B_4^{(1)}$. Thus, to find the pullback of the full parallelism $\omega\in\Omega^1(B_4^{(1)},\mathfrak{su}_\star)$ as in \S\ref{FundInv}, it remains to find an expression for the pullback of $\psi$, which we will also call $\psi$. To accomplish this, we differentiate $\tau$ and $\gamma^1$ according to the structure equations \eqref{SE}. We begin with $\tau$, 
\begin{align*}
\dd\tau&=\tfrac{\im}{2}\gamma^1\wedge\eta^{\overline{1}}-\tfrac{\im}{2}\gamma^{\overline{1}}\wedge\eta^1+\tfrac{\im}{2}\gamma^2\wedge\eta^{\overline{2}}-\tfrac{\im}{2}\gamma^{\overline{2}}\wedge\eta^2\\
&+\frac{1}{128(x_3)^{\nicefrac{3}{2}}}\eta^0\wedge((1+\im)\eta^1+(1-\im)\eta^{\overline{1}}-(1-\im)\eta^2-(1+\im)\eta^{\overline{2}}), 
\end{align*}
so we see 
\begin{align*}
 \psi\equiv\frac{1}{128(x_3)^{\nicefrac{3}{2}}}((1+\im)\eta^1+(1-\im)\eta^{\overline{1}}-(1-\im)\eta^2-(1+\im)\eta^{\overline{2}})\mod\{\eta^0\}.
\end{align*}
To find the coefficient of $\eta^0$ in the full expansion of $\psi$, one takes the real part of the coefficient of $\eta^0\wedge\eta^1$ in the expression 
\begin{align*}
\dd\gamma^1-(\tau-\im\varrho)\wedge\gamma^1+\gamma^{\overline{2}}\wedge\eta^3-\im F^3_2\gamma^{\overline{1}}\wedge\eta^0-F^1\gamma^{\overline{1}}\wedge\eta^2+F^1\gamma^2\wedge\eta^{\overline{1}}. 
\end{align*}
We simply state that the result of this calculation is 
\begin{align*}
 \psi=\frac{1}{128(z^3+z^{\overline{3}})^2}\eta^0+\frac{1}{128(z^3+z^{\overline{3}})^{\nicefrac{3}{2}}}((1+\im)\eta^1+(1-\im)\eta^{\overline{1}}-(1-\im)\eta^2-(1+\im)\eta^{\overline{2}}).
\end{align*}
With this one-form in hand, the pullback of the parallelism $\omega$ to $M$ is completely determined.

\bibliographystyle{amsalpha}

\bibliography{References}

\providecommand{\bysame}{\leavevmode\hbox to3em{\hrulefill}\thinspace}
\providecommand{\MR}{\relax\ifhmode\unskip\space\fi MR }
\providecommand{\MRhref}[2]{%
  \href{http://www.ams.org/mathscinet-getitem?mr=#1}{#2}
}
\providecommand{\href}[2]{#2}
\begin{thebibliography}{{San}15}

\bibitem[AMN10]{CRorbits}
Andrea Altomani, Costantino Medori, and Mauro Nacinovich, \emph{Orbits of real
  forms in complex flag manifolds}, Ann. Sc. Norm. Super. Pisa Cl. Sci. (5)
  \textbf{9} (2010), no.~1, 69--109. \MR{2668874 (2011i:53064)}

\bibitem[BGG03]{BryantGriffithsGrossman}
Robert Bryant, Phillip Griffiths, and Daniel Grossman, \emph{Exterior
  differential systems and {E}uler-{L}agrange partial differential equations},
  Chicago Lectures in Mathematics, University of Chicago Press, Chicago, IL,
  2003. \MR{1985469 (2004g:58001)}

\bibitem[Car33]{CartanCR}
Elie Cartan, \emph{Sur la g\'eom\'etrie pseudo-conforme des hypersurfaces de
  l'espace de deux variables complexes}, Ann. Mat. Pura Appl. \textbf{11}
  (1933), no.~1, 17--90. \MR{1553196}

\bibitem[Chi91]{ChirkaCR}
E.~M. Chirka, \emph{An introduction to the geometry of {CR} manifolds}, Uspekhi
  Mat. Nauk \textbf{46} (1991), no.~1(277), 81--164, 240. \MR{1109037
  (92m:32012)}

\bibitem[CM74]{ChernMoser}
S.~S. Chern and J.~K. Moser, \emph{Real hypersurfaces in complex manifolds},
  Acta Math. \textbf{133} (1974), 219--271. \MR{0425155 (54 \#13112)}

\bibitem[{\v{C}}S09]{CapSlovak}
Andreas {\v{C}}ap and Jan Slov{\'a}k, \emph{Parabolic geometries. {I}},
  Mathematical Surveys and Monographs, vol. 154, American Mathematical Society,
  Providence, RI, 2009, Background and general theory. \MR{2532439
  (2010j:53037)}

\bibitem[Ebe98]{EbCubic}
Peter Ebenfelt, \emph{New invariant tensors in {CR} structures and a normal
  form for real hypersurfaces at a generic {L}evi degeneracy}, J. Differential
  Geom. \textbf{50} (1998), no.~2, 207--247. \MR{1684982 (2000a:32077)}

\bibitem[Ebe01]{Eb5dim}
\bysame, \emph{Uniformly {L}evi degenerate {CR} manifolds: the 5-dimensional
  case}, Duke Math. J. \textbf{110} (2001), no.~1, 37--80. \MR{1861088
  (2002g:32044)}

\bibitem[Ebe06]{EbCorrection}
\bysame, \emph{Correction to: ``{U}niformly {L}evi degenerate {CR} manifolds:
  the 5-dimensional case''}, Duke Math. J. \textbf{131} (2006), no.~3,
  589--591. \MR{2219252 (2007a:32034)}

\bibitem[Fre74]{FreemanComplexFoliations}
Michael Freeman, \emph{Local complex foliation of real submanifolds}, Math.
  Ann. \textbf{209} (1974), 1--30. \MR{0346185 (49 \#10911)}

\bibitem[Fre77a]{FreemanStraightening}
\bysame, \emph{Local biholomorphic straightening of real submanifolds}, Ann. of
  Math. (2) \textbf{106} (1977), no.~2, 319--352. \MR{0463480 (57 \#3429)}

\bibitem[Fre77b]{FreemanLeviDeg}
\bysame, \emph{Real submanifolds with degenerate {L}evi form}, Several complex
  variables ({P}roc. {S}ympos. {P}ure {M}ath., {V}ol. {XXX}, {W}illiams
  {C}oll., {W}illiamstown, {M}ass., 1975), {P}art 1, Amer. Math. Soc.,
  Providence, R.I., 1977, pp.~141--147. \MR{0457767 (56 \#15971)}

\bibitem[Gar89]{GardnerEquiv}
Robert~B. Gardner, \emph{The method of equivalence and its applications},
  CBMS-NSF Regional Conference Series in Applied Mathematics, vol.~58, Society
  for Industrial and Applied Mathematics (SIAM), Philadelphia, PA, 1989.
  \MR{1062197 (91j:58007)}

\bibitem[IL03]{CfB}
Thomas~A. Ivey and J.~M. Landsberg, \emph{Cartan for beginners: Differential
  geometry via moving frames and exterior differential systems}, Graduate
  Studies in Mathematics, vol.~61, American Mathematical Society, Providence,
  RI, 2003. \MR{2003610 (2004g:53002)}

\bibitem[IZ13]{IsaevZaitsev}
Alexander Isaev and Dmitri Zaitsev, \emph{Reduction of five-dimensional
  uniformly {L}evi degenerate {CR} structures to absolute parallelisms}, J.
  Geom. Anal. \textbf{23} (2013), no.~3, 1571--1605. \MR{3078365}

\bibitem[MS14]{MedoriSpiro}
Costantino Medori and Andrea Spiro, \emph{The equivalence problem for
  five-dimensional {L}evi degenerate {CR} manifolds}, Int. Math. Res. Not. IMRN
  (2014), no.~20, 5602--5647. \MR{3271183}

\bibitem[{San}15]{Santi_models}
A.~{Santi}, \emph{{Homogeneous models for Levi-degenerate CR manifolds}}, ArXiv
  e-prints (2015).

\bibitem[Tan62]{TanakaCR}
Noboru Tanaka, \emph{On the pseudo-conformal geometry of hypersurfaces of the
  space of {$n$}\ complex variables}, J. Math. Soc. Japan \textbf{14} (1962),
  397--429. \MR{0145555 (26 \#3086)}

\bibitem[Web95]{WebsterCubic}
S.~M. Webster, \emph{The holomorphic contact geometry of a real hypersurface},
  Modern methods in complex analysis ({P}rinceton, {NJ}, 1992), Ann. of Math.
  Stud., vol. 137, Princeton Univ. Press, Princeton, NJ, 1995, pp.~327--342.
  \MR{1369146 (96k:32035)}

\bibitem[Wol69]{WolfRealOrbits}
Joseph~A. Wolf, \emph{The action of a real semisimple group on a complex flag
  manifold. {I}. {O}rbit structure and holomorphic arc components}, Bull. Amer.
  Math. Soc. \textbf{75} (1969), 1121--1237. \MR{0251246 (40 \#4477)}

\end{thebibliography}

\end{document}